\documentclass[pdflatex,sn-mathphys-num]{sn-jnl}

\usepackage{graphicx}%
\usepackage{multirow}%
\usepackage{amsmath,amssymb,amsfonts}%
\usepackage{amsthm}%
\usepackage{mathrsfs}%
\usepackage[title]{appendix}%
\usepackage{xcolor}%
\usepackage{textcomp}%
\usepackage{manyfoot}%
\usepackage{booktabs}%
\usepackage{algorithm}%
\usepackage{algorithmicx}%
\usepackage{algpseudocode}%
\usepackage{listings}%
\usepackage{subcaption}%

\theoremstyle{thmstyleone}%
\newtheorem{theorem}{Theorem}%
\newtheorem{lemma}[theorem]{Lemma}%
\newtheorem{proposition}[theorem]{Proposition}%

\theoremstyle{thmstyletwo}%
\newtheorem{remark}{Remark}%

\theoremstyle{thmstylethree}%
\newtheorem{definition}{Definition}%
\providecommand{\textcite}{\citet}

\usepackage{siunitx}
\usepackage{tikz,pgfplots}
\usepackage{adjustbox}
\usetikzlibrary{arrows.meta,calc,patterns,positioning,fit,backgrounds,decorations.pathmorphing}
\pgfplotsset{compat=1.18}
\definecolor{cividisDark}{RGB}{53,69,108}    
\definecolor{cividisMid}{RGB}{129,127,120}   
\definecolor{cividisLight}{RGB}{212,193,95}  

\usepackage{mathtools,amssymb,amsthm, bbold}
\makeatletter
\renewcommand*\env@matrix[1][*\c@MaxMatrixCols c]{%
  \hskip -\arraycolsep
  \let\@ifnextchar\new@ifnextchar
  \array{#1}}
\makeatother
\numberwithin{equation}{section}

\renewcommand{\tilde}[1]{\widetilde{#1}}

\newtheorem{assumption}{Assumption}

\usepackage{xparse}
\NewDocumentCommand{\snorm}{m O{0} O{\infty} O{d}}{%
  \lVert #1 \rVert_{W^{#2,#3}((0,1)^{#4})}%
}

\newcommand{\norm}[1]{\lVert #1 \rVert}
\newcommand{\euknorm}[1]{\lvert #1 \rvert }

\newcommand{\opnorm}[1]{\norm{#1}_\text{op}}
\newcommand{\E}{\mathbb{E}}
\newcommand{\N}{\mathbb{N}}
\newcommand{\R}{\mathbb{R}}
\newcommand{\G}{\mathbb{G}}

\newcommand{\V}{\mathbb{V}}
\newcommand{\A}{\mathbb{A}}
\newcommand{\I}{\mathbb{I}}
\newcommand{\Acal}{\mathcal{A}}
\newcommand{\Ocal}{\mathcal{O}}
\newcommand{\Tcal}{\mathcal{T}}
\newcommand{\Ecal}{\mathcal{E}}
\newcommand{\Pcal}{\mathcal{P}}
\newcommand{\Fcal}{\mathcal{F}}
\newcommand{\Hcal}{\mathcal{H}}
\newcommand{\Scal}{\mathcal{S}}
\newcommand{\Lcal}{\mathcal{L}}
\newcommand{\Gcal}{\mathcal{G}}
\newcommand{\Xcal}{\mathcal{X}}

\newcommand{\prob}{\mathbb{P}}

\newcommand{\uht}{u_h^{\mu, \nu, t}}
\newcommand{\uh}{u_h^{\mu, \nu}}
\newcommand{\uhthat}{\hat{u}_h^{\mu, \nu, t}}

\newcommand{\uhtprered}{u_\text{pre-red}^{\mu, \nu, t}}

\begin{document}

\title[A Deep Learning based Reduced Order Model for Hybrid-Type Parabolic PDEs]
{A New Adaptive Deep Learning based Reduced Order Model for Hybrid-Type Parabolic PDEs: Rigorous Error Analysis and Applications}

\author*[1]{\fnm{Dawid} \sur{Kotowski}}\email{dawid.kotowski@uni-muenster.de}
\author[1]{\fnm{Mario} \sur{Ohlberger}}\email{mario.ohlberger@uni-muenster.de}

\affil*[1]{\orgdiv{Institute for Analysis and Numerics}, \orgname{Münster University}, \orgaddress{\street{Einsteinstraße 62}, \city{Münster}, \postcode{48149}, \country{Germany}}}

\abstract{This contribution proposes novel data-driven surrogate modeling approaches for parameterized parabolic PDEs, 
where the parameter dependence can be split into two parts with different decay behavior of the Kolmogorov N-width. 
Such problems naturally arise in many industrial flow processes with dominant advection or traveling fronts in the solution
trajectories. 
To tackle this challenge, we extend the \emph{Deep Orthogonal Decomposition (DOD)} method, recently introduced for related stationary problems, to the time-dependent setting. 
We introduce and rigorously analyze two DOD based approaches:
Our approach is based on two novel adaptive deep learning-based surrogate models: The \emph{DOD-DL-ROM} method which is a Reduced Order Model (ROM) that leverages the adaptive nature of DOD, and the \emph{DOD+DFNN} method, which combines DOD with a generic Deep Feed-Forward Neural Network (DFNN). 
On the theory side, we generalize data-driven POD-based ROM arguments to the DOD setting, establishing a quantitative link between online performance and the regularity of an associated optimal map. 
Furthermore, we identify specific problem size and error tolerance requirements for DOD-based ROMs to outperform POD-based ROMs in hybrid-type problem classes, which is crucial for efficient computation. 
The significance of this work lies in its potential to accelerate the solution of complex PDEs, enabling faster design and optimization of industrial processes. 
The proposed approaches are demonstrated on a catalyst filter benchmark problem, showcasing their effectiveness and comparing favorably to traditional POD-based methods.
}


\maketitle

\section{Introduction}

In the pursuit of accelerated product development and digital transformation, industries reliant on 
complex physical simulations face a pressing challenge: balancing high-fidelity accuracy with 
extreme computational efficiency. Traditional High-Fidelity (HF) simulations, such as Finite Element Methods (FEM) 
and Computational Fluid Dynamics (CFD), are indispensable for predicting intricate physical phenomena, 
but their prohibitive computational costs hinder applications requiring rapid responses or repeated evaluations, 
like real-time control, Digital Twins, and multi-objective design optimization. 
To bridge this gap, Model Order Reduction (MOR) techniques and Deep Learning (DL) based surrogate modelling 
have been employed, but they often struggle for problems with evolving fronts and convection dominated regimes, 
which are characterized by a slow decay of the so called Kolmogorov N-width \cite{OhlRav16}. 
Meanwhile, purely data-driven Machine Learning (ML) approaches, although flexible, may compromise 
physical consistency, leading to unreliable "black-box" predictions. 

Over the past decades the field of Reduced Order Modeling \cite{MR3672144,MR3701994} has undergone a noticeable progress,
shifting from classical projection-based methods to modern data-driven architectures.
The earliest intrusive ROMs originated from the \emph{Proper Orthogonal Decomposition} (POD)
framework, introduced in turbulence analysis by \cite{lumley1967} and \cite{sirovich1987},
and later consolidated in the POD-Galerkin formulation for parametrized PDEs (PPDEs)
\cite{holmes1996,kunischvolkwein2001} and finally the POD-Greedy method \cite{HO2008}.
Subsequent developments sought to address the computational bottlenecks of nonlinear problems
through hyper-reduction techniques such as the \emph{Empirical Interpolation Method} (EIM)~\cite{barrault2004}
and its variants~\cite{chaturantabut2010,DHO12}.
Alternative intrusive formulations include the least-squares Petrov-Galerkin projection~\cite{carlberg2011}
and the GNAT method~\cite{carlberg2013}, which extended ROMs to highly nonlinear and turbulent regimes.

Simultaneously, a non-intrusive and data-driven perspective emerged, aiming to approximate
the parameter-to-solution map directly from snapshot data.
Early methods such as \emph{Dynamic Mode Decomposition} (DMD)~\cite{schmid2010}
and its interpretation via the Koopman operator~\cite{rowley2009}
established a foundation for learning reduced dynamics without explicit projection.
Further, regression-based frameworks such as \emph{Operator Inference}~\cite{peherstorfer2018}
and \emph{POD with interpolation} (PODI)~\cite{barrault2004}
demonstrated the feasibility of reconstructing reduced operators from data.

The integration of deep learning with model order reduction around 2020
led to the first \emph{Deep Learning Reduced Order Models} (DL-ROMs)~\cite{fresca2020},
which combined convolutional autoencoders and neural regressors to emulate low-dimensional dynamics.
Their refinement, the \emph{POD-DL-ROM}~\cite{fresca2022,brivio2023},
fused the interpretability of POD with the flexibility of deep networks.
An alternative approach that combines classical Reduced Basis (RB) methods with ML has recently been 
introduced with error certification as the adaptive and hierarchical \emph{{RB}-{ML}-{ROM}} 
\cite{haasdonk2022new}.  In \cite{klein2025multifidelity}, the approach has been used
for multi-fidelity learning of reduced oder models for parabolic PDE constrained optimization.

For machine learning approaches in the context of reactive flow in porous media with application to the prediction of 
catalysts, we particularly refer to \cite{Gavrilenko2022378}, where a full order, reduced order and machine learning 
model pipeline has been introduced and to \cite{Fokina202567} for ML methods to predict breakthrough curves. 
ML algorithms for parameter identification in reactive flow in porous media have been studied in \cite{Fokina202491}. 
Furthermore, we refer to \cite{Biermann2025}  for physics-enhanced neural networks to enabling micro-kinetics based simulation of industrial packed-bed reactors. 

In this article, we propose novel approaches that combine the strengths of deep learning-based surrogates with 
projection-based MOR, specifically tailored for parameterized time-dependent PDEs with distinct decay behaviors 
of the Kolmogorov N-width, when particular subsets of a parametrization are considered. 
By integrating the Deep Orthogonal Decomposition (DOD) \cite{franco2024} method with Reduced Order Models (ROM) or generic Deep Feed-Forward Neural Networks (DFNN), 
our methodologies, dubbed DOD-DL-ROM and DOD+DFNN, aim to provide robust, real-time, and physically consistent 
solution approaches. Most importantly, we show complexity bounds for the proposed methods, as well as a general lower precision bounds for an optimal choice of neural network weights. These results then motivate a comparison in
terms of maximal complexity needed to guarantee a fixed relative error with high probability.
In fact, we will show a complexity result under some assumptions on the underlying problem class 
that proves that the number of weights to reach an error bound in the DOD-DL-ROM architecture is 
asymptotically smaller than for the POD-DL-ROM architecture (see~Proposition \ref{prop: regularity discussion for POD vs DOD} for details). 
The analysis is based on recent advances in approximation theory of neural networks \cite{yarotsky2017nn,gühring2019nn} that has been used to develop complexity bounds on data-driven ROMs.
Some notable references for such recent cases of usage are \cite{franco2023approx,bhattacharya2021,marwah2021}.
Finally, we will present numerical experiments that underline our theoretical results
for an industrially relevant benchmark problem.

The rest of this article is structured as follows. In Section \ref{sec:problem} we introduce the general problem formulation.
Section \ref{section: data-driven ROMs} then provides some notation and preliminaries for neural networks and statistical learning
and introduces existing as well as our novel deep learning ROMs. 
A thorough error and complexity analysis based on approximation theory is presented in Section \ref{sec:error_analysis}. 
Finally, numerical experiments in Section \ref{sec:experiments} compare the studied DL-ROMs for a benchmark problem motivated 
from catalytic filters.

\section{Mathematical Formulation of the Problem}\label{sec:problem}

Let the parameter spaces $\Theta \subset \R^p$ and $\Theta' \subset \R^q$
be compact and $T>0$. 
We consider a discretized parametric partial differential equation (PPDE) with $N_h$ degrees of freedom in space. 
The resulting FOM solution can be represented as
\begin{equation} \label{eq: fom solution identity}
    t \mapsto u_h(\cdot, t; \mu, \nu) = \sum_{j = 1}^{N_h} [\uh(t)]_j \, \varphi_j \in C^1([0, T]; \Hcal_h) 
\end{equation}
with $\{\varphi_j\}_{j=1}^{N_h}$ a basis of a Hilbert space $\Hcal_h \cong \R^{N_h}$  and $\uh(t) \in \R^{N_h}$ 
for each $t \in [0, T]$. The resulting discrete equations system for the coefficients $\uh(t)$ can be written as
\begin{equation} \label{eq: fom ppde}
    \begin{cases}
    \displaystyle
    \G \frac{\partial}{\partial{t}} \uh(t) + A(\mu, \nu) \uh(t) + 
    N(\uh(t), \mu, \nu)= f(t; \mu, \nu)  \quad &t \in [0, T], \\[0.8em]
    \displaystyle
    \uh(0) = u_0(\mu, \nu) \quad &t=0, \end{cases} 
\end{equation}
where we assume 
    $\G \in \R^{N_h \times N_h}$, $\G_{ij} := \langle \varphi_i, \varphi_j \rangle_{\Hcal_h}$ to be the mass matrix,
    $A: \Theta \times \Theta' \to \R^{N_h \times N_h}$ a mapping to a symmetric positive definite stiffness matrix, 
    $N: \R^{N_h} \times \Theta \times \Theta' \to \R^{N_h}$ a non-linear term, 
    $f: [0, T] \times \R^{N_h} \times \Theta \times \Theta' \to \R^{N_h}$ the source term of the equation, and
    $u_0: \Theta \times \Theta' \to \R^{N_h}$ denotes the parameterized initial datum. 

This allows us to equip $\Scal := \{ \uht \in \R^{N_h} \, \vert \, (\mu, \nu, t) \in \Theta \times \Theta' \times [0, T] \} \subset \R^{N_h}$, the space of the coefficients 
of any solution in $\Hcal_h$, with the following
norm
\begin{equation} \label{eq: definition mass matrix norm}
    \norm{u} := \sqrt{u^T \G u}, \quad u \in \Scal.
\end{equation}
With this definition, we have $\norm{\uht} = \norm{u_h(\cdot, t; \mu, \nu)}_{\Hcal_h}$.

The general problem of finding solutions of (discretized) PPDEs
is equivalent to finding a map $\Gcal_h: \Theta \to \Hcal_h$. Some approaches
like Operator Inference \cite{peherstorfer2018} aim directly at approximating $\Gcal_h$.
However, most approaches try to attain an approximation from the FOM solutions, which are
considered the image of the coefficient identity maps of the solutions of $\Gcal_h$, namely
\begin{equation*}
    \Gcal: \Theta \times [0, T] \to \R^{N_h}; \quad (\mu, t) \mapsto u_h^{\mu, t}.
\end{equation*}
We particularly refer to \cite{brivio2023}, where approximation theory was applied to 
a Proper Orthogonal Decomposition (POD) based data-driven model, to achieve bounds on the complexity
of the network in use. However,
as was shown in \cite{brivio2023}, the POD-based method in question has an inevitable lower precision bound
caused by the Kolmogorov $N$-width (KnW) of the underlying PPDE, which is defined as follows

\begin{definition}[Kolmogorov $N$-width] \label{def: linear kolmogorov}
    Let $N \in \N$, $N \leq N_h$ and $\Scal = \{ u_h^{\mu, t} \in \R^{N_h} \, 
    \vert \, (\mu, t) \in \Theta \times [0, T] \}$ 
    be some parametric manifold. The linear Kolmogorov $N$-width of $\Scal$ is defined as
    \begin{equation*}
        d_N(\Scal) = \inf\limits_{\V \in \R^{N_h \times N}} \sup\limits_{u \in S_{N_h}} 
        \norm{u - \V \V^T \G u}.
    \end{equation*}
\end{definition}

Such lower bound can be of problematic nature for transport-dominated problems. In order to tackle
such problems, without losing the advantage of POD-optimality for linear subspaces, we now introduce
Hybrid-Type PPDEs. For this, we specify the general setting of a single parameter PPDE into
a setting with two different parameter sets $\Theta \subset \R^p$ and $\Theta' \subset \R^q$ for
$p,q \in \N$ respectively.
The results of this paper mostly do not necessitate the following assumption, though
its assertion justifies the usage of any DOD-based ROM instead of simply using the POD for reduction
purposes.

\begin{assumption} \label{assumption KnW decay}
    Let $\mathcal{S}$ 
    denote the \emph{total solution manifold} and $\mathcal{S}_{\mu, t} := 
    \{ \uht \in \R^{N_h} \, \vert \, \nu \in \Theta' \}$ denote the \emph{solution 
    submanifold} for a fixed tuple $(\mu, t) \in \Theta \times [0, T]$, then we assume:
    \begin{enumerate}
\item The linear Kolmogorov $N$-width of $\Scal$ decays slowly, i.e., $d_N(\Scal) \leq C N^{-\alpha}$ for $0 < \alpha \leq 1$ and $C > 0$.
\item The linear Kolmogorov $N$-width of $\Scal_{\mu, t}$ decays quickly for all $\mu \in \Theta$ and $t \in [0, T]$, i.e., $\sup_{(\mu, t) \in \Theta \times [0, T]} d_N(\Scal_{\mu, t}) \leq C' N^{-\beta}$ for $\beta > 1$ and $C' > 0$.
\end{enumerate}
\end{assumption}

From now on, we consistently rely on the following assumptions.
\begin{assumption} \label{assumption sample}    
    We assume that all $N_s \in \N_{\geq 2}$ parameter data points for the
    solution trajectory data set are sampled 
    independent and identically distributed in the joint parameter space $\Theta \times \Theta'$.
    For the time variable a uniform grid 
    $\Tcal := \{\Delta t, 2 \Delta t, \dots, N_t \Delta t \}$, $\Delta t = T / N_t, N_t \in \mathbb{N}_{\geq 2}$ is used.
    $N_{s_1}$ and $N_{s_2}$ 
    are the respective number of random samples in $\Theta$ and $\Theta'$,
    $N_s := N_{s_1} N_{s_2} \in \mathbb{N}$ is the total sample size.
    The overall data size is $N_\text{data} := N_s N_t$ .
\end{assumption}

\begin{assumption} \label{assumption parameter-to-solution map}
    Let $\mathcal{G}: \Theta \times \Theta' \times [0, T] \to \R^{N_h}$ be the parameter-to-solution map, 
    mapping $(\mu, \nu, t) \mapsto \uht$. Then, we assume that
    \begin{enumerate}
        \item $m := 
        \underset{(\mu, \nu, t) \in \Theta \times \Theta' \times [0, T]}{\operatorname{ess}\, \inf}
        \norm{\uht} > 0, \,$ and $ M := 
        \underset{(\mu, \nu, t) \in \Theta \times \Theta' \times [0, T]}{\operatorname{ess}\, \sup}
        \norm{\uht} < \infty$,
        \item $\mathcal{G}$ is Lipschitz-continuous with the constant $L > 0$.
    \end{enumerate}
\end{assumption}

There is a subliminal assumption in the definition of such a parameter-to-solution map; the solution 
$u_h(\cdot, \cdot \, ; \mu, \nu)$ exists and is unique. This is meant, at least, in an $L^2$-embedded sense, 
since the time-derivative can be relaxed using a variational argument. 

\begin{assumption} \label{assumption perfect embedding}
    Let $N \in \N$ be fixed.
    We assume that for $s, s'$ sufficiently large, there are two maps 
    $\psi_*: \R^n \to \R^N$ and $\psi_*': \R^N \to \R^n$ for an $n \in \N$ with $n \leq 2(p + q) + 3$, 
    that still attain a nonlinear embedding without loss of the reduced solution manifold, i.e.
    for $\mathcal{S}_N := \{ q(\mu, \nu, t) = \A^T u(\mu, \nu, t) \, \vert \, (\mu, \nu, t) 
    \in \Theta \times \Theta' \times \Tcal \}$, where $\A \in \R^{N_h \times N}$ is the POD matrix,
    \begin{equation*}
        \sup\limits_{u \in S_{N}} 
        \euknorm{u - \psi_*(\psi_*'(u))} = 0.
    \end{equation*}
\end{assumption}

Assumption (\ref{assumption sample}) is required
for the derivation of any convergence result to a true parameter-to-solution map, 
when sampling an ever-growing number of examples. 
Assumption (\ref{assumption parameter-to-solution map})
imposes regularity on the mentioned parameter-to-solution map. It ensures
the existence of a true solution of the discrete problem for any specified parameter choice. 
Assumption (\ref{assumption perfect embedding}) gives us the ability to later meaningfully 
use an Autoencoder to achieve a reduction. Its application is justified with results from \cite{franco2023}.

\section{Deep Orthogonal Decomposition and resulting Reduced Order Models} \label{section: data-driven ROMs}

The data driven ROMs that we will introduce and study in the following make use of 
a hierarchy of reduced dimensions. We make the following assumption. 

\begin{assumption}[Dimension Hierarchy] \label{assumption: dimension hierarchy}
    Let the following natural number hierarchy be fixed
    \begin{equation*}
        n < N' < N < N_A \leq N_h.
    \end{equation*}
\end{assumption}

Before introducing our novel approaches based on the DOD, let us fix some notation and concepts for neural networks.

\subsection{Neural Networks and Statistical Learning}

We generally use the conventions and definitions of neural networks and their composites 
according to \textcite{gühring2019nn}. 
A neural network $\Phi$ with
input dimension $d$ and $L$ layers is given by
$\Phi = ((A_1, b_1), \dots, (A_L, b_L))$, such that $N_0 = d$ and $N_1, \dots, N_L \in \N$, where each $A_l$ is
an $N_l \times \sum_{k=0}^{l-1} N_k$ dimensional matrix and $b_l \in \R^{N_l}$. We further define $R_\rho(\Phi): \R^d \to \R^{N_L}$ as a \emph{realization of $\Phi$ with activation function $\rho : \R \to \R$} as the result of the following scheme for some input $x_0 := x \in R^d$:
\begin{align*}
    x_l := \rho \left( A_l \begin{bmatrix}[c|c|c]x_0^T & \cdots & x_{l-1}^T\end{bmatrix}^T + b_l \right),  l = 1, \dots, L-1, \quad
    x_L := A_L \begin{bmatrix}[c|c|c]x_0^T & \cdots & x_{l-1}^T\end{bmatrix}^T + b_L,
\end{align*}
where $\rho$ acts entry-wise. Let $\Phi^2$ be a neural network with the 
output dimension of the input dimension of the neural network $\Phi^1$, then we denote with $\Phi^1 \odot \Phi^2$ the
formal concatenation of both networks according to \cite[Remark 2.8]{gühring2019nn}. Moreover, we denote a neural network
architecture $\Acal_\Phi$ as a neural network with $\{0, 1\}$-valued entries. Two networks are said to have the same
architecture, if their non-zero entries are the same.
If $\Phi^1$ and $\Phi^2$ have the
same input dimension, then $P(\Phi^1, \Phi^2)$ denotes the formal parallelization of the networks 
according to \cite[Lemma 2.9]{gühring2019nn}.
We use the ReLU activation function $\rho(x) = x \mathbb{1}_{x \geq 0}$ for all theoretical results, as
justified by \textcite{hornik1991nn}.\\

{\bf Supervised Learning.} 
We consider a feature space $\Xcal$, a label space $\mathcal{Y}$, and an underlying true classification function $f$. 
Given a sample set $\{x_i, f(x_i)\}_{i=1}^N \subset \Xcal \times \mathcal{Y}$, we define a loss function $\Lcal$ as a map that sends all maps $\hat{f} : \Xcal \to \mathcal{Y}$ to $[0, \infty)$. The loss function must satisfy $\Lcal(f \, \vert \, \{x_i, f(x_i)\}_{i=1}^N) = 0$.
To separate statements about complexity from realistic attainability, we introduce a notion for neural network weights. This is motivated by the specifications for training, as it can highlight the difference in training one module while using another "frozen" embedded neural network.

\begin{definition}[Neural Network Weights] \label{def: neural network weights}
   Let $\Phi$ be a neural network with input dimension $d$, $L$ layers, and output dimension $N_L$. We define the \emph{masks of $\Phi$} as:
    $I_l := \{ (i, j) \, : \, [A_l]_{ij} \neq 0 \},  J_l := \{i \, : \, [b_l]_i \neq 0 \},$ 
    for each layer $1 \leq l \leq L$. We further set the \emph{admissible weight space} as
    \begin{equation*}
        \Theta(\Phi) := \prod_{l=1}^L \left( \R^{I_l} \times \R^{J_l} \right),
    \end{equation*}
    and its elements $\theta \in \Theta(\Phi)$ with
    \begin{equation*}
        \theta = \left( (\theta_{l, ij}^A)_{(i, j) \in I_l}, (\theta_{l, i}^b)_{i \in J_l} \right)_{l=1}^L.
    \end{equation*}
    Then $\Phi(\theta) = 
    ((\hat A_1(\theta), \hat b_1(\theta)), \dots, (\hat A_L(\theta), \hat b_L(\theta)))$
    can be uniquely defined for a choice of $\theta \in \Theta(\Phi)$ via the lifting
    \begin{equation*}
        \hat A_l(\theta)_{ij} = 
        \begin{cases}
            \theta^A_{l,ij}, & (i,j)\in I_l,\\
            0, & \text{otherwise},
        \end{cases}
        \qquad
        \hat b_l(\theta)_{i} =
        \begin{cases}
            \theta^b_{l,i}, & i\in J_l,\\
            0, & \text{otherwise}.
        \end{cases}
    \end{equation*}
    This allows us to classify realizations of a neural network uniquely by
    a map, when knowing
    its architecture and a choice of neural network weights 
    $\theta \in \Theta(\Phi)$, via
    $ 
        R_\rho(\Phi(\theta)) : \R^d \to \R^{N_L}.
    $ 
    In particular, if $\Phi$ has architecture $\Acal_\Phi$, then there is $e \in \Theta(\Phi)$,
    such that $\Phi(e) = \Acal_\Phi$.

\end{definition}


\subsection{Deep Learning ROMs and Deep Orthogonal Decomposition}

As a benchmark for all models used in this paper, we introduce the \emph{POD-DL-ROM} architecture, 
which is based on the \emph{DL-ROM} introduced in \cite{fresca2020}. Our primary reference for the POD-DL-ROM is \textcite{fresca2022}. However, we deviate slightly from the original definition to achieve an estimate close to the high-fidelity solution in an $\Hcal_h$-sense by applying the mass matrix $\G$.

The POD-DL-ROM architecture offers several benefits, including scalability in $N_h$ a
nd reduced computational cost for offline training. 
This is achieved by employing a POD (see Definition (\ref{def: POD})) on the stacked solution trajectories
\begin{equation*}
    U := \left[u_h^{\mu_1, \nu_1} \big\vert \cdots \big\vert u_h^{\mu_{N_{s_1}}, \nu_1} \big\vert \cdots \big\vert 
    u_h^{\mu_{N_{s_1}}, \nu_{N_{s_2}}}\right]^T \in \R^{N_h \times N_\text{data}},
\end{equation*}
of the sample solution set, resulting in the POD matrix 
$\A_\text{P} \in \R^{N_h \times N}$ for $N < N_h$ from Assumption (\ref{assumption: dimension hierarchy}). 
Using the POD matrix, we define the reduced solution snapshots as
\begin{equation*}
    u_\text{red}^{\mu, \nu, t} := \A_\text{P} ^T \G \uht \in \R^{N} \qquad \text{for all } (\mu, \nu, t) \in \Pcal_1 \times \Pcal_2 \times \Tcal.
\end{equation*}
By $\G$-orthonormality of $\A_\text{P}$, we define the approximate coefficient solution $\uhthat \in \R^{N_h}$ for each $\mu \in \Theta$, $\nu \in \Theta'$, and $t \in \Tcal$ as
\begin{equation*}
\uhthat = \A_\text{P}  \cdot \left(R_\rho(\Psi_N) \circ R_\rho(\Phi_n)\right) (\mu, \nu, t) \in \R^{N_h},
\end{equation*}
where $\Psi_N = f^\text{D}_{N}$ is the decoder architecture with input dimension $n$ and output dimension $N$
($\Psi'_n = f^\text{E}_n$ the corresponding encoder architecture) of a convolutional Auto Encoder (AE). 
Then $f^\text{E}_n$ denotes the projection onto the reduced trial solution manifold, and we 
let $\Phi_n = \Phi_n^\text{DF}$  
be the DFNN architecture of input dimension $p + q + 1$ and output dimension $n$.

The use of the low-dimensional Euclidean norm is justified by Proposition \ref{prop: switch to euclidean norm}, which allows us to ignore the averaging over all deterministic POD projection errors for training without misrepresenting the solution as a bare vector in Euclidean space.

Hence, with the neural network weights $\theta = 
(\theta_\text{D}, \theta_\text{E}, \theta_\text{DF}) 
\in \Theta(\Psi_N) \times \Theta(\Psi_n') \times \Theta(\Phi_n)$ we define the loss function of the POD-DL-ROM by
\begin{multline*}
    \Lcal_\text{POD-DL-ROM}\left(R_\rho(P(\Psi_N \odot \Phi_n, \Psi'_n)(\theta)) \big\vert \{\mu_i, \nu_j, k \Delta t\}_{i,j,k}, 
    \{u_\text{red}^{\mu_i, \nu_j, k \Delta t} \}_{i,j,k}\right) =  \\
        \frac{1}{N_\text{data}} \sum_{i=1}^{N_{s_1}} \sum_{j=1}^{N_{s_2}} 
        \sum_{k = 1}^{N_t} \, \bigg( \frac{\omega_h}{2} \, 
        \euknorm{u_\text{red}^{\mu_i, \nu_j, k \Delta t} - 
        \left(R_\rho(\Psi_N(\theta_\text{D})) \circ R_\rho(\Phi_n(\theta_\text{DF}))\right) (\mu_i, \nu_j, k \Delta t)}^2 \\
        + \, \frac{1 - \omega_h}{2} \, \euknorm{R_\rho(\Psi'_n(\theta_\text{E}))(u_\text{red}^{\mu_i, \nu_j, k \Delta t}) 
        - R_\rho(\Phi_n(\theta_\text{DF}))(\mu_i, \nu_j, k \Delta t)}^2 \bigg).
\end{multline*}

To overcome the limitations of using a linear projector on the full solution manifold, we can employ a $(\mu, t)$-sensitive neural network to find an adaptive basis for projection purposes. This approach can lead to a further reduction in dimensionality $N' < N$ compared to the POD, while maintaining a similar energy loss, due to the hybrid structure of the solution submanifold.
We introduce the concept of Deep Orthogonal Decomposition (DOD), which combines the strengths of POD and neural networks to achieve a more efficient reduction of the solution manifold.
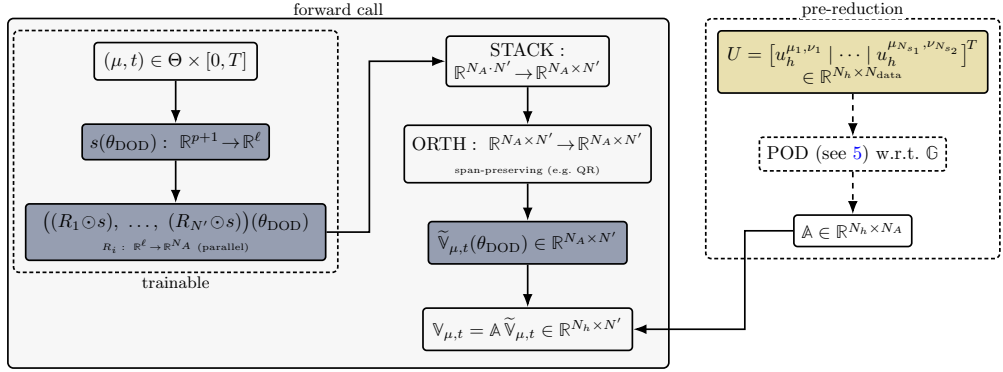
\begin{figure}[t]
    \centering
    \begin{adjustbox}{max width=\linewidth}
    \begin{tikzpicture}[
      >=Latex,
      node distance=9mm and 12mm,
      every path/.style={line width=.9pt},
      box/.style={draw, rounded corners=2pt, align=center, inner sep=5pt},
      head/.style={draw, rounded corners=2pt, align=center, inner sep=6pt, fill=black!5},
      small/.style={font=\footnotesize},
      tiny/.style={font=\scriptsize},
      lbl/.style={font=\scriptsize, inner sep=1pt},
      op/.style={draw, align=center, inner sep=4pt, rounded corners=2pt},
      dashedop/.style={op, dash pattern=on 2.1pt off 1.4pt},
      transform shape
    ]
  
    \node[box, anchor=west, fill=black!3] (INPT) at (0,0) {$(\mu,t)\in\Theta\times[0,T]$};
  
    \node[op, below=8mm of INPT, fill=cividisDark, fill opacity=0.5, text opacity=1] (SNET) {$s(\theta_\text{DOD}):\ \R^{p+1}\!\to\!\R^{\ell}$};
    \draw[->] (INPT) -- (SNET);
  
    \node[op, below=8mm of SNET, minimum width=56mm, fill=cividisDark, fill opacity=0.5, text opacity=1] (PAR) {$\big((R_1\!\odot\! s),\ \dots,\ (R_{N'}\!\odot\! s)\big)(\theta_\text{DOD})$\\
      \tiny $R_i:\ \R^{\ell}\!\to\!\R^{N_A}$ (parallel)};
    \draw[->] (SNET) -- (PAR);
  
    \node[op, anchor=west] (STACK) at ([xshift=35mm]INPT.east) {$\text{STACK}:$\\$\R^{N_A\cdot N'}\!\to\!\R^{N_A\times N'}$};
    \node[op, below=7mm of STACK] (ORTH) {$\text{ORTH}:\ \R^{N_A\times N'}\!\to\!\R^{N_A\times N'}$\\\tiny span-preserving (e.g.\ QR)};
    \node[box, below=7mm of ORTH, fill=cividisDark, fill opacity=0.5, text opacity=1] (Vt) {$\tilde{\V}_{\mu,t}(\theta_{\text{DOD}})\in\R^{N_A\times N'}$};
    \node[box, below=8mm of Vt, fill=black!3] (V) {$\V_{\mu,t}=\A\,\tilde{\V}_{\mu,t}\in\R^{N_h\times N'}$};

    \node[box, anchor=north west, fill=cividisLight, fill opacity=0.5, text opacity=1] (U) at ([xshift=20mm,yshift=1mm]STACK.north east)
        {$U=\big[u_h^{\mu_1,\nu_1}\mid\cdots\mid u_h^{\mu_{N_{s_1}},\nu_{N_{s_2}}}\big]^{\!T}$\\
         $\in\R^{N_h\times N_{\text{data}}}$};
    \node[dashedop, below=8mm of U] (POD) {POD (see \ref{def: POD}) w.r.t.\ $\G$};
    \node[op, below=8mm of POD] (A) {$\A\in\R^{N_h\times N_A}$};
  
    \draw[->, dashed] (U) -- (POD);
    \draw[->, dashed] (POD) -- (A);
  
    \draw[->] (PAR.east) -- ++(8mm,0) |- (STACK.west);
    \draw[->] (STACK) -- (ORTH);
    \draw[->] (ORTH) -- (Vt);
    \draw[->] (Vt) -- (V);
    \draw[->] (A.west) -- ++(-9mm,0) |- (V.east);
  
    \begin{pgfonlayer}{background}
      \node[draw, rounded corners=3pt, dash pattern=on 2.1pt off 1.4pt, inner sep=6pt,
            fit=(U)(POD)(A)] (OFFBOX) {};
    \end{pgfonlayer}
    \node[tiny, fill=white, inner sep=1pt] at ([yshift=1.6mm]OFFBOX.north) {pre-reduction};
  
    \begin{pgfonlayer}{background}
      \node[draw, rounded corners=3pt, inner sep=9pt, fill=black!3,
            fit=(INPT)(SNET)(PAR)(STACK)(ORTH)(Vt)(V)] (ONBOX) {};
    \end{pgfonlayer}
    \node[tiny, fill=white, inner sep=1pt] at ([yshift=1.6mm]ONBOX.north) {forward call};

    \begin{pgfonlayer}{background}
        \node[draw, rounded corners=3pt, dash pattern=on 2.1pt off 1.4pt, inner sep=6pt,
              fit=(INPT)(SNET)(PAR)] (TRAIN) {};
    \end{pgfonlayer}
    \node[tiny, fill=black!3, inner sep=1pt] at ([yshift=-1.6mm]TRAIN.south) {trainable};
  
    \end{tikzpicture}
    \end{adjustbox}
    \caption{Realization of the Deep Orthogonal Decomposition for given $\theta_\text{DOD} \in \Theta(\Phi_{\tilde{\V}})$.}
    \label{fig:dod-matrix-workflow}
  \end{figure}
\begin{definition}[Deep Orthogonal Decomposition] 
    \label{def: DOD}
    Recall  the dimensional hierarchy of Assumption (\ref{assumption: dimension hierarchy}).
    We employ a POD (according to Definition \ref{def: POD}) with $\G$
    given as above, using
    the stacked solution trajectories $U \in \R^{N_h \times N_\text{data}}$
    This results in the POD matrix $\A \in \R^{N_h \times N_A}$.
    Next, we set up a Neural Network architecture 
    $\Phi_{\tilde{\V}}$ with input dimension $p + 1$ and output dimension $N_A \times N'$,
    such that
    \begin{equation*}
        \Phi_{\tilde{\V}} = \text{ORTH} \circ \text{STACK} \circ (P((R_1 \odot s), \dots, (R_N \odot s))),
    \end{equation*}
    where $P$ is the network parallelization operator.
    We have the following components:
    \begin{enumerate}
        \item $s$ is a Deep Feed Forward Neural Network architecture with input dimension $p + 1$ and output dimension $l$,
        \item $R_1, \dots, R_N$ are a collection of Deep Feed Forward Neural Network architectures
        with input dimensions $l$ and output dimensions $N_A$,
        \item $\text{STACK}: \R^{N_A \cdot N'} \to \R^{N_A \times N'}$ is a canonical stacking operator, 
        \item $\text{ORTH}: \R^{N_A \times N'} \to \R^{N_A \times N'}$ is an orthonormalization unit, that orthonormalizes
        any given matrix $W \in \R^{N_A \times N'}$ into $\tilde{W} \in \R^{N_A \times N'}$ such that 
        $\text{span}(W) = \text{span}(\tilde W)$. This is most often done using a QR decomposition \cite{golubvanloan2013}. 
    \end{enumerate}
    The realization of the network $\Phi_{\tilde{\V}}$ for some $\mu \in \Theta$, $t \in [0, T]$ and
    any chosen neural network weight $\theta_\text{DOD} \in \Theta(\Phi_{\tilde{\V}})$ 
    according to Definition \ref{def: neural network weights} is abbreviated as
    \begin{equation} \label{eq: inner module DOD}
        \tilde{\V}_{\mu, t}(\theta_\text{DOD}) := R_\rho(\Phi_{\tilde{\V}}(\theta_\text{DOD}))(\mu, t).
    \end{equation}
    For any choice of weights, this is called the inner DOD module.

    \emph{The Deep Orthogonal Decomposition} can be achieved by upward projection under the POD matrix $\A$ of 
    the inner DOD module. For a choice of weights
    $\theta_\text{DOD} \in \Theta(\Phi_{\tilde{\V}})$ and parameters
    $\mu \in \Theta$ and $t \in [0, T]$, the DOD linear projector is given by
    \begin{equation} \label{eq: DOD}
        \V_{\mu, t}(\theta_\text{DOD}) := \A \cdot \tilde{\V}_{\mu, t}(\theta_\text{DOD}) \in \R^{N_h \times N'}.
    \end{equation}
    This process is visualized in Figure \ref{fig:dod-matrix-workflow}.
\end{definition}

While this definition is a crucial step, it does not guarantee a "good" DOD. To ensure the quality of the DOD, we introduce a qualitative loss function and collect relevant data through pre-reduction, similar to our previous approach. We define
\begin{equation*}
    u_\text{pre-red}^{\mu, \nu, t} := \A^T \G \uht \in \R^{N_A} \qquad \text{for all } (\mu, \nu, t) \in \Pcal_1 \times \Pcal_2 \times \Tcal.
\end{equation*}
We then formulate the optimization problem by minimizing over $\theta_\text{DOD} 
\in \Theta(\Phi_{\tilde{\V}})$ the loss function 
\begin{multline*}
    \Lcal_\text{DOD}\left(R_\rho(\Phi_{\tilde{\V}}(\theta_\text{DOD})) \big\vert \{\mu_i, \nu_j, k \Delta t\}_{i,j,k}, 
    \{u_\text{pre-red}^{\mu_i, \nu_j, k \Delta t} \}_{i,j,k}\right) \\
    = \frac{1}{N_\text{data}}\sum_{i=1}^{N_{s_1}} \sum_{j=1}^{N_{s_2}} 
    \sum_{k = 1}^{N_t} \, \euknorm{u_\text{pre-red}^{\mu_i, \nu_j, k \Delta t} - 
    \tilde{\V}_{\mu_i, k \Delta t}(\theta_\text{DOD}) \tilde{\V}^T_{\mu_i, k \Delta t}(\theta_\text{DOD}) 
    u_\text{pre-red}^{\mu_i, \nu_j, k \Delta t}}^2.
\end{multline*}

Note that this loss function is $N_A$-dimensional, thanks to the $\G$-orthonormality of the POD matrix and Proposition \ref{prop: switch to euclidean norm}. We denote the optimized neural network weights as $\theta_\text{DOD}^* \in \Theta(\Phi_{\tilde{\V}})$.
At this point, we have not yet searched for the approximate parameter-to-solution map $\Gcal$. After all, the frozen DOD is a linear orthogonal projector, not an estimation of $\Gcal$, and can be seen as an analog to the POD. 
We now propose two novel data-driven ROMs, based on the DOD.

\subsection{DOD+DFNN} \label{subsec: DOD+DFNN}

Let $\Phi_{N'} := \Phi_{N'}^\text{DF}$ 
be some Deep Feed-Forward Neural Network architecture
with input dimension $p+q+1$ and output dimension $N'$. Then the approximate parameter-to-solution map 
of the \emph{DOD+DFNN}
can be given for each $(\mu, \nu, t) \in \Theta \times \Theta'
\times [0, T]$ via
\begin{equation*}
    \uhthat = \A \cdot R_\rho(\Phi_{\tilde{\V}})(\mu, t) \cdot R_\rho(\Phi_{N'})(\mu, \nu, t) \in \R^{N_h}.
\end{equation*}

The training is done in series and not in parallel.
We first find the optimal DOD parameter $\theta^*_\text{DOD} \in \Theta(\Phi_{\tilde{\V}})$ by 
solving the optimization problem
in Definition \ref{def: DOD}. This freezes the DOD, and we set $\V_{\mu, t} := \V_{\mu, t}(\theta^*_\text{DOD})$ for 
$(\mu, t) \in \Theta \times [0, T]$.

Next, we use the frozen DOD matrix to reduce the training data
\begin{equation*}
    u_{\text{red}^2}^{\mu, \nu, t} := \V_{\mu, t}^T \G \uht \in \R^{N'} 
    \qquad \text{for all } (\mu, \nu, t) \in \Pcal_1 \times \Pcal_2 \times \Tcal.
\end{equation*}
The optimization problem for $\theta_\text{DF} \in \Theta(\Phi_{N'})$
is given by minimizing the loss function
\begin{multline*}
    \Lcal_\text{DOD+DFNN}\left(R_\rho(\Phi_{N'}(\theta_\text{DF})) \big\vert \{\mu_i, \nu_j, k \Delta t\}_{i,j,k}, 
    \{u_{\text{red}^2}^{\mu_i, \nu_j, k \Delta t} \}_{i,j,k}\right) \\
    = \frac{1}{N_\text{data}}\sum_{i=1}^{N_{s_1}} \sum_{j=1}^{N_{s_2}} 
    \sum_{k = 1}^{N_t} \, \euknorm{u_{\text{red}^2}^{\mu_i, \nu_j, k \Delta t} - 
    R_\rho(\Phi_{N'}(\theta_\text{DF}))(\mu_i, \nu_j, k \Delta t)}^2.
\end{multline*}
This is a rather obvious extension of \cite{franco2024}, but works surprisingly well.

\subsection{DOD-DL-ROM} \label{subsec: dod-dl-rom}

To define the approximate parameter-to-solution map of the DOD-DL-ROM, we introduce a decoder architecture $\Psi_{N'} := f_{N'}^\text{D}$ with input dimension $n$ and output dimension $N'$, and its corresponding Encoder $\Psi'_n := f_n^\text{E}$. 
Let $\Phi_n := \Phi_n^\text{DF}$ denote a DFNN architecture with input dimension $p+q+1$ and output dimension $n$. 
Then the \emph{DOD-DL-ROM} map
is given for each $(\mu, \nu, t) \in \Theta \times \Theta'
\times [0, T]$ via
\begin{equation*}
    \uhthat = \A \cdot R_\rho(\Phi_{\tilde{\V}})(\mu, t) \cdot 
    \left(R_\rho(\Psi_{N'}) \circ R_\rho(\Phi_n)\right) (\mu, \nu, t) \in \R^{N_h}.
\end{equation*}

The training process for the DOD-DL-ROM is similar to that of the DOD+DFNN, where we first find the optimal DOD parameter $\theta^*_\text{DOD} \in \Theta(\Phi_{\tilde{\V}})$ by solving the optimization problem in Definition \ref{def: DOD}. This freezes the DOD, and we set $\V_{\mu, t} := \V_{\mu, t}(\theta_\text{DOD}^*)$ for $(\mu, t) \in \Theta \times [0, T]$.
We then reduce the training data using
\begin{equation*}
    u_{\text{red}^2}^{\mu, \nu, t} := \V_{\mu, t}^T \G \uht \in \R^{N'} \qquad 
    \text{for all } (\mu, \nu, t) \in \Pcal_1 \times \Pcal_2 \times \Tcal.
\end{equation*}
The optimization problem for the DOD-DL-ROM for $\theta = 
(\theta_\text{DOD}, \theta_\text{D}, \theta_\text{DF}) \in 
\Theta(\Phi_{\tilde{\V}}) \times \Theta(f_N^\text{D}) \times \Theta(\Phi_n^\text{DF})$ is given by minimizing the loss function
\begin{multline*}
    \Lcal_\text{DOD-DL-ROM}\left(R_\rho(P(\Psi_{N'} \odot \Phi_n, \Psi'_n)(\theta)) \big\vert \{\mu_i, \nu_j, k \Delta t\}_{i,j,k}, 
    \{u_{\text{red}^2}^{\mu_i, \nu_j, k \Delta t} \}_{i,j,k}\right) =  \\
        \frac{1}{N_\text{data}} \sum_{i=1}^{N_{s_1}} \sum_{j=1}^{N_{s_2}} 
        \sum_{k = 1}^{N_t} \, \bigg( \frac{\omega_h}{2} \, 
        \euknorm{u_{\text{red}^2}^{\mu_i, \nu_j, k \Delta t} - 
        \left(R_\rho(\Psi_{N'}(\theta_\text{D})) \circ R_\rho(\Phi_n(\theta_\text{DF}))\right) (\mu_i, \nu_j, k \Delta t)}^2 \\
        + \, \frac{1 - \omega_h}{2} \, \euknorm{R_\rho(\Psi'_n(\theta_\text{E}))(u_{\text{red}^2}^{\mu_i, \nu_j, k \Delta t}) 
        - R_\rho(\Phi_n^\text{DF}(\theta_\text{DF}))(\mu_i, \nu_j, k \Delta t)}^2 \bigg).
\end{multline*}
This reduction allows us to use an algorithm that only requires the $n$- and $N'$-dimensional Euclidean norm, 
resulting in lower computational complexity for each training epoch, assuming an already optimized DOD.

\section{Error and complexity analysis} \label{sec:error_analysis}

\subsection{Error decomposition and convergence analysis} \label{section: error decomposition}

To analyze the ROMs in question, it is essential to recognize the different aspects of error inherent to their deployment. Specifically, we can subdivide the inaccuracy using the triangle inequality, which occurs in both the POD- and DOD-based data-driven ROMs. This separation allows us to distinguish between the error of suboptimal projection and the error of suboptimal latent parametric dynamics.

Note that the general handling of "DOD-based" (DOD+DNN) or "POD-based" (POD+DNN) refers to any architecture that relies on a projection by DOD or POD and uses a neural network of any architecture to fit the coefficients for the modes.

Before formulating and proving the decompositions, we predefine some elements of the corresponding theorems to make them less cumbersome to read. The visual portrayal of their workflow is given by Figure \ref{fig:dimensional-workflow-pod-dod}.

Let $n, N', N, N_A$ and $N_h$ be the fixed natural number hierarchy from 
Assumption \ref{assumption: dimension hierarchy}. 
We define the approximate POD-based parameter-to-solution map as follows:
\begin{equation*}
    \Gcal_\text{POD+DNN}: \Theta \times \Theta' \times [0, T] \to \R^{N_h}; \quad (\mu, \nu, t) \mapsto 
    \A_\text{P} \cdot \hat{q}_\text{P}(\mu, \nu, t),
\end{equation*}
where $\A_\text{P} \in \R^{N_h \times N}$ is a POD computed by Definition \ref{def: POD} for the dimension $N$ with the mass matrix $\G$, and
\begin{equation*}
    \hat{q}_\text{P} := R_\rho(\Phi^\text{P}_N(\theta_\text{PDNN})) : \R^{p+q+1} \to \R^N
\end{equation*}
is the realization of some neural network $\Phi^\text{P}_N$ with input dimension $p + q + 1$ and output dimension $N$,
and any neural network weights $\theta_\text{PDNN} \in \Theta(\Phi^\text{P}_N)$. 
Similarly, we define the approximate DOD-based parameter-to-solution map as follows:
\begin{equation*}
    \Gcal_\text{DOD+DNN}: \Theta \times \Theta' \times [0, T] \to \R^{N_h}; \quad (\mu, \nu, t) \mapsto 
    \V_{\mu, t} \cdot \hat{q}_\text{D}(\mu, \nu, t),
\end{equation*}
where 
\begin{equation*}
    \V_{\mu, t} := \A \cdot \tilde{\V}_{\mu, t}(\theta_\text{DOD}) \in \R^{N_h \times N'},
    \quad \text{for } (\mu, t) \in \Theta \times [0, T],
\end{equation*}
and any neural network weights $\theta_\text{DOD} \in \Theta(\Phi_{\tilde{\V}})$ according to Definition 
\ref{def: DOD} and the used POD matrix $\A \in \R^{N_h \times N_A}$. Further, let
\begin{equation*}
    \hat{q}_\text{D} := R_\rho(\Phi^\text{D}_{N'}(\theta_\text{DDNN})) : \R^{p+q+1} \to \R^{N'}
\end{equation*}
be the realization of some neural network $\Phi^\text{D}_{N'}$ with input dimension $p + q + 1$ and output dimension $N'$,
and any neural network weights $\theta_\text{DDNN} \in \Theta(\Phi^\text{D}_{N'})$.

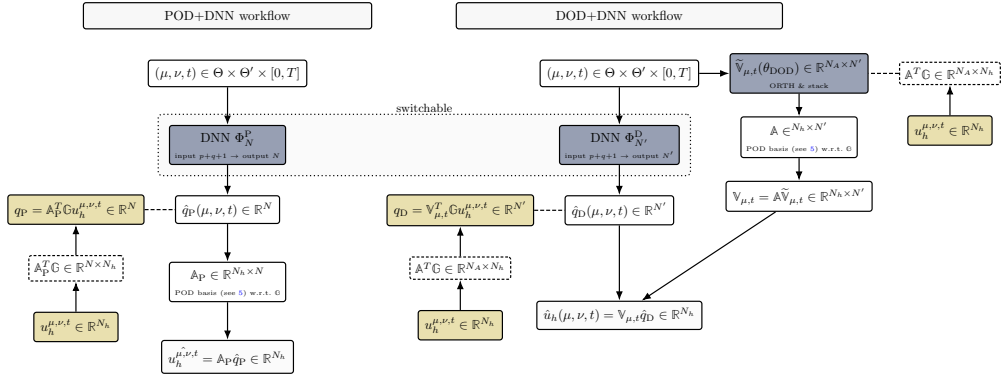
\begin{figure}[tb]
    \centering
    \begin{adjustbox}{max width=\linewidth}
    \begin{tikzpicture}[
      >=Latex,
      node distance=10mm and 14mm,
      every path/.style={line width=.9pt},
      every node/.append style={fill opacity=0.5, text opacity=1},
      box/.style={draw, rounded corners=2pt, align=center, inner sep=5pt},
      head/.style={draw, rounded corners=2pt, align=center, inner sep=6pt, fill=black!5},
      small/.style={font=\footnotesize},
      tiny/.style={font=\scriptsize},
      lbl/.style={font=\scriptsize, inner sep=1pt},
      op/.style={draw, align=center, inner sep=4pt, rounded corners=2pt},
      dashedop/.style={op, dash pattern=on 2.1pt off 1.4pt},
      transform shape
    ]
    \node[head, minimum width=7.6cm] (PODHEAD) {POD+DNN workflow};
    \node[head, minimum width=8.6cm, right=22mm of PODHEAD] (DODHEAD) {DOD+DNN workflow};
    \node[box, below=8mm of PODHEAD, xshift=0mm] (INP1) {$(\mu,\nu,t)\in\Theta\times\Theta'\times[0,T]$};
    \node[box, below=8mm of DODHEAD] (INP2) {$(\mu,\nu,t)\in\Theta\times\Theta'\times[0,T]$};
  
    \node[op, below=10mm of INP1, fill=cividisDark] (DNNP) {DNN $\Phi^{\mathrm{P}}_{N}$\\\tiny input $p{+}q{+}1$ $\rightarrow$ output $N$};
    \node[box, below=8mm of DNNP] (QP) {$\hat q_{\mathrm{P}}(\mu,\nu,t)\in\mathbb{R}^{N}$};
  
    \node[op, below=10mm of QP] (AP) {$\A_{\mathrm{P}}\in\mathbb{R}^{N_h\times N}$\\\tiny POD basis (see \ref{def: POD}) w.r.t. $\G$};
  
    \node[box, below=10mm of AP, fill=black!3] (UHP) {$\hat \uht =  \A_{\text{P}}\hat q_{\text{P}} \in \R^{N_h}$};
  
    \node[box, left=8mm of QP, yshift=0mm, fill=cividisLight] (QPTrue) {$q_{\mathrm{P}}=\A_{\mathrm{P}}^{T}\G \uht \in \R^{N}$};
    \node[dashedop, below=8mm of QPTrue, yshift=0mm] (APTG) {$\A_{\mathrm{P}}^{T} \G \in \R^{N \times N_h}$};
    \node[box, below=8mm of APTG, yshift=0mm, fill=cividisLight] (UTrueP) {$\uht \in \R^{N_h}$};
  
    \draw[->] (INP1) -- (DNNP);
    \draw[->] (DNNP) -- (QP);
    \draw[->] (QP) -- (AP);
    \draw[->] (AP) -- (UHP);
  
    \draw[->] (UTrueP) -- (APTG);
    \draw[->] (APTG) -- (QPTrue);
    \draw[dash pattern=on 3pt off 1.5pt] (QPTrue.east) -- ($(QP.west)+(0,0)$);

    \node[op, below=10mm of INP2, fill=cividisDark] (DNND) {DNN $\Phi^{\mathrm{D}}_{N'}$\\\tiny input $p{+}q{+}1$ $\rightarrow$ output $N'$};
    \node[op, right=8mm of INP2, fill=cividisDark] (DODINNER) {$\tilde \V_{\mu,t}(\theta_{\mathrm{DOD}})\in\mathbb{R}^{N_A\times N'}$\\\tiny ORTH \& stack};

    \node[dashedop, right=8mm of DODINNER] (ATG) {$\A^{T} \G \in \R^{N_A \times N_h}$};
    \node[box, below=8mm of ATG, fill=cividisLight] (UTrueD2) {$\uht \in \R^{N_h}$};

    \node[op, below=6mm of DODINNER] (AUP) {$\A \in ^{N_h\times N'}$\\\tiny POD basis (see \ref{def: POD}) w.r.t. $\G$};
    \node[box, below=6mm of AUP] (VMT) {$\V_{\mu,t}=\A \tilde \V_{\mu,t} \in \R^{N_h\times N'}$};
  
    \node[box, below=8mm of DNND] (QD) {$\hat q_{\mathrm{D}}(\mu,\nu,t)\in\mathbb{R}^{N'}$};
  
    \node[box, below=20mm of QD, fill=black!3] (UHD) {$\hat u_h(\mu,\nu,t)= \V_{\mu,t} \hat q_{\mathrm{D}}\in\R^{N_h}$};
  
    \node[box, left=8mm of QD, fill=cividisLight] (QDTrue) {$q_{\mathrm{D}} = \V_{\mu,t}^T \G \uht \in \R^{N'}$};
    \node[dashedop, below=8mm of QDTrue] (VTTG) {$\A^{T} \G \in \R^{N_A \times N_h}$};
    \node[box, below=8mm of VTTG, fill=cividisLight] (UTrueD) {$\uht \in \R^{N_h}$};
  
    \draw[->] (INP2) -- (DODINNER);
    \draw[->] (DODINNER) -- (AUP);
    \draw[->] (AUP) -- (VMT);
    \draw[->] (INP2) -- (DNND);
    \draw[->] (DNND) -- (QD);
    \draw[->] (VMT) -- (UHD);
    \draw[->] (QD) -- (UHD);
  
    \draw[->] (UTrueD) -- (VTTG);
    \draw[->] (VTTG) -- (QDTrue);
    \draw[dash pattern=on 3pt off 1.5pt] (QDTrue.east) -- ($(QD.west)+(0,0)$);
    \draw[dash pattern=on 3pt off 1.5pt] (ATG) -- (DODINNER);
    \draw[->] (UTrueD2) -- (ATG);
 
    \begin{pgfonlayer}{background}
      \node[draw, rounded corners=3pt, dotted, line width=.9pt, fill=black!4, inner sep=8pt,
            fit=(DNNP)(DNND)] (SWBOX) {};
    \end{pgfonlayer}
    \node[tiny, fill=white, inner sep=1pt] at ([yshift=1.6mm]SWBOX.north) {switchable};

    \end{tikzpicture}
    \end{adjustbox}
    \caption{Dimensional workflow for POD+DNN (left) and DOD+DNN (right). 
    Solid arrows show the inference; dashed lines denote training.}
    \label{fig:dimensional-workflow-pod-dod}
  \end{figure}

We assume that we have a set of training data
\begin{equation*}
    \left\{ (\mu_j, \nu_j, t_j), u_j \right\}_{j=1}^{N_\text{data}} :=
    \left\{(\mu_i, \nu_j, k \Delta t), u_h^{\mu_i, \nu_j, k \Delta t}\right\}_{(i, j, k)=(1, 1, 1)}^{(N_{s_1}, N_{s_2}, N_t)}
    \! \subset (\Theta \times \Theta' \times [0, T]) \times \R^{N_h},
\end{equation*}
conforming to the label and feature notation above. 
The labels have been sampled in accordance with Assumption \ref{assumption sample}.

For both algorithms, we define the relative error as
\begin{equation} \label{eq: relative error}
    \Ecal_\text{R} := \left( \int_{\Theta \times \Theta' \times [0, T]} \frac{\norm{\uht - \uhthat}^2}{\norm{\uht}^2} 
    d(\mu, \nu, t)\right)^{1/2},
\end{equation}
where $\uhthat \in \R^{N_h}$ can be given either with regard to the 
approximate parameter-to-solution map of the DOD+DNN or the POD+DNN.

To follow a similar strategy of proof as in \textcite[Theorem 1]{brivio2023} for the POD+DNN, we need 
an auxiliary statement that precedes the similarity of the definition of the DOD projection error to the POD projection 
error. 
The key idea is to employ the Hornik Theorem to ensure the existence of a neural network architecture and corresponding neural 
network that finds an approximation for $s$. Here, $s(\mu, t)$ is the unique projector matrix that minimizes the objective functional
\begin{equation*}
    (\mu, t, \V) \mapsto \int_{\Theta'} \norm{\uht - \V \V^T \G \uht}^2 d (\nu),
\end{equation*}
for all choices of $\V \in \R^{N_h \times N'}$. 
Under the triangle inequality for the distance to the optimal minimizer of the objective functional, this can be separated 
into the mistake of fitting the network to 
the minimizer, which is arbitrarily small, and the mistake of the optimal minimizer. This is specifically 
the Kolmogorov $N'$-width under fixed $(\mu, t)$. Note that we also need to take into account the peculiar 
structure of the DOD, which uses pre-reduction via POD. This creates an additional error, which is zero in the case 
that the pre-reduction is skipped completely.
To address this, we define the optimal DOD as follows:
\begin{definition}[Optimal DOD] \label{def: DOD functional}
    Let $N', N_A, N_h \in \N$ be fixed by Assumption \ref{assumption: dimension hierarchy}. 
    Assume we have $\mu \in \Theta, t \in [0, T]$ and $\V = \A \tilde{\V} \in [-1, 1]^{N_h \times N'}$, with
    $\A \in \R^{N_h \times N_A}$ fixed being $\G$-orthonormal, and $\tilde{\V} \in [-1, 1]^{N_A \times N'}$. Then 
    we define $J: \Theta \times [0, T] \times [-1, 1]^{N_h \times N'} \to \R$ the \emph{objective DOD functional} via
    \begin{equation*}
        J(\mu, t, \V) := \int_{\Xcal} \norm{\uht - \V \V^T \G \uht}^2 d(\nu),
    \end{equation*}
    for any compact set $\Xcal \subset \Theta'$ with $\text{dim}(\Xcal) = q$.
    We also define a Borel measurable map $s: \Theta \times [0, T] \to [-1, 1]^{N_A \times N'}$ 
    that produces $\G$-orthonormal matrices given by
    \begin{equation*}
        s: (\mu, t) \mapsto \underset{{\tilde{\V} \in [-1, 1]^{N_A \times N'}}}{\text{argmin}}  
        J(\mu, t, \A \tilde{\V}),
    \end{equation*}
    which is called the \emph{optimal DOD} for $\Xcal$, if it is uniquely defined.
\end{definition}
We have the following lemma:
\begin{lemma} \label{lemma: DOD functional is lipschitz}
    Let $J: \Theta \times [0, T] \times [-1, 1]^{N_h \times N'} \to \R$ be an objective DOD functional
    for some $\Xcal \subset \Theta'$ with $\text{dim}(\Xcal) = q$ as in 
    Definition \ref{def: DOD functional},
    then under Assumption \ref{assumption parameter-to-solution map}, we have that $J$ is Lipschitz,
    and the optimal DOD $s$ exists.
\end{lemma}
The auxiliary statement for the upcoming theorem is a refined version of \textcite[Theorem 1]{franco2024}.
It quantitates the proximity of an arbitrary DOD to the optimal choice. Both
proofs are included in the Appendix \ref{appendix: proofs for main matter}.
\begin{proposition}[DOD Bound] \label{prop: DOD Existence}
    For each geometric parameter $\mu \in \Theta$ and time $t \in [0, T]$, let $\Scal^\text{pre-red}_{\mu, t} :=
    \{ \A^T \G \uht \}_{\nu \in \Theta'}$ be the 
    $(\mu, t)$-invariant pre-reduced solution submanifold, where $\A \in \R^{N_h \times N_A}$ is a 
    $\G$-orthonormal matrix. Let $\epsilon > 0$ be arbitrary.
    Then, under the Assumptions \ref{assumption sample} and \ref{assumption parameter-to-solution map}, 
    there is a neural network architecture $\Phi_{\tilde{\V}}$ with
    neural network weights $\theta_\text{DOD}^* \in \Theta(\Phi_{\tilde{\V}})$ 
    for a DOD according to Definition \ref{def: DOD} and
    a constant $c_\A' > 0$ only depending on $\A, N_A$ and the parameter space, such that
    \begin{equation*}
        \int_{\Theta \times \Theta' \times [0, T]} 
        \norm{\uht - \V_{\mu, t} \V_{\mu, t}^T \G \uht}^2 d(\mu, \nu, t)
        \leq \epsilon  + c_\A' + \int_{\Theta \times [0, T]} d_{N'}(\mathcal{S}^\text{pre-red}_{\mu, t})^2 d(\mu, t).
    \end{equation*}
    Here, the DOD matrix is given by
    $\V_{\mu, t} := \V_{\mu, t}(\theta_\text{DOD}^*) \in \R^{N_h \times N'}$ for each $(\mu, t) \in \Theta \times [0, T]$.
\end{proposition}

\begin{remark} \label{remark: discrete optimal DOD}
    We have structured Proposition \ref{prop: DOD Existence} with the $L^1$-integral in mind, but we 
    may completely omit the integral with regard to $d(\mu, t)$, 
    by formally setting the 
    objective DOD functional for $\Pcal_2$ in regards to the statement and proof
    as
    \begin{equation*}
        \hat J(\mu, t, \V) := \frac{\euknorm{\Theta'}}{N_{s_2}} \sum_{\nu \in \Pcal_2} \norm{\uht - \V \V^T \G \uht}^2.
    \end{equation*}
    This culminates in the following version of Proposition \ref{prop: DOD Existence} for $\V_{\mu, t} :=
    \V_{\mu, t}(\theta^*_\text{DOD})$, when discarding the very last step in 
    equation (\ref{eq: last step for optimal DOD})
    \begin{equation*}
        \frac{\euknorm{\Theta'}}{N_{s_2}} \sum_{\nu \in \Pcal_2}
        \norm{\uht - \V_{\mu, t}
        \V_{\mu, t}^T \G \uht}^2 \\
        \leq \epsilon  + c_\A + \frac{\euknorm{\Theta'}}{N_{s_2}} 
        \min\limits_{\tilde{\V} \in \R^{N_A \times N'}} 
        \sum_{\nu \in \Pcal_2} 
        \euknorm{\uhtprered - \tilde{\V} \tilde{\V}^T \uhtprered}^2,
    \end{equation*}
    for any choice of $(\mu, t) \in \Theta \times [0, T]$.

    Note, that if $N_A = N_h$, we have $c_\A = 0$ by the definition
    of the POD as a canonical $\G$-orthonormalization. In this case, the right-hand-side sum
    would be over the non-reduced average of the KnW.
\end{remark}

We can now use this to assemble a decomposition in a comparable manner to the source
\cite{brivio2023}. This is the case, since we now have a similar statement
as the one of Proposition \ref{prop: eigenvalues and optimal projection under sampling}, in the form
of Remark \ref{remark: discrete optimal DOD}. In fact, define
\begin{equation} \label{eq: definition virtual eigenvalues}
    \sigma(\mu, t)_1^2 \geq \cdots \geq \sigma(\mu, t)_{n}^2 \quad \text{for any } n \leq r,
\end{equation}
as the eigenvalues of a virtual POD (see Definition \ref{def: POD}) of the discrete correlation matrix
\begin{equation*}
    K(\mu, t) := \frac{\euknorm{\Theta'}}{N_{s_2}} U_\text{pre-red}(\mu, t) U_\text{pre-red}(\mu, t)^T.
\end{equation*}
Here, the snapshot matrix $U_\text{pre-red}$ is given by
\begin{equation*}
    U_\text{pre-red}(\mu, t) := \left[u_\text{pre-red}^{\mu, \nu_1, t} \big\vert \cdots \big\vert u_\text{pre-red}^{\mu, \nu_{N_{s_2}}, t} \right]^T 
    \in \R^{N_A \times N_{s_{2}}},
\end{equation*}
for each fixed $(\mu, t) \in \Theta \times [0, T]$. 
By Proposition \ref{prop: eigenvalues and optimal projection under sampling}, it fulfills
\begin{equation*}
    \frac{\euknorm{\Theta'}}{N_{s_2}} 
        \min\limits_{\tilde{\V} \in \R^{N_A \times n}} 
        \sum_{\nu \in \Pcal_2} 
        \euknorm{\uhtprered - \tilde{\V} \tilde{\V}^T \uhtprered}^2 = \sum_{k > n} \sigma(\mu, t)^2_k,
\end{equation*}
for any $n \leq r$ as above.
In the same manner, we define a virtual infinite limit POD with the respective infinite correlation
matrix and their eigenvalues 
\begin{equation} \label{eq: definition virtual infinity eigenvalues}
    \sigma(\mu, t)_{1, \infty}^2 \geq \cdots \geq \sigma(\mu, t)_{n, \infty}^2 \quad \text{for any } n \leq r,
\end{equation}
that fulfill 
\begin{equation*} 
        \min\limits_{\tilde{V} \in \R^{N_A \times n}} 
        \int_{\nu \in \Theta'} 
        \euknorm{\uhtprered - \tilde{\V} \tilde{\V}^T \uhtprered}^2 d(\nu) = \sum_{k > n} \sigma(\mu, t)^2_{k, \infty}.
\end{equation*}

These virtual eigenvalues will help us to identify the error inherent to the DOD reduction. To be precise, this
will enable a distinction between network contribution and truly solution submanifold
dependent error.

\begin{theorem} \label{theo: relative error DOD+DNN}
    Define the preliminaries such as in the introduction to this Section \ref{section: error decomposition}. 
    Further, set $\V_{\mu, t} := \V_{\mu, t}(\theta^*_\text{DOD})$ for the choice of weights $\theta_\text{DOD}^* \in
    \Theta(\Phi_{\tilde{\V}})$ as in Proposition \ref{prop: DOD Existence}.
    Then, under the Assumptions \ref{assumption sample} and 
    \ref{assumption parameter-to-solution map}, we have for the relative error of the DOD+DNN
    
    \begin{equation}
        \mathcal{E}^\text{DOD+DNN}_R \leq \mathcal{E}_\text{S} + \mathcal{E}_\A 
        + \mathcal{E}_\text{DOD} + \mathcal{E}_\text{NN},
    \end{equation}
    where 
    
    \begin{enumerate}
        \item The sampling error 
        \[
            \mathcal{E}_\text{S} := \mathcal{E}_\text{S}\left(\mathcal{G}, 
            (\nu_i)_{i = 1}^{N_{s_2}}, N'\right)
        \]
        satisfies $\mathcal{E}_\text{S} \to 0$ uniformly a.s. in $\Theta \times [0, T]$ as $N_{s_2} \to \infty$, 
        with expectation 
        \[
            \E_\nu[\mathcal{E}_\text{S}] \leq \mathcal{O}(N_{s_2}^{-1/4}).
        \]
    
        \item The ambient error
        \[
            \mathcal{E}_\A := \Ecal_\A\left(\mathcal{G}, (\mu_i, \nu_i, t_i)_{i = 1}^{N_{\text{data}}},
            N_A\right)
        \]
        being equal to zero, if $N_A = N_h$.
        \item The DOD projection error 
        \[
            \mathcal{E}_{\text{DOD}} := \mathcal{E}_{\text{DOD}}\left(\mathcal{G}, 
            (\nu_i)_{i = 1}^{N_{s_2}}, N'\right)
        \]
        satisfies $\mathcal{E}_{\text{DOD}} \to \mathcal{E}_{\text{DOD}, \infty} + \varrho$
        uniformly a.s. in $\Theta \times [0, T]$ as $N_{s_2} \to \infty$
        for $\varrho = \varrho(\Phi_{\tilde{\V}})> 0$ arbitrarily small depending on
        the size and choice of weights of the inner DOD module, where
        $
            \mathcal{E}_{\text{DOD}, \infty} := \mathcal{E}_{\text{DOD}, \infty}(\mathcal{G}, N')
        $
        is independent of the data snapshots and of 
        order of the maximal linear KnW of the solution submanifold 
        $\mathcal{S}_{\mu, t} = \{ \uht \, \vert \, \nu \in \Theta' \}$ for any $(\mu, t) \in \Theta \times [0,T]$.
    
        \item The neural network approximation error
        \[
            \mathcal{E}_{\text{NN}} := \mathcal{E}_{\text{NN}}\left(\mathcal{G}, N', \Phi^\text{D}_{N'}\right)
        \]
    can be made arbitrarily small depending on the size and choice of weights of  $\Phi^\text{D}_{N'}$.
    \end{enumerate}

\end{theorem}
\begin{proof}
    Let 
    \begin{equation*}
        \norm{\cdot}_{L^2_\omega} := \left(\int_{\Theta \times \Theta' \times [0, T]} \frac{\norm{\cdot}^2}{\norm{\uht}^2} 
        d(\mu, \nu, t) \right)^{1/2},
    \end{equation*}
    define a map $\norm{\cdot}_{L^2_\omega}: L^2(\Theta \times \Theta' \times [0, T]; \R^{N_h}) \to \R^+$. Then this is
    a norm (see \cite[Appendix]{brivio2023}).
    Now define
    \begin{equation*}
        q_\text{D}(\mu, \nu, t) := \V_{\mu, t}^T \G \uht \in \R^{N'}, 
    \end{equation*}
    as the reduced solution regarding the DOD projection.
    Then we can apply the triangle inequality to achieve 
    \begin{equation*}
        \Ecal_\text{R}^\text{DOD+DNN} \leq \norm{\uht - \V_{\mu, t} \V_{\mu, t}^T \G \uht}_{L^2_\omega} + 
        \norm{\V_{\mu, t} q_\text{D}(\mu, \nu, t) - \V_{\mu, t} \hat{q}_\text{D}(\mu, \nu, t)}_{L^2_\omega}.
    \end{equation*}
    In turn, from here we can already infer the error stemming from the latent dynamics identification under the 
    neural network realization $\hat{q}_\text{D}$, i.e.
    \begin{equation}
        \mathcal{E}_\text{NN} := \left( \int_{\Theta \times \Theta' \times [0, T]} 
        \frac{\norm{\V_{\mu, t} q_\text{D}(\mu, \nu, t) - \V_{\mu, t} \hat{q}_\text{D}(\mu, \nu, t)}^2}{\norm{\uht}^2} 
        d(\mu, \nu, t) \right)^{1/2}.
    \end{equation}
    
    Moreover, we can extrapolate a similar procedure of splitting the remaining term into the respective errors
    as was done in \cite[Theorem 1]{brivio2023}.
    For this, we assume for the fixed $N' \in \N$, the quasi-optimal choice of weights $\theta_\text{DOD}^* \in
    \Theta(\Phi_{\tilde{\V}})$ for the underlying DOD neural network in accordance to Remark
    (ref{remark: discrete optimal DOD} for some fixed $\epsilon > 0$. We use the definition of
    $c_\A>0$ from that Remark and its subsequently defined virtual discrete 
    correlation eigenvalues $\sigma(\mu, t)_k$
    in equation (\ref{eq: definition virtual eigenvalues}) to achieve
    \begin{align*}
        &\norm{\uht - \V_{\mu, t} \V_{\mu, t}^T \G \uht}_{L^2_\omega} \\
        &\leq m^{-1} \left( \int_{\Theta \times [0, T]} \int_{\Theta'} 
        \norm{\uht - \V_{\mu, t} \V_{\mu, t}^T \G \uht}^2 d(\nu) d(\mu, t) \right)^{1/2} \\
        &= m^{-1} \left( \int \int
        \norm{\uht - \V_{\mu, t} \V_{\mu, t}^T \G \uht}^2 d(\nu) - \sum_{k > N'} \sigma(\mu, t)_k^2 +
        \sum_{k > N'} \sigma(\mu, t)_k^2 \, d(\mu, t) \right)^{1/2} \\
        &\leq \underbrace{m^{-1} \left\vert \, \int_{\Theta \times [0, T]} \int_{\Theta'} 
        \norm{\uht - \V_{\mu, t} \V_{\mu, t}^T 
        \G \uht}^2 d(\nu) -  \hat J(\mu, t, \A s) \,
        d(\mu, t) \, \right\vert^{1/2}}_{=: \Ecal_\text{S}} \\
        &\quad + \underbrace{m^{-1} \left( \int_{\Theta \times [0, T]} \norm{\uht - \A^T \A \G \uht}^2 d(\mu, t)
        \right)^{1/2}}_{=: \Ecal_\A}\\
        &\quad + \underbrace{m^{-1} \euknorm{\Theta \times [0, T]}^{1/2} \epsilon^{1/2} + m^{-1} \euknorm{\Theta \times [0, T]}^{1/2}
        \sup\limits_{(\mu, t) \in \Theta \times [0, T]} \sqrt{\sum_{k > N'} \sigma(\mu, t)_k^2}}_{=: \Ecal_\text{DOD}}.
    \end{align*}
    Note, that we used $\sqrt{a + b} \leq \sqrt{a} + \sqrt{b}$ for $a, b \geq 0$ for the split 
    and that $\hat J$ is given in Remark~\ref{remark: discrete optimal DOD}.
    Furthermore, we can easily see, that the supremum chosen in the last step is actually an exaggeration, since the mean
    \begin{equation*}
        \Ecal'_\text{DOD} := m^{-1} \left( \int_{\Theta \times [0, T]} \sqrt{\sum_{k > N'} \sigma(\mu, t)^2} 
        d(\mu, t) \right)^{1/2} + \frac{\epsilon^{1/2} \euknorm{\Theta \times [0, T]}^{1/2}}{m}
    \end{equation*} 
    is a more direct bound, also satisfying the upcoming convergence. 
    However, we state this in regard to the former definition to stay consistent
    with the later discussion surrounding Assumption \ref{assumption KnW decay}, that
    is formulated in $L^\infty$-norm.

    From here, we assert that $\Ecal_\text{S} \to 0$ almost surely for $N_{s_2} \to \infty$ by using
    the uniform strong law of large numbers \cite{newey1994}.
    In fact, we calculate
    \begin{equation*}
        \hat f(\nu; \mu, t) := \norm{\uht - \V_{\mu, t} \V_{\mu, t}^T \G \uht}^2 
        \leq M^2 \opnorm{\I_{N_h} - \V_{\mu, t}\V_{\mu,t}^T \G}^2
        \leq M^2 < \infty,
    \end{equation*}
    uniformly bounded and continuous 
    in $(\mu, t) \in \Theta \times [0, T]$ by Lemma \ref{lemma: DOD functional is lipschitz}.
    Assumption~\ref{assumption sample} gives compactness of $\Theta \times [0, T]$.
    Hence, we use the mentioned uniform law, to get a bound on the supremum, i.e.
    \begin{equation*}
        \sup\limits_{(\mu, t) \in \Theta \times [0, T]} \left\vert
            \int_{\Theta'} \hat f(\nu; \mu, t) d(\nu) - 
            \frac{\euknorm{\Theta'}}{N_{s_2}} \sum_{\nu \in \Pcal_2} \hat f(\nu; \mu, t)
        \right\vert \underset{\text{a.s.}}{\longrightarrow} 0.
    \end{equation*}
    Since this is an upper bound (with additional constant $\euknorm{\Theta \times \Theta' \times [0, T]}$) 
    for the $L^2$-integral, we get the wanted result.
    We also have $\E_\nu[\Ecal_\text{S}] \leq \Ocal(N_{s_2}^{-1/4})$ as a special case 
    of \cite[Proposition 1]{brivio2023}. Remark, that $\E_\nu$
    denotes the integral with regard to the samples in $\Theta'$.
    Moreover, we have from Proposition \ref{prop: existence and convergence to sigma infty}, that 
    for any $(\mu, t) \in \Theta \times [0, T]$
    \begin{equation*}
        \hat J(\mu, t, \A s(\mu, t)) = \sum_{k > N'} \sigma(\mu, t)_k^2 \underset{N_{s_2} \to \infty}{\longrightarrow} \sum_{k > N'}  
        \sigma(\mu, t)_{k, \infty}^2,
    \end{equation*}
    where $\hat J: \Theta \times [0, T] \times [-1, 1]^{N_A \times N'}$ is the objective DOD functional, 
    $s \in L^2(\Theta \times [0, T]; [-1, 1]^{N_A \times N'})$ the
    optimal DOD and $\A \in \R^{N_h \times N'}$ is the pre-reduction POD matrix from
    Remark \ref{remark: discrete optimal DOD}. This allows us to assert, that indeed, via
    an argument of restriction, 
    the map $(\mu, t) \mapsto \hat J(\mu, t, \A s(\mu, t))$ is continuous 
    and bounded by Lemma \ref{lemma: DOD functional is lipschitz}.
    Hence,
    \begin{equation*}
        (\mu, t) \mapsto \sqrt{\sum_{k>N'} \sigma(\mu, t)_k^2}
    \end{equation*}
    is continuous as well. This allows for the same use of the uniform strong law of large numbers
    \cite{newey1994}, resulting in
    \begin{multline*}
        \Ecal_\text{DOD} = \frac{\euknorm{\Theta \times [0, T]}^{1/2}}{m}
        \sup\limits_{(\mu, t) \in \Theta \times [0, T]} \sqrt{\sum_{k > N'} \sigma(\mu, t)_{k}^2}
        + \frac{\epsilon^{1/2} \euknorm{\Theta \times [0, T]}^{1/2}}{m}
        \\
        \underset{N_{s_2} \to \infty}{\longrightarrow} 
        \frac{\euknorm{\Theta \times [0, T]}^{1/2}}{m}
        \sup\limits_{(\mu, t) \in \Theta \times [0, T]} \sqrt{\sum_{k > N'} \sigma(\mu, t)_{k, \infty}^2}
        + \frac{\epsilon^{1/2} \euknorm{\Theta \times [0, T]}^{1/2}}{m} \\
        =: \Ecal_{\text{DOD}, \infty} + \varrho(\Phi_{\tilde{\V}}).
    \end{multline*}
    Here, the constant $\varrho(\Phi_{\tilde{\V}})$ can be arbitrarily improved for a more complex choice of the 
    inner DOD module architecture. This dependency
    on the complexity is suggested by Yarotsky Theorem \ref{theo: yarotsky}, since the $L^\infty$ error 
    correlates with the 
    $\Hcal_h$ representation error. This is due to the compactness of the time-parameter space, that is given
    by Assumption \ref{assumption sample}. But that correlation is only truly
    provable under additional assumptions since the necessary regularity of $s$ has not been shown.

    Finally, we know that by Remark~\ref{remark: discrete optimal DOD} $c_\A = 0$ implying 
    $\Ecal_\A = 0$, if $N_A = N_h$, which thus
    concludes the proof.
\end{proof}

As another important result, similar to \cite[Theorem 2]{brivio2023}, we realize a lower bound. This is due
to the inherent flaw of using a linear projector for specified $(\mu, t)$. Intuitively, the specific
proportionality must be related to the relation of the local magnitude differences of the solution. This observation
can be reformulated into the following theorem.

\begin{theorem} \label{theo: lower bound dod+dnn}
    Under the assumptions of Theorem \ref{theo: relative error DOD+DNN}, we have that 
    \begin{equation*}
        \Ecal_\text{R}^\text{DOD+DNN} \geq \frac{m}{M} \Ecal'_{\text{DOD}, \infty},
    \end{equation*}
    where we set
    \begin{equation*}
        \Ecal'_{\text{DOD}, \infty} := m^{-1}
    \int_{\Theta \times [0, T]} \sqrt{\sum_{k > N'} \sigma(\mu, t)_{k, \infty}^2} d(\mu, t),
    \end{equation*}
    and $\sigma(\mu, t)_{k, \infty}$ the eigenvalues for a non-reduced continuous correlation matrix.
\end{theorem}
\begin{proof}
    Generally, we know the projector to be a suboptimal choice of projection under any neural network
    realization
    \begin{equation*}
        \frac{\norm{\uht - \V_{\mu, t} \hat{q}_\text{D}(\mu, \nu, t)}^2}{\norm{\uht}^2}
        \geq
        \frac{\norm{\uht - \V_{\mu, t} \V_{\mu, t}^T \G \uht}^2}{\norm{\uht}^2},
    \end{equation*}
    where $\V_{\mu, t} \in \R^{N_h \times N'}$ is the DOD matrix as given by Definition \ref{def: DOD}
    for any $(\mu, \nu, t) \in \Theta \times \Theta' \times [0, T]$. On the other hand,
    we start from the opposite direction of the inequality for a virtual $(\mu, t)$-dependent 
    POD matrix $\A^\infty_{\mu, t} \in \R^{N_h \times N'}$ for each
    $(\mu, t) \in \Theta \times [0, T]$. This is the POD obtained as a data limit of 
    the continuous correlation matrix
    \begin{equation*}
        U := \int_{\Theta'} \left(u_\text{pre-red}^{\mu, \nu, t}\right)^T u_\text{pre-red}^{\mu, \nu, t} d(\nu),
    \end{equation*}
    which exists and is well-defined via the Proposition 
    \ref{prop: eigenvalues and optimal projection under sampling}. This limit POD matrix $\A_\infty$ has eigenvalues
    $\sigma(\mu, t)^2_{1, \infty} \geq \cdots \geq \sigma(\mu, t)^2_{N', \infty}$. Hence, using the 
    definition of $\Ecal'_{\text{DOD}, \infty}$, described in the proof of Theorem \ref{theo: relative error DOD+DNN}, we can
    calculate
    \begin{align}
        (\Ecal'_{\text{DOD}, \infty})^2 &= m^{-2} 
        \int_{\Theta \times [0, T]} \sum_{k > N'} \sigma(\mu, t)_{k,\infty}^2 \, d(\mu, t) \\
        &= m^{-2}
        \int_{\Theta \times [0, T]} 
        \int_{\Theta'} \norm{\uht - \A^\infty_{\mu, t} (\A^\infty_{\mu, t})^T \G \uht}^2 d(\nu) d(\mu, t)
        \label{align: lower bound dod 1} \\
        &\leq m^{-2}
        \int_{\Theta \times [0, T]}
        \int_{\Theta'} \norm{\uht - \V_{\mu, t} \V_{\mu, t}^T \G \uht}^2 d(\nu) d(\mu, t)
        \label{align: lower bound dod 2} \\
        &= m^{-2} \int_{\Theta \times \Theta' \times [0, T]} \frac{\norm{\uht - \V_{\mu, t} \V_{\mu, t}^T \G \uht}^2}{\norm{\uht}^2}
        \norm{\uht}^2 d(\mu, \nu, t) \\
        &\leq \frac{M^2}{m^2} (\Ecal_\text{R})^2. \label{align: lower bound dod 3}
    \end{align}
    Here, we use Proposition~\ref{prop: eigenvalues and optimal projection under sampling} for the
    continuous POD matrix in Step (\ref{align: lower bound dod 1}), the optimality of the same proposition 
    for Step (\ref{align: lower bound dod 2}) and the first equation of this proof in 
    Step (\ref{align: lower bound dod 3}) with the use of the bound on $\norm{\uht}$ from Assumption
    \ref{assumption parameter-to-solution map}. 
\end{proof}

\subsection{Complexity Bounds} \label{section: upper complexity bounds}

In this section, we will specify under the Assumptions \ref{assumption sample}--\ref{assumption: dimension hierarchy} the upper complexity bounds for
our three main algorithms: the POD-DL-ROM, the DOD+DFNN and the DOD-DL-ROM, as introduced in
Section \ref{section: data-driven ROMs}. One further assumption is necessary for the result concerning the DOD-based
ROMs, and as such will be introduced after the statement regarding the POD-DL-ROM.

For the following discussion, we again set some preliminaries to shorten the theorem's layout.
We assume the setup of the previous Section \ref{section: error decomposition}, where we recall
defining a sample to be used for all subsequent statements. Furthermore, based on this sample,
set up $N$ and $N'$ as follows
\begin{equation} \label{eq: definition of N, N'}
  \begin{aligned}
    N = \underset{n \in \N}{\operatorname{arg \, min}} \left( \sum_{k > n} \sigma_k^2 \leq \frac{m^2}{9} \varepsilon^2 \right),\\ 
    N' = \underset{n \in \N}{\operatorname{arg \, min}} \left( \sup\limits_{(\mu, t) \in \Theta \times [0, T]} 
    \sum_{k > n} \sigma(\mu, t)_k^2 \leq \frac{m^2}{16 \euknorm{\Theta \times [0, T]}} 
    \varepsilon^2 \right),
  \end{aligned}
\end{equation}
with $\varepsilon > 0$ yet unspecified, and $\sigma_k^2$ according to 
Definition \ref{def: POD} and $\sigma(\mu, t)_k^2$ to
equation~(\ref{eq: definition virtual eigenvalues}). Note, that under Assumption~\ref{assumption KnW decay} the hierarchy
of Assumption~\ref{assumption: dimension hierarchy} is still preserved, if the choice of $\varepsilon$ is small
enough and $N_\text{data}$ is large enough.
Next, we simplify the definition of the approximate parameter-to-solution map notation 
\begin{equation*}
    \Gcal_\text{POD-DL-ROM} : \Theta \times \Theta' \times [0, T] \to \R^{N_h}; \quad (\mu, \nu, t)
    \mapsto \A_\text{P} \cdot \left( \psi_N \circ \phi_n \right) (\mu, \nu, t),
\end{equation*}
for the POD-DL-ROM
under weights $\theta_\text{POD-DL-ROM} = (\theta_\text{D}, \theta_\text{DF})$ with 
\begin{equation*}
    \quad \phi_n := R_\rho(\Phi_n(\theta_\text{DF})): \R^{p+q+1} \to \R^n,
    \psi_N := R_\rho(\Psi_N(\theta_\text{D})): \R^n \to \R^{N}.
\end{equation*}
Note, that in general $\A_\text{P} \neq \A$, since like before
$\A_\text{P} \in \R^{N_h \times N}$
and $\A \in \R^{N_h \times N_A}$. Similarly, we simplify the approximate parameter-to-solution map notation
\begin{equation*}
    \Gcal_\text{DOD+DFNN} : \Theta \times \Theta' \times [0, T] \to \R^{N_h}; \quad (\mu, \nu, t)
    \mapsto \V_{\mu, t} \cdot \phi_{N'}(\mu, \nu, t),
\end{equation*}
for the DOD+DFNN (cf.\!  Sec.\!  \ref{subsec: DOD+DFNN})
under weights $\theta_\text{DOD+DFNN} = (\theta_\text{DOD}, \theta_{\text{DF}_2})$ with
\begin{gather*}
    \V_{\mu, t} := \A \cdot R_\rho(\Phi_{\tilde{\V}}(\theta_\text{DOD}))(\mu, t) \in \R^{N_h \times N'},
    \qquad \text{for each } (\mu, t) \in \Theta \times [0, T] \text{, and } \\
    \phi_{N'} := R_\rho(\Phi_{N'}(\theta_{\text{DF}_2})): \R^{p+q+1} \to \R^{N'}.
\end{gather*} 
Furthermore, we can also simplify the approximate parameter-to-solution map
\begin{equation*}
    \Gcal_\text{DOD-DL-ROM} : \Theta \times \Theta' \times [0, T] \to \R^{N_h}; \quad (\mu, \nu, t)
    \mapsto \V_{\mu, t} \cdot (\psi_{N'} \circ \phi_n)(\mu, \nu, t),
\end{equation*}
of the DOD-DL-ROM (cf.\! Sec.\!  \ref{subsec: dod-dl-rom}) 
with weights $\theta_\text{DOD-DL-ROM} = (\theta_\text{DOD}, \theta_{\text{D}_2}, 
\theta_{\text{DF}_3})$ as
\begin{gather*}
    \V_{\mu, t} := \A \cdot R_\rho(\Phi_{\tilde{\V}}(\theta_\text{DOD}))(\mu, t) \in \R^{N_h \times N'},
    \qquad \text{for each } (\mu, t) \in \Theta \times [0, T] \text{, and } \\
    \psi_{N'} := R_\rho(\Psi_{N'}(\theta_{\text{D}_2})): \R^n \to \R^{N'}, \quad
    \phi_{n} := R_\rho(\Phi_n(\theta_{\text{DF}_3})): \R^{p+q+1} \to \R^n.
\end{gather*}
In analogy to the previous Section \ref{section: error decomposition}, we visualize this with
Figure \ref{fig:dimensional-workflow-pod-dlrom-dod-dlrom}. Note, that we omit the DOD+DFNN, since
it is build in structure exactly like the DOD+DNN.

\begin{figure}[tb]
    \centering
    \begin{adjustbox}{max width=\linewidth}
    \begin{tikzpicture}[
      >=Latex,
      node distance=10mm and 14mm,
      every path/.style={line width=.9pt},
      every node/.append style={fill opacity=0.5, text opacity=1},
      box/.style={draw, rounded corners=2pt, align=center, inner sep=5pt},
      head/.style={draw, rounded corners=2pt, align=center, inner sep=6pt, fill=black!5},
      small/.style={font=\footnotesize},
      tiny/.style={font=\scriptsize},
      lbl/.style={font=\scriptsize, inner sep=1pt},
      op/.style={draw, align=center, inner sep=4pt, rounded corners=2pt},
      dashedop/.style={op, dash pattern=on 2.1pt off 1.4pt},
      transform shape
    ]
    \node[head, minimum width=7.6cm] (PODHEAD) {POD-DL-ROM workflow};
    \node[head, minimum width=8.6cm, right=22mm of PODHEAD] (DODHEAD) {DOD-DL-ROM workflow};
    \node[box, below=8mm of PODHEAD] (INP1) {$(\mu,\nu,t)\in\Theta\times\Theta'\times[0,T]$};
    \node[box, below=8mm of DODHEAD] (INP2) {$(\mu,\nu,t)\in\Theta\times\Theta'\times[0,T]$};
  
    \node[op, below=10mm of INP1, fill=cividisDark] (DFNNP) {DFNN $\phi_{n}$\\\tiny $\R^{p+q+1}\!\to\!\R^{n}$};
    \node[op, above=9mm of QPTrue, fill=cividisDark] (ENCP) {Encoder $\Psi'_{n}$\\\tiny $\R^{N}\!\to\!\R^{n}$};
    \node[op, below=5mm of DFNNP, fill=cividisDark] (DECP) {Decoder $\Psi_{N}$\\\tiny $\R^{n}\!\to\!\R^{N}$};
    \node[box, below=10mm of DECP] (QP) {$\hat q_{\mathrm P}(\mu,\nu,t)\in\R^{N}$};
  
    \node[op, below=10mm of QP] (AP) {$\A_{\mathrm P}\in\R^{N_h\times N}$\\\tiny POD basis (see \ref{def: POD}) w.r.t.\ $\G$};
    \node[box, below=10mm of AP, fill=black!3] (UHP) {$\hat \uht=\A_{\mathrm P}\,\hat q_{\mathrm P}\in\R^{N_h}$};
  
    \node[box, left=8mm of QP, fill=cividisLight] (QPTrue) {$q_{\mathrm P}=\A_{\mathrm P}^{T}\G\,\uht\in\R^{N}$};
    \node[dashedop, below=8mm of QPTrue] (APTG) {$\A_{\mathrm P}^{T}\G\in\R^{N\times N_h}$};
    \node[box, below=8mm of APTG, fill=cividisLight] (UTrueP) {$\uht\in\R^{N_h}$};
  
    \draw[->] (INP1) -- (DFNNP);
    \draw[->] (DFNNP) -- (DECP);
    \draw[->] (DECP) -- (QP);
    \draw[->] (QP) -- (AP);
    \draw[->] (AP) -- (UHP);
  
    \draw[->] (UTrueP) -- (APTG);
    \draw[->] (APTG) -- (QPTrue);
    \draw[dash pattern=on 3pt off 1.5pt] (QPTrue.east) -- (QP.west);
  
    \draw[dash pattern=on 3pt off 1.5pt] (ENCP) -- (DFNNP);
    \draw[->] (QPTrue) -- (ENCP);
  
    \node[op, right=8mm of INP2, fill=cividisDark] (DODINNER) {$\tilde \V_{\mu,t}(\theta_{\mathrm{DOD}})\in\R^{N_A\times N'}$\\\tiny ORTH \& stack};
    \node[dashedop, right=8mm of DODINNER] (ATG) {$\A^{T}\G\in\R^{N_A\times N_h}$};
    \node[box, below=8mm of ATG, fill=cividisLight] (UTrueD2) {$\uht\in\R^{N_h}$};
  
    \node[op, below=6mm of DODINNER] (AUP) {$\A\in\R^{N_h\times N_A}$\\\tiny POD basis (see \ref{def: POD}) w.r.t.\ $\G$};
    \node[box, below=6mm of AUP] (VMT) {$\V_{\mu,t}=\A\,\tilde \V_{\mu,t}\in\R^{N_h\times N'}$};
  
    \node[op, below=10mm of INP2, fill=cividisDark] (DFNND) {DFNN $\phi_{n}$\\\tiny $\R^{p+q+1}\!\to\!\R^{n}$};
    \node[op, above=9mm of QDTrue, fill=cividisDark] (ENCD) {Encoder $\Psi'_{n}$\\\tiny $\R^{N'}\!\to\!\R^{n}$};
    \node[op, below=5mm of DFNND, fill=cividisDark] (DECD) {Decoder $\Psi_{N'}$\\\tiny $\R^{n}\!\to\!\R^{N'}$};
    \node[box, below=10mm of DECD] (QD) {$\hat q_{\mathrm D}(\mu,\nu,t)\in\R^{N'}$};
  
    \node[box, below=20mm of QD, fill=black!3] (UHD) {$\hat u_h(\mu,\nu,t)=\V_{\mu,t}\,\hat q_{\mathrm D}\in\R^{N_h}$};
  
    \node[box, left=8mm of QD, fill=cividisLight] (QDTrue) {$q_{\mathrm D}=\V_{\mu,t}^{T}\G\,\uht\in\R^{N'}$};
    \node[dashedop, below=8mm of QDTrue] (VTTG) {$\A^{T}\G\in\R^{N_A\times N_h}$};
    \node[box, below=8mm of VTTG, fill=cividisLight] (UTrueD) {$\uht\in\R^{N_h}$};
  
    \draw[->] (INP2) -- (DODINNER);
    \draw[->] (DODINNER) -- (AUP);
    \draw[->] (AUP) -- (VMT);
    \draw[->] (INP2) -- (DFNND);
    \draw[->] (DFNND) -- (DECD);
    \draw[->] (DECD) -- (QD);
    \draw[->] (VMT) -- (UHD);
    \draw[->] (QD) -- (UHD);
  
    \draw[->] (UTrueD) -- (VTTG);
    \draw[->] (VTTG) -- (QDTrue);
    \draw[dash pattern=on 3pt off 1.5pt] (QDTrue.east) -- (QD.west);
    \draw[dash pattern=on 3pt off 1.5pt] (ATG) -- (DODINNER);
    \draw[->] (UTrueD2) -- (ATG);
  
    \draw[dash pattern=on 3pt off 1.5pt] (ENCD) -- (DFNND);
    \draw[->] (QDTrue) -- (ENCD);

    \begin{pgfonlayer}{background}
      \node[draw, rounded corners=3pt, dotted, line width=.9pt, fill=black!4, inner sep=8pt,
            fit=(ENCP)(ENCD)(DFNNP)(DECP)(DFNND)(DECD)] (SWBOX) {};
    \end{pgfonlayer}
    \node[tiny, fill=white, inner sep=1pt] at ([yshift=1.6mm]SWBOX.north) {DL-ROM};
  
    \end{tikzpicture}
    \end{adjustbox}
    \caption{POD-DL-ROM (left) and DOD-DL-ROM (right) with autoencoder training. Solid arrows show inference; 
    dashed lines denote training.}
    \label{fig:dimensional-workflow-pod-dlrom-dod-dlrom}
  \end{figure}
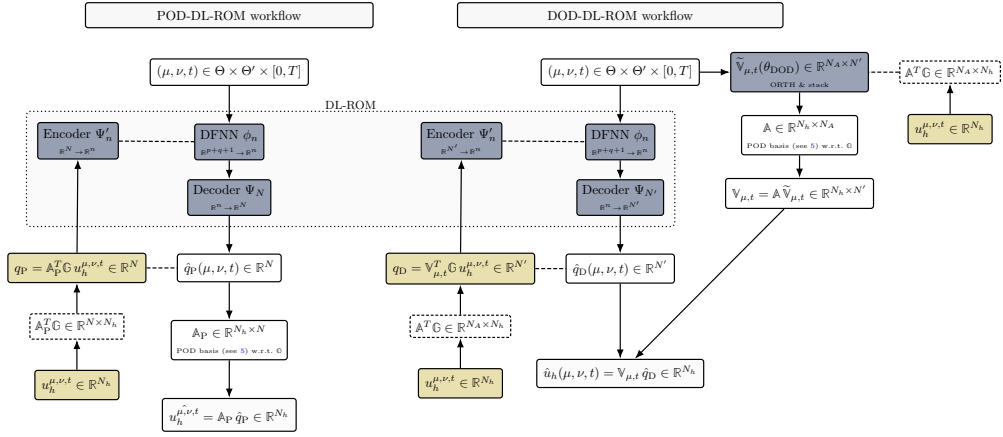

Furthermore, let $\chi_* \in C^s(\R^n; \R^{N'})$ and $\chi_*' \in C^{s'}(\R^{N'}; \R^n)$ respectively be the maps attaining
the perfect embedding Assumption \ref{assumption perfect embedding}, with $s \gg s' \geq 2$. Since $N' < N$ in
general, the perfect embedding is implied for $N$ as well; let $\psi_* \in C^s(\R^n; \R^N)$ 
and $\psi_*' \in C^{s'}(\R^N; \R^n)$ be the maps attaining the equivalent for $N$.

Finally, set constants
\begin{equation*}
    C_1' := \sup\limits_{\vert \alpha \vert \leq s'} \sup\limits_{v \in \R^{N'}} \euknorm{D^\alpha \chi_*(v)}, 
    \quad 
    C_2' := \sup\limits_{\vert \alpha \vert \leq s} \sup\limits_{w \in \R^n} \euknorm{D^\alpha \chi_*'(w)},
\end{equation*} 
and $C_1, C_2$ analogously for $\psi_*$ and $\psi_*'$. 

Under these preliminaries for the POD-DL-ROM, the authors have shown \cite[Theorem 3]{brivio2023}. This
theorem shows an upper bound for the complexity of all neural networks in the model under a given 
Probably-Almost-Correct (PAC) bound
with a precision level $0 < \varepsilon < 1$ and a probability $0 < \delta < 1$. These complexity bounds are
given by
\begin{enumerate}
    \item[] $L_{\Psi_N} = c \log(\varepsilon^{-1})$ layers and
    $\omega_{\Psi_N} = cN \varepsilon^{-n/(s-1)} \log(\varepsilon^{-1})$ active weights,
\end{enumerate}
for the Decoder, and 
\begin{enumerate}
    \item[] $L_{\Phi_n} = c \log(\varepsilon^{-1})$ layers and
    $\omega_{\Phi_n} = cn \varepsilon^{-(p+q+1)} \log(\varepsilon^{-1})$ active weights,
\end{enumerate}
for the Deep Feedforward Neural Network.

In order to achieve a similar result for the DOD-based data-driven ROMs, we need one more crucial assumption
for the regularity of the optimal DOD map $s$. This assumption has not been discussed yet and
its properties could pose an interesting question for further research. However, it is necessary in order to
make use of Yarotsky Theorem \ref{theo: yarotsky} for a bound on the complexity.

\begin{assumption} \label{assumption additional lipschitz}
    Let the objective DOD functional $J: \Theta \times [0, T] \times [-1, 1]^{N_h \times N'}$ be given
    for $\Xcal = \Pcal_2$ as in Definition \ref{def: DOD functional}, and the corresponding
    optimal DOD
    \begin{equation*}
        s: \Theta \times [0, T] \mapsto [0, 1]^{N_A \times N'}; \quad (\mu, t) \mapsto 
        \underset{{\tilde{\V} \in [-1, 1]^{N_A \times N'}}}{\text{argmin}}  
        J(\mu, t, \A \tilde{\V})
    \end{equation*}
    be $\hat{L}$-Lipschitz.
\end{assumption}

With this additional assumption, we now formulate our main complexity result.

\begin{theorem} \label{theo: PAC bound for DOD-DL-ROM}
    Let the preliminaries of the DOD-DL-ROM be given as in the introduction of this Section 
    \ref{section: upper complexity bounds}, $N_A = N_h$ explicitly and the Assumption 
    \ref{assumption additional lipschitz} be true.
    Then there is $N_{s_2} = N_{s_2}(\delta, \varepsilon)$, 
    a constant $c = c(\Theta, \Theta', T, L, \hat{L}, C_1, C_2, p, q, n, s, s')$,
    an architecture and a choice of neural network weights for that architecture, such that
    the DOD-DL-ROM approximate parameter-to solution map $(\mu, \nu, t) \mapsto \V_{\mu, t} \cdot 
    \hat{q}_\text{D}(\mu, \nu, t) :=
    \V_{\mu, t} \cdot (\psi_{N'} \circ \phi_n)(\mu, \nu, t)$
    attains
    \begin{equation*}
        \prob_\nu\left[\Ecal^{DOD-DL-ROM}_R < \varepsilon\right] > 1 - \delta.
    \end{equation*}
    Further, the complexity of the underlying architecture is at most
    \begin{enumerate}
        \item[] $L_{\Phi_{\tilde{\V}}} = c \log(\varepsilon^{-1})$ layers and
        $\omega_{\Phi_{\tilde{\V}}} = cN_A N' \varepsilon^{-(p+1)} \log(\varepsilon^{-1})$ active weights,
    \end{enumerate}
    for the inner DOD module,
    \begin{enumerate}
        \item[] $L_{\Psi_{N'}} = c \log(\varepsilon^{-1})$ layers and
        $\omega_{\Psi_{N'}} = cN' \varepsilon^{-n/(s-1)} \log(\varepsilon^{-1})$ active weights,
    \end{enumerate}
    for the Decoder, and
    \begin{enumerate}
        \item[] $L_{\Phi_n} = c \log(\varepsilon^{-1})$ layers and
        $\omega_{\Phi_n} = cn \varepsilon^{-(p+q+1)} \log(\varepsilon^{-1})$ active weights,
    \end{enumerate}
    for the Deep Feedforward Neural Network.
\end{theorem}
\begin{proof}
    Similar as to the proof of \cite[Theorem 3]{brivio2023}, we will try to bound each quantity
    of the statement in Theorem \ref{theo: relative error DOD+DNN}, but with the distinction, that $\Ecal_{\text{DOD}}$ is
    stochastic in nature. For this, we will employ Remark \ref{remark: discrete optimal DOD}. 
    Here, we remark, that taking $N_A = N_h$ implies $\Ecal_\A = 0$.
    
    The sampling error can be bounded in the same way as before, i.e. we can directly bound 
    $\Ecal_S = \Ecal_S(N_{s_2})$ independently of $N'$, under the Assumption \ref{assumption sample}
    by the uniform law of large numbers \cite{newey1994}. Indeed, for all $1 > \delta > 0$ and $\varepsilon > 0$ there is $N_{s_2}$, such that
    \begin{equation} \label{eq: S error dod-dl-rom}
        \prob_\nu[\Ecal_S < \varepsilon / 4] > 1 - \delta.
    \end{equation}
    Now recall the definition of $\Ecal_{\text{DOD}}$ in the proof of Theorem \ref{theo: relative error DOD+DNN} and 
    further setting the $\epsilon > 0$ from the upstream Remark \ref{remark: discrete optimal DOD}, realizing
    an "optimal" discrete approximation for a fixed data size $N_{s_2}$, as
    \begin{equation*}
        \epsilon = \frac{\varepsilon m}{4 \euknorm{\Theta \times [0, T]}^{1/2}}.
    \end{equation*}
    Hence, concluding, under the additional use of definition of $N'$ in equation
    (\ref{eq: definition of N, N'}), in the following approximation for the DOD error
    \begin{equation} \label{eq: dod error dod-dl-rom}
        \begin{aligned}
            \Ecal_{\text{DOD}} &= \frac{\euknorm{\Theta \times [0, T]}^{1/2}}{m} \left( \epsilon + 
            \sup\limits_{(\mu, t) \in \Theta \times [0, T]} \sqrt{\sum_{k > N'} \sigma(\mu, t)_k^2} \right) \\
            &< \frac{\varepsilon}{4} + \left( \frac{\euknorm{\Theta \times [0, T]}^{1/2}}{m}
            \sup\limits_{(\mu, t) \in \Theta \times [0, T]} \sqrt{\sum_{k > N'} \sigma(\mu, t)_k^2}
            \right) \\
            &\leq \frac{\varepsilon}{4} + \left( \frac{\euknorm{\Theta \times [0, T]}^{1/2}}{m} \cdot 
            \frac{\varepsilon m}{4 \euknorm{\Theta \times [0, T]}^{1/2}} \right)
            = \frac{\varepsilon}{2}.
        \end{aligned}
    \end{equation}

    For the following,
    we assume w.l.o.g. a normalized 
    feature space $\Theta \times [0, T] \subset (0, 1)^{p+1}$.
    Under the Assumption \ref{assumption additional lipschitz},
    we get
    \begin{align*}
        \norm{s}_{W^{1, \infty}((0,1)^{N_A \times N'})} &= \norm{s}_{L^\infty((0,1)^{N_A \times N'})} + 
        \norm{D s}_{L^\infty((0,1)^{N_A \times N'})} 
        \leq 1 + L =: B.
    \end{align*}
    This means, we have $[s]_{ij} \in \Fcal_{1, p+1, \infty, B}$ for each $(i, j) \in \{1, \dots, N_A\} \times \{1, \dots, N'\}$.
    Therefore, there is
    a neural network architecture $[\tilde{\V}_0]_{ij}$
    with corresponding weights $\theta^*_{ij}$, such that
    \begin{equation*}
        \norm{[s]_{ij} - R_\rho([\tilde{\V}_0]_{ij}(\theta^*_{ij}))}_{L^\infty(\Theta \times [0, T]; \R)} \leq \epsilon,
    \end{equation*}
    for each $(i, j) \in \{1, \dots, N_A\} \times \{1, \dots, N'\}$. This architecture has at most 
    \begin{enumerate}
        \item[] $L_{[\tilde{\V}_0]_{ij}} = c'' \log(\varepsilon^{-1})$ layers and
        $\omega_{[\tilde{\V}_0]_{ij}} = c'' \varepsilon^{-(p+1)} \log(\varepsilon^{-1})$ active weights,
    \end{enumerate}
    for some $c'' = c''(\Theta, T, p, \hat{L}) > 0$ by Yarotsky Theorem \ref{theo: yarotsky}. 
    The use of this theorem is due to 
    the $\hat{L}$-Lipschitz property of the optimal DOD, justifying $W^{1, \infty}$-regularity with bound $\hat{L}$.

    From this point, 
    using the steps in the proof of Proposition \ref{prop: DOD Existence} under the specifications of Remark
    \ref{remark: discrete optimal DOD}, we can find a neural network architecture $\Phi_{\tilde{\V}}$ for
    the inner DOD module and corresponding optimal weight $\theta^*_\text{DOD} \in \Theta(\Phi_{\tilde{\V}})$,
    such that
    \begin{equation*}
        \norm{s - R_\rho(\Phi_{\tilde{\V}}(\theta^*_\text{DOD}))}_{L^\infty(\Theta \times [0, T]; \R^{N_h \times N'})} 
        \leq \epsilon,
    \end{equation*}
    with at most
    $L_{\Phi_{\tilde{\V}}} = c' \log(\varepsilon^{-1})$ layers and
        $\omega_{\Phi_{\tilde{\V}}} = c' N_A N' \varepsilon^{-(p+1)} \log(\varepsilon^{-1})$ active weights,
    where $c' = 3 + c''$.

    Moreover, we can use the exact same arguments as in the bound for $\Ecal_{\text{NN}}$ for the POD-DL-ROM as in the
    proof of \cite[Theorem 3]{brivio2023}, which yields
    \begin{equation} \label{eq: nn error dod-dl-rom}
        \Ecal_{\text{NN}} < \frac{\varepsilon}{4}.
    \end{equation}
    For sake of uniformity, we then take $c > 0$ to be the maximal of each individual constant of all three
    separate neural networks.
    Finally, putting equations (\ref{eq: S error dod-dl-rom}), (\ref{eq: dod error dod-dl-rom}) and (\ref{eq: nn error dod-dl-rom})
    together, we get 
    \begin{equation*}
        \Ecal_R \leq \Ecal_S + \Ecal_{\text{DOD}} + \Ecal_{\text{NN}} < \frac{\varepsilon}{4} +  \frac{\varepsilon}{2}
        + \frac{\varepsilon}{4} = \varepsilon,
    \end{equation*}
    with probability greater than $1 - \delta$ concluding our proof.
\end{proof}

In a very similar vein, we obtain the following theorem for the DOD+DFNN 
(see Appendix \ref{appendix: proofs for main matter} for a proof). 

\begin{theorem} \label{theo: PAC bound for DOD+DFNN}
    Let the preliminaries of the DOD+DFNN be given as in the introduction of this Section 
    \ref{section: upper complexity bounds}, $N_A = N_h$ explicitly and the Assumption 
    \ref{assumption additional lipschitz} be true.
    Then there is $N_{s_2} = N_{s_2}(\delta, \varepsilon)$, 
    a constant $c = c(\Theta, \Theta', T, L, \hat{L}, p, q, N')$,
    an architecture and a choice of neural network weights for that architecture, such that
    the DOD+DFNN approximate parameter-to solution map $(\mu, \nu, t) \mapsto 
    \V_{\mu, t} \cdot \hat{q}_\text{D}(\mu, \nu, t) :=
    \V_{\mu, t} \cdot \phi_{N'}(\mu, \nu, t)$ 
    attains
    \begin{equation*}
        \prob_\nu\left[\Ecal^{DOD+DFNN}_R < \varepsilon\right] > 1 - \delta.
    \end{equation*}
    Further, the complexity of the underlying architecture is at most
    \begin{enumerate}
        \item[] $L_{\Phi_{\tilde{\V}}} = c \log(\varepsilon^{-1})$ layers and
        $\omega_{\Phi_{\tilde{\V}}} = cN_A N' \varepsilon^{-(p+1)} \log(\varepsilon^{-1})$ active weights,
    \end{enumerate}
    for the inner DOD module, and
    \begin{enumerate}
        \item[] $L_{\Phi_{N'}} = c \log(\varepsilon^{-1})$ layers and
        $\omega_{\Phi_{N'}} = c N' \varepsilon^{-(p+q+1)} \log(\varepsilon^{-1})$ active weights,
    \end{enumerate}
    for the Deep Feedforward Neural Network.
\end{theorem}

\subsection{Comparison of POD and DOD based approaches}

In this section, we will compare the presented data-driven ROM techniques with each other
in the context of the Theorems \cite[Theorem 3]{brivio2023},
\ref{theo: PAC bound for DOD-DL-ROM} and \ref{theo: PAC bound for DOD+DFNN}, or in other
words; we will examine the differences of the upper complexity bounds for the POD-based 
neural network approach and both DOD-based neural network approaches.

Within this framework, we will play around with potentially sharper assumptions, especially
hypothetically improving regularity of the parameter-to-solution map $\Gcal$ in Assumption 
\ref{assumption parameter-to-solution map} and the regularity of the optimal DOD $s$ in
Assumption \ref{assumption additional lipschitz}. 
To that end, additionally to the framework of the Section \ref{section: upper complexity bounds}, let
$\Gcal \in W^{r, \infty}(\Theta \times \Theta' \times [0, T]; \R^{N_h})$ for some $r \geq 1$,
    $s \in W^{m, \infty}(\Theta \times [0, T]; \R^{N_h \times N'})$ for some $m \geq 1$.

These two regularities namely determine 
the effectivity of the neural network under Yartosky (and/or Gühring) Theorem \ref{theo: yarotsky}.
To have a substantial difference in the use of the DOD, as opposed to the POD, we focus
the reader's attention on Assumption \ref{assumption KnW decay}, which ensures the following, as a result of Proposition 
\ref{prop: existence and convergence to sigma infty} and the definitions of the continuous correlation
matrices in Definition \ref{def: POD}. 

\begin{enumerate}
    \item The eigenvalues of the POD matrix given
    under the continuous correlation matrix of the entire solution manifold $\Scal$
    exhibit a slow decay, i.e. for $C > 0$ and $0 < \alpha < 1$
    \begin{equation*}
         \sqrt{\sum_{k > n} \sigma_{k, \infty}^2} \leq C n^{-\alpha}.
    \end{equation*}
    \item The eigenvalues of the $(\mu, t)$-fixed virtual POD matrix given
    under the continuous non-reduced correlation matrix of the $(\mu, t)$-dependent solution submanifold
    $\Scal_{\mu, t} = \{ \uht \, \vert \, \nu \in \Theta' \}$ exhibit a quick decay, i.e. for $C' > 0$ and $\beta \geq 1$.
    \begin{equation*}
        \sup_{(\mu, t) \in \Theta \times [0, T]} \sqrt{\sum_{k > n} \sigma(\mu, t)_{k, \infty}^2}
        \leq C' n^{-\beta}.
    \end{equation*}
\end{enumerate}

With these two bounds, we can presumably infer the order of the mode dimensions $N$ and $N'$ for very large 
data $N_s, N_t \to \infty$ and $N_A = N_h$. Namely,
\begin{equation} \label{eq: order of N, N'}
    N = \Ocal\left(\varepsilon^{-\frac{1}{\alpha}}\right), \qquad N' = \Ocal\left(\varepsilon^{-\frac{1}{\beta}}\right),
\end{equation}
for the upper complexity bounds of the POD-DL-ROM, the DOD+DFNN and the DOD-DL-ROM 
of fixed precision $\varepsilon > 0$ as is dictated in the beginning of Section
\ref{section: upper complexity bounds}. In fact, we know that by
Proposition \ref{prop: existence and convergence to sigma infty} $\sum_{k > n} \sigma_k^2 \to \sum_{k > n}
\sigma_{k, \infty}^2$ for $N_s, N_t \to \infty$, and therefore
\begin{equation*}
    N = \underset{n \in \N}{\operatorname{arg \, min}} \left\{\sum_{k > n} \sigma_k^2 \geq
    \frac{m^2}{9} \varepsilon^2
    \right\} \underset{N_s, N_t \to \infty}{\longrightarrow} 
    \underset{n \in \N}{\operatorname{arg \, min}} \left\{ C n^{-2\alpha} \geq
        \frac{m^2}{9} \varepsilon^2 
    \right\} = \Ocal(\varepsilon^{-\frac{1}{\alpha}}),
\end{equation*}
and a similar argument for $N'$. These rough estimates for the order of $N$ and $N'$ then let us 
get an idea for the upcoming comparisons. It will mainly involve the weighing of the higher complexity
cost of the DOD-based ROMs against the lighter, but potentially less reductive, POD-based methods.
Note, that the assumption of $N_A = N_h$ could be relaxed, if the pre-reduced $(\mu, t)$-dependent solution
manifold could be guaranteed to have fast KnW decay.

\begin{proposition} \label{prop: regularity discussion for POD vs DOD}
    Let a precision $0 < \varepsilon < 1$ and $0 < \delta < 1$ be given uniformly
    for the Theorems \cite[Theorem 3]{brivio2023}, \ref{theo: PAC bound for DOD-DL-ROM} and 
    \ref{theo: PAC bound for DOD+DFNN} under the Assumptions 
    \ref{assumption sample}--\ref{assumption additional lipschitz}. This is equivalent to having
    \begin{align*}
        \prob[\Ecal_\text{R}^\text{POD-DL-ROM} < \varepsilon] &> 1 - \delta, \\
        \prob_\nu[\Ecal_\text{R}^\text{DOD-DL-ROM} < \varepsilon] &> 1 - \delta, \\
        \prob_\nu[\Ecal_\text{R}^\text{DOD+DFNN} < \varepsilon] &> 1 - \delta.
    \end{align*}
    Further, let the constants $N_h, n, s \in \N$ from the respective assumptions satisfy
    $
        \frac{n}{s-1} > \frac{\log(N_h)}{\log(\varepsilon^{-1})}.
    $
    Additionally, assume that the optimal DOD map $s \in W^{m, \infty}(\Theta \times [0,T])$ with $m \in \N$,
    such that
    $
        m \gtrsim \frac{p+1}{\frac{n}{s-1} - \log_\varepsilon\left(\frac{N - N'}{N_h N'}\right)}
    $
    is satisfied. Then it is guaranteed that either
    \begin{equation*}
        \omega_\text{DOD+DFNN} \lesssim \omega_\text{DOD-DL-ROM} \lesssim \omega_\text{POD-DL-ROM}, 
        \qquad \text{if } r \geq \frac{1}{n}(s-1)(p+q+1),
    \end{equation*}
    or only
    \begin{equation*}
        \omega_\text{DOD-DL-ROM} \lesssim \omega_\text{POD-DL-ROM},
        \qquad \text{if } r \leq \frac{1}{n}(s-1)(p+q+1),
    \end{equation*}
    depending on $r \in \N$ setting the regularity 
    of the parameter-to-solution map $\Gcal \in W^{r, \infty}(\Theta \times \Theta' \times [0, T])$.
    If we additionally set up Assumption \ref{assumption KnW decay} with the constants $0 < \alpha < 1$
    and $\beta \geq 1$, and let $N_s, N_t \to \infty$, we may infer a sharper minimal order of $m$, with
    \begin{equation*}
        m \gtrsim \frac{p+1}{\frac{n}{s-1} - \frac{\log(N_h)}{\log(\varepsilon^{-1})} + 
        \frac{\beta - \alpha}{\beta \alpha}}.
    \end{equation*}
\end{proposition}
\begin{proof}
    For the following, we start with the wanted statement and simply rearrange the targeted inequality
    until it poses a lower bound for $m$. Note beforehand, that in the case $r \geq (s-1)(p+q+1)/n$,
    we know the number of maximal weights for the DOD-DL-ROM to be larger than for the DOD+DFNN, excusing
    the distinction.
    We can now set up
    \begin{equation*}
        \omega_\text{DOD-DL-ROM} \lesssim \omega_\text{POD-DL-ROM}
    \end{equation*}
    as a starting point. This itself implies, since we can subtract the latent dynamics contributions
    without changing the order, that
    \begin{equation*}
        N_h N' \varepsilon^{-(p+1)/m} + 
        N' \varepsilon^{-n/(s-1)} \lesssim N \varepsilon^{-n/(s-1)}.
    \end{equation*}
    From here, we subtract the decoder contribution from the left to fit it to the right, then divide by 
    the positive term $\varepsilon^{-n/(s-1)}$ and set the constants on one side, to attain
    \begin{equation*}
        \varepsilon^{-\left(\frac{(p+1)}{m} - \frac{n}{(s-1)}\right)} \lesssim \frac{N - N'}{N_h N'}.
    \end{equation*}
    By Assumption \ref{assumption: dimension hierarchy},
    we have that $0 < (N-N')/(N_hN') \leq (N + N')/(N_hN') < 1$. 
    Then taking the logarithm with regard to $\varepsilon < 1$ as base, hence being 
    a monotonely decreasing function, therefore 
    flipping the inequality, 
    and subsequently rearranging and multiplying by negative one, 
    leaves us with
    \begin{equation*}
        \frac{p+1}{m} \lesssim \frac{n}{s-1} - \log_\varepsilon\left( \frac{N - N'}{N_hN'} \right).
    \end{equation*}
    With the note, that $\log_\varepsilon(\cdot) = -\log(\cdot)/\log(\varepsilon^{-1})$ 
    we can conclude, through the assumption on $N_h, n$ and $s$ from
    the statement of the proposition that both sides are strictly positive,
    hence resulting in the wanted bound after inverting and multiplying by $(p+1)$. 

    Let us now assume $N_s, N_t \to \infty$ and the Assumption \ref{assumption KnW decay} in order
    to prove the extension of the proposition. Immediately, one can recognize from the beginning of
    the section, that $N = \Ocal(\varepsilon^{-1/\alpha})$ and $N' = \Ocal(\varepsilon^{-1/\beta})$.
    We now deduce, that
    \begin{equation*}
        \frac{N - N'}{N_h N'} = \frac{1}{N_h} \left( \frac{N}{N'} - 1 \right) = \Ocal\left( 
            \frac{1}{N_h} \left( \varepsilon^{-\frac{1}{\alpha} + \frac{1}{\beta}} - 1 \right) \right)
            \lesssim \frac{1}{N_h} \varepsilon^{\frac{\alpha - \beta}{\alpha \beta}},
    \end{equation*}
    thus implying by monotonicity of the logarithm
    \begin{equation*}
        \log_\varepsilon\left(\frac{N-N'}{N_h N'}\right) = 
        -\frac{\log\left(\frac{N-N'}{N_h N'}\right)}{\log(\varepsilon^{-1})}
        \gtrsim -\frac{\log(1/N_h)}{\log(\varepsilon^{-1})} + \frac{\alpha - \beta}{\alpha \beta}
        = \frac{\log(N_h)}{\log(\varepsilon^{-1})} - \frac{\beta - \alpha}{\beta \alpha},
    \end{equation*}
    which concludes our proof. Note, that we rewrite the logarithms in such a way, that the signs make sense
    in the context of the constants.
\end{proof}

\section{Numerical Results} \label{sec:experiments}

In this section, we present numerical experiments surrounding an industry-relevant problem
to support the theoretical findings. The problem in question is of a medium, subject to
in- and outflow through a catalytic filter. Its pressure and velocity distribution is modelled using a
parametric Darcy flow equation and a subsequent time-dependent advection-reaction equation.

{\bf Darcy flow problem.}
We model the catalytic filter as a porous material, whose location inside the reaction chamber is parametric
(see Figure \ref{fig: darcy model}).
The strong formulation of the Darcy flow problem is thus given by
\begin{equation} \label{eq: darcy problem}
    \begin{cases}
    \vec{b}^\mu = -k^\mu \nabla p, & \text{in } \Omega=(0,1)^2,\\
    \nabla \cdot \vec{b}^\mu  = 0, & \text{in } \Omega,\\[1mm]
    \end{cases}
\end{equation}
with boundary conditions $p = 0  \text{ on } \Gamma_{\mathrm{out}}$, $\vec{b}^\mu \cdot \vec{n} = \vec{j}^\mu \cdot \vec{n}  \text{ on } \Gamma_{\mathrm{in}}$ and $\vec{b}^\mu \cdot \vec{n} = 0  \text{ on } \Gamma_{\mathrm{wall}}$, where $\vec{n}$ denotes the unit normal, and the non-parametric part of the domain is
\begin{equation*}
    \Gamma_{\mathrm{in}}=\{0\}\times(0.8, 1), \quad
    \Gamma_{\mathrm{out}}=\{1\}\times(0, 0.2), \quad
    \Gamma_{\mathrm{wall}}=\partial\Omega - (\Gamma_{\mathrm{in}}\cup\Gamma_{\mathrm{out}}).
\end{equation*}
Furthermore, we set the parametric inflow condition using the shifted 
Poiseuille profile in dependence of an inflow angle $\mu_2 \in [-\pi/4, \pi/4]$ and non-parametric
inflow velocity $v_\mathrm{in} = 10$ and spike $\eta = 0.2$, i.e.
\begin{equation*}
    \vec{j}^\mu(x)=
    \frac{v_{\mathrm{in}}}{4\eta} \left(0.1^2 - (x_2 - 0.9)^2\right) \cdot \binom{\cos(\mu_2)}{\sin(\mu_2)}.
\end{equation*}
On the other hand, we specifically allow the coating height $h_c^\mu := \mu_1$ 
and the washcoat height $h_w^\mu := 0.3 - \mu_1$, to be
parametric for $\mu_1 \in [0.1, 0.3]$ under the condition, 
that they fill the space between midpoint and inflow and outflow boundary respectively. 
Such a formulation also gives us the definition of the catalyst filter regions via
\begin{equation*}
    \Omega_c^\mu = \{(x_1, x_2) \in \Omega : h_w^\mu< \euknorm{x_2-\tfrac12} < h_w^\mu +h_c^\mu \}, \quad
    \Omega_w^\mu = \{(x_1, x_2) \in \Omega : \euknorm{x_2-\tfrac12} < h_w^\mu \}.
\end{equation*}
This is meant to model
a situation in which the general connection geometry of a catalyst component cannot be varied,
but the internals may. For simplicity, we assume that the permeabilities $k_c = 0.02$ and $k_w = 0.5$ are 
constant and non-parametric in the interiors of the parametric regions $\Omega_c^\mu$ and $\Omega_w^\mu$ as follows
\begin{equation*}
    k^\mu := \mathbb{1}_{\Omega - (\Omega_c^\mu \cup \Omega_w^\mu)} + 
        k_c \mathbb{1}_{\Omega_c^\mu} + k_w \mathbb{1}_{\Omega_w^\mu}.
\end{equation*}

\begin{figure}[tbp]
    \centering
    \begin{subfigure}{0.48\linewidth}
        \centering
        \begin{tikzpicture}[>=Latex]

            \def\W{4.2}
            \def\H{4.2}

            \draw[thick] (0,0) rectangle (\W,\H);

            \fill[white] (0,3.15) rectangle (\W,4.2);
            \fill[gray!25] (0,2.65) rectangle (\W,3.15);
            \fill[gray!45] (0,1.55) rectangle (\W,2.65);
            \fill[gray!25] (0,1.05) rectangle (\W,1.55);
            \fill[white] (0,0) rectangle (\W,1.05);

            \draw[gray!60] (0,3.15) -- (\W,3.15);
            \draw[gray!60] (0,2.65) -- (\W,2.65);
            \draw[gray!60] (0,1.55) -- (\W,1.55);
            \draw[gray!60] (0,1.05) -- (\W,1.05);
            \draw[dotted, black] (0,2.1) -- (1.7, 2.1);
            \draw[dotted, black] (2.5,2.1) -- (4.2, 2.1);
            \draw[<->, black, thin] (0.5, 2.1) -- (0.5, 1.55);
            \draw[<->, black, thin] (0.1, 2.1) -- (0.1, 3.15);

            \draw[line width=1.4pt,cividisLight] (0,3.15) -- (0,4.2);
            \draw[line width=1.4pt,cividisDark] (\W,0) -- (\W,1.05);

            \node at (0.9, 1.8) {$h^\mu_w$};
            \node at (0.5, 2.85) {$h^\mu_c$};

            \node at (0.55,3.7) {$\Gamma_{\mathrm{in}}$};
            \node at (2.1,2.1) {$\Omega^\mu_w$};
            \node at (3.55,0.52) {$\Gamma_{\mathrm{out}}$};
        \end{tikzpicture}
        \caption{Catalytic filter model}
    \end{subfigure}
    \begin{subfigure}{0.48\linewidth}
        \centering
        \begin{tikzpicture}[>=Latex]
            \def\S{4.2}
            \draw[thin] (0,0) rectangle (\S,\S);

            \foreach \x/\y/\ux/\uy/\len in {
                0.217/0.217/0.0270/-0.9996/0.113,
                0.637/0.217/0.2715/-0.9624/0.114,
                1.057/0.217/0.7072/-0.7071/0.115,
                1.477/0.217/0.7791/-0.6269/0.119,
                1.897/0.217/0.8246/-0.5658/0.122,
                2.317/0.217/0.8563/-0.5165/0.126,
                2.737/0.217/0.8794/-0.4760/0.131,
                3.157/0.217/0.8954/-0.4452/0.136,
                3.577/0.217/0.9068/-0.4216/0.141,
                3.997/0.217/0.9150/-0.4035/0.147,
                0.217/0.637/0.0014/-1.0000/0.149,
                0.637/0.637/0.0143/-0.9999/0.149,
                1.057/0.637/0.0412/-0.9992/0.149,
                1.477/0.637/0.0692/-0.9976/0.150,
                1.897/0.637/0.0958/-0.9954/0.151,
                2.317/0.637/0.1196/-0.9928/0.152,
                2.737/0.637/0.1402/-0.9901/0.153,
                3.157/0.637/0.1577/-0.9875/0.154,
                3.577/0.637/0.1723/-0.9850/0.155,
                3.997/0.637/0.1846/-0.9828/0.156,
                0.217/1.057/0.0001/-1.0000/0.145,
                0.637/1.057/0.0014/-1.0000/0.145,
                1.057/1.057/0.0042/-1.0000/0.145,
                1.477/1.057/0.0070/-1.0000/0.145,
                1.897/1.057/0.0098/-1.0000/0.145,
                2.317/1.057/0.0124/-0.9999/0.145,
                2.737/1.057/0.0148/-0.9999/0.145,
                3.157/1.057/0.0169/-0.9999/0.145,
                3.577/1.057/0.0188/-0.9998/0.145,
                3.997/1.057/0.0205/-0.9998/0.145,
                0.217/1.477/0.0000/-1.0000/0.145,
                0.637/1.477/0.0004/-1.0000/0.145,
                1.057/1.477/0.0012/-1.0000/0.145,
                1.477/1.477/0.0021/-1.0000/0.145,
                1.897/1.477/0.0029/-1.0000/0.145,
                2.317/1.477/0.0037/-1.0000/0.145,
                2.737/1.477/0.0044/-1.0000/0.145,
                3.157/1.477/0.0050/-1.0000/0.145,
                3.577/1.477/0.0055/-1.0000/0.145,
                3.997/1.477/0.0059/-1.0000/0.145,
                0.217/1.897/0.0000/-1.0000/0.145,
                0.637/1.897/0.0001/-1.0000/0.145,
                1.057/1.897/0.0003/-1.0000/0.145,
                1.477/1.897/0.0005/-1.0000/0.145,
                1.897/1.897/0.0007/-1.0000/0.145,
                2.317/1.897/0.0009/-1.0000/0.145,
                2.737/1.897/0.0011/-1.0000/0.145,
                3.157/1.897/0.0012/-1.0000/0.145,
                3.577/1.897/0.0014/-1.0000/0.145,
                3.997/1.897/0.0015/-1.0000/0.145,
                0.217/2.317/0.0000/-1.0000/0.145,
                0.637/2.317/0.0001/-1.0000/0.145,
                1.057/2.317/0.0003/-1.0000/0.145,
                1.477/2.317/0.0005/-1.0000/0.145,
                1.897/2.317/0.0007/-1.0000/0.145,
                2.317/2.317/0.0009/-1.0000/0.145,
                2.737/2.317/0.0011/-1.0000/0.145,
                3.157/2.317/0.0012/-1.0000/0.145,
                3.577/2.317/0.0014/-1.0000/0.145,
                3.997/2.317/0.0015/-1.0000/0.145,
                0.217/2.737/0.0010/-1.0000/0.145,
                0.637/2.737/0.0101/-0.9999/0.145,
                1.057/2.737/0.0290/-0.9996/0.145,
                1.477/2.737/0.0480/-0.9988/0.145,
                1.897/2.737/0.0657/-0.9978/0.145,
                2.317/2.737/0.0813/-0.9967/0.146,
                2.737/2.737/0.0947/-0.9955/0.146,
                3.157/2.737/0.1060/-0.9944/0.146,
                3.577/2.737/0.1156/-0.9933/0.147,
                3.997/2.737/0.1236/-0.9923/0.147,
                0.217/3.157/0.0039/-1.0000/0.149,
                0.637/3.157/0.0392/-0.9992/0.149,
                1.057/3.157/0.1090/-0.9940/0.150,
                1.477/3.157/0.1753/-0.9845/0.151,
                1.897/3.157/0.2324/-0.9726/0.153,
                2.317/3.157/0.2797/-0.9601/0.154,
                2.737/3.157/0.3183/-0.9479/0.156,
                3.157/3.157/0.3498/-0.9368/0.157,
                3.577/3.157/0.3758/-0.9267/0.159,
                3.997/3.157/0.3974/-0.9177/0.160,
                0.217/3.577/0.1612/-0.9869/0.222,
                0.637/3.577/0.8243/-0.5662/0.245,
                1.057/3.577/0.9492/-0.3148/0.266,
                1.477/3.577/0.9808/-0.1948/0.285,
                1.897/3.577/0.9920/-0.1266/0.301,
                2.317/3.577/0.9965/-0.0840/0.314,
                2.737/3.577/0.9984/-0.0558/0.324,
                3.157/3.577/0.9993/-0.0363/0.332,
                3.577/3.577/0.9997/-0.0225/0.338,
                0.217/3.997/0.9003/0.4353/0.331,
                0.637/3.997/0.9979/0.0642/0.349,
                1.057/3.997/0.9999/-0.0165/0.360,
                1.477/3.997/0.9996/-0.0284/0.365,
                1.897/3.997/0.9994/-0.0346/0.368,
                2.317/3.997/0.9992/-0.0386/0.370,
                2.737/3.997/0.9991/-0.0416/0.372,
                3.157/3.997/0.9990/-0.0439/0.373,
                3.577/3.997/0.9990/-0.0458/0.374
            }
            {
                \draw[-{Stealth[length=1.8mm,width=0.9mm]}, line width=0.35pt]
                    (\x,\y) -- ++({1.5*\len*\ux},{1.5*\len*\uy});
            }
        \end{tikzpicture}
        \caption{Darcy flow field $\vec{b}^\mu$}
    \end{subfigure}
    \caption{Illustration of the catalytic filter benchmark problem.\label{fig: darcy model}}
\end{figure}
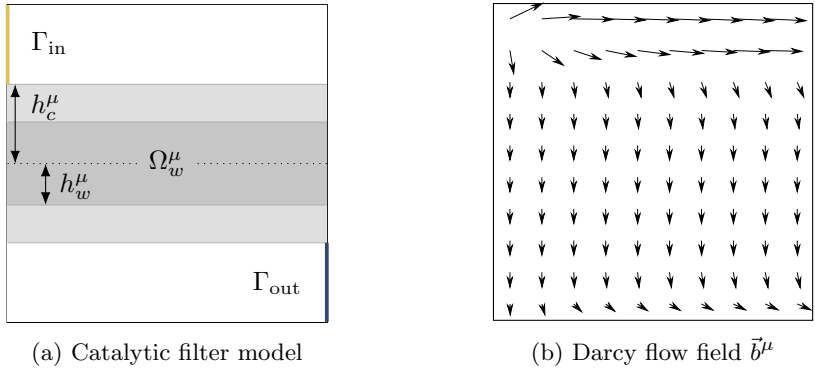

{\bf Adapted transport problem.}
Let $u^{\mu, \nu} = u^{\mu, \nu}(x,t)$ for $x, t \in \Omega \times [0, T]$ denote some transported scalar 
concentration through the catalyst filter.
In the implementation, the transport velocity (advection) is not prescribed analytically, but taken from the
previous Darcy flow problem solution $\vec{b}^\mu$ in equation (\ref{eq: darcy problem}).
With zero diffusion ($A \equiv 0$), the strong form is given by
\begin{equation} \label{eq: transport problem}
\begin{cases}
    \partial_t u^{\mu, \nu} + \nabla \cdot \big(\vec{b}^{\mu} u\big) + c^{\mu, \nu} u = 0,
    & \text{in } \Omega\times(0,T],\\[1mm]
    u^{\mu, \nu} = g_{\mathrm{in}},
    & \text{on } \Gamma_{\mathrm{in}}\times[0, T],\\
    u^{\mu, \nu}(\cdot,0) \equiv 0,
    & \text{on } \Omega - \Gamma_\mathrm{in} \times \{0\}.
\end{cases}
\end{equation}
The imposed inflow concentration profile is
\begin{equation*}
    g_{\mathrm{in}}(x)=
    \begin{cases}
        \frac12 \left(1+\cos \left(\pi\,\dfrac{|x_2-0.9|}{0.1}\right)\right), & |x_2-0.9|<0.1,\\[2mm]
        0, & \text{otherwise},
    \end{cases}
\end{equation*}
and is re-enforced each time step to keep inlet feed constant in time.
Moreover, we set the reaction coefficient $c^{\mu, \nu}$ to be piece-wise constant on the parametric
coating and washcoat regions. Specifically,
\begin{equation*}
    c^{\mu, \nu} :=
        \nu_1 \mathbb{1}_{\Omega_c^\mu} + \nu_2 \mathbb{1}_{\Omega_w^\mu},
\end{equation*}
where $\nu = (\nu_1, \nu_2) \in [0.1, 0.3] \times [0.3, 1]$.

\subsection{Training specifications}

For all models, we uniformly use the LeakyReLU instead of the ReLU to form the neural networks
contained in the models. This activation function is mathematically
equivalent to the ReLU as justified by the paper \cite{geist2020}. Hence, we overwrite the
definition of the ReLU with
\begin{equation*}
    \rho(x) := \begin{cases}
        x, \quad \text{if }x \geq 0 \\
        a x, \quad \text{if }x < 0.
    \end{cases}
\end{equation*}
We will put $a = 0.1$ as the constant in every instance of the activation unit.

{\bf Normalization.}
For the normalization, we use the min-max normalization in a generalized sense; for
each parameter set we define 
\begin{equation*}
    M_\text{max}^{i} = \max\limits_{j = 1, \dots, N_{s_1}} M_{ij}^\text{train}, \quad 
    M_\text{min}^{i} = \min\limits_{j = 1, \dots, N_{s_1}} M_{ij}^\text{train},
\end{equation*}
and its subsequent min-max normalization via
\begin{equation*} \label{eq: min-max normalization}
    M_{ij}^\text{train} \mapsto \frac{M_{ij}^\text{train} - M_\text{max}^i}{M_\text{max}^i - M_\text{min}^i},
    \quad \text{for } i = 1, \dots, p, \, j = 1, \dots, N_{s_1}.
\end{equation*}
For $N^\text{train}$ the analogous approach can be used. The solution batch is normalized in the
same way, after defining the minimum and maximum as
\begin{equation*}
    S_\text{max} = \max\limits_{i = 1, \dots, N_h} \max\limits_{j = 1, \dots, N_s N_t} S_{ij}^\text{train},
    \quad S_\text{min} = \min\limits_{i = 1, \dots, N_h} \min\limits_{j = 1, \dots, N_s N_t} S_{ij}^\text{train}.
\end{equation*}

{\bf Optimizer and Schedular.}
All offline training phases will be performed with the AdamW optimizer \cite{loshchilov2019} 
and the ReduceLROnPlateau schedular \cite{alkababji2022} 
to achieve adaptive learning rates. Moreover, an early stopping criterion is imposed upon the
validation loss.

{\bf Relative Error.}
For evaluation, we rely on the following discretized version of the relative error $\Ecal_\text{R} \in [0, \infty)$
over a test sample $\{(\mu_i, \nu_i), u_h^{\mu_i, \nu_i}\}_{i=1}^{N_\text{test}}$ with time-trajectories
in $\Tcal$ defined as
\begin{equation*}
    \Ecal_\text{R} = \frac{1}{N_\text{test}} \sum_{i=1}^{N_\text{test}} \left(
        \frac{
            \sqrt{\frac{1}{N_t} \sum_{k=1}^{N_t} \norm{u_h^{\mu_i, \nu_i, k \Delta t} - 
            \hat{u}_h^{\mu_i, \nu_i, k \Delta t}}^2}
        }{
            \sqrt{\frac{1}{N_t} \sum_{k=1}^{N_t} \norm{u_h^{\mu_i, \nu_i, k \Delta t}}^2}
        }
        \right),
\end{equation*}
where we set $\hat{u}_h^{\mu, \nu, t} \in \R^{N_h}$ to denote some solution prediction for parameters $(\mu, \nu, t)
\in \Theta \times \Theta' \times [0, T]$. 

{\bf Hardware Specification and Source Code.}
Training Data was achieved in \texttt{DUNE}~\cite{dune}, using the DUNE-PDELab~\cite{dune-pdelab} discretization toolbox. 
For the implementation and data inheritance to \texttt{PyTorch}, 
a combination of \texttt{pyMOR}~\cite{pymor2016} and \texttt{Numpy}~\cite{numpy}
was used.
Model training is performed using the Python library \texttt{PyTorch}~\cite{pytorch2019} on
an Nvidia Quadro RTX 6000. 
All source code required to reproduce the presented results, accompanied by a concise 
\texttt{README.md}, is openly available at 
\url{https://github.com/dawid-kotowski/doddlrom}.

\begin{flushleft}
    \refstepcounter{algorithm}
    \textbf{Algorithm \thealgorithm.} DOD-DL-ROM training algorithm (offline)
    \label{alg: dod-dl-rom training}
    
    \begin{algorithmic}[1]
    \Require Parameter matrices $M \in \R^{p\times N_{s_1}}, N \in \R^{q \times N_{s_2}}$, 
    snapshot matrix $S \in \R{N_h \times N_t N_s}$, 
    training-validation splitting fraction $\alpha$, starting learning rate $\eta^0$, 
    batch size $N_b$, batch number $N_{nb}$, maximum number of epochs $N_{\text{epochs}}$, dimensions $n, N', N_A$
    \Ensure Optimal model parameters $\theta_\text{DOD-DL-ROM}^* = (\theta_\text{DOD}^*, \theta_\text{D}^*,
    \theta_\text{DF}^*)$
    \State Compute pre-reduction POD basis matrix $\A \in \R^{N_h \times N_A}$ via Definition
    \ref{def: POD}
    \State Randomly shuffle $M$ and $S$
    \State Split data in $M = [M^{\text{train}}, M^{\text{val}}], N=[N^\text{train}, N^\text{val}]$ and 
    $S = [S^{\text{train}}, S^{\text{val}}]$ according to $\alpha$
    \State $S_{N_A}^\text{train} \gets \A^T \G S^\text{train}$ and $S_{N_A}^\text{val} \gets \A^T \G S^\text{val}$
    \State Normalize data in $M, N$ and $S_{N_A} = [S_{N_A}^{\text{train}}, S_{N_A}^{\text{val}}]$
    \State Train inner DOD module $\Phi_{\tilde{\V}}(\theta^*_\text{DOD})$ (see \cite{franco2024})
    \State Initialize DL-ROM weights $(\theta^0_\text{D}, \theta^0_\text{E}, \theta^0_\text{DF}) 
    \gets \text{RandomInit}(\theta^0_\text{D}, \theta^0_\text{E}, \theta^0_\text{DF})$
    \State $n_e = 0$
    \While{(¬early-stopping and $n_e \leq N_{\text{epochs}}$)}
    \For{$k = 1:N_{nb}$}
        \State Sample a batch $(M^{\text{batch}}, N^{\text{batch}}, S_{N_A}^{\text{batch}}) \subset 
        (M^{\text{train}}, N^{\text{train}}, S_{N_A}^{\text{train}})$
        \For{$t =1:N_t$}
            \State $S_{N'}^{\text{batch}, t} \gets \tilde{\V}^T_{M^\text{batch}, t} \cdot 
            S_{N_A}^{\text{batch}, t} \in \R^{N_b \times N'}$
            \State $S_{N'}^{\text{batch}, t} \gets 
            \text{reshape}(S_{N'}^{\text{batch}, t}) \in 
            \R^{N_b \times \sqrt{N'} \times \sqrt{N'}}$
            \State $z^{(n_e N_{\text{nb}}+k), t} \gets
            R_\rho(\Psi_n'(\theta_\text{E}^{n_e N_{\text{nb}}+k}))(S_{N'}^{\text{batch}, t})$
            \State $\hat S_{N'}^{\text{batch}, t} \gets
            R_\rho(\Psi_{N'}(\theta_\text{D}^{n_e N_{\text{nb}}+k}) \odot \Phi_n(\theta_\text{DF}^{n_e N_{\text{nb}}+k}))(M^{\text{batch}}, N^\text{batch}, t)$
            \State $\hat S_{N'}^{\text{batch}, t} \gets
            \text{reshape}(\hat S_{N'}^{\text{batch}, t}) 
            \in \R^{N_b \times {N'}}$
        \EndFor
        \State Collect $\Lcal$
        from Subsection \ref{subsec: dod-dl-rom} on $(M^{\text{batch}}, N^\text{batch}, S_{N_A}^{\text{batch}})$ 
        \State $\theta^{n_e N_{\text{nb}}+k+1} = 
        \text{AdamW}(\eta^0, \nabla_\theta \mathcal{J}, \theta^{n_e N_{\text{nb}}+k})$
        \State $\eta^{n_e + 1} \gets \text{ReduceOnPlateau}(\Lcal_\text{DOD-DL-ROM})$
    \EndFor
    \State Repeat 9--24 on $(M^{\text{val}}, N^\text{val}, S_{N_A}^{\text{val}})$ 
    with the updated weights $\theta^{n_e N_{\text{nb}}+k+1}$
    \State Accumulate loss $\Lcal_\text{DOD-DL-ROM}$
        from Subsection \ref{subsec: dod-dl-rom} on $(M^{\text{val}}, N^\text{val}, S_{N_A}^{\text{val}})$
    \State $n_e = n_e + 1$
    \EndWhile
    \end{algorithmic}
\end{flushleft}

\subsection{Benchmark results}

In order to support the theoretical findings, we now present the model analytics of the parametric catalyst filter
problem introduced in the beginning of the section. 

\begin{figure}[b]
    \centering
    \begin{minipage}[t]{0.55\linewidth}
        \vspace{0pt}
        \centering
        \includegraphics[width=.9\linewidth]{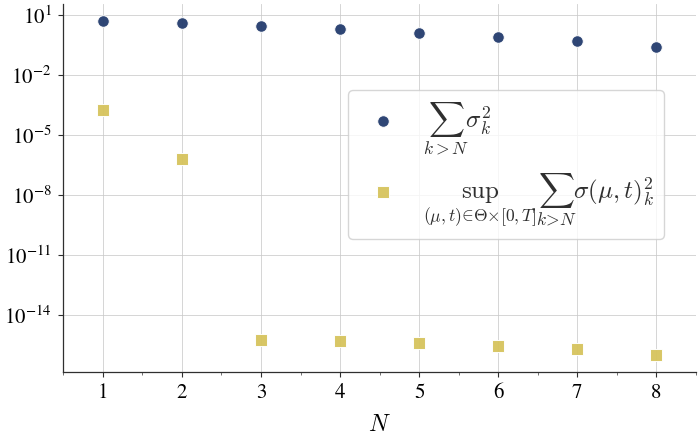}
        \caption{KnW decay}
        \label{fig: knw decay}
    \end{minipage}\hfill
    \begin{minipage}[t]{0.40\linewidth}
        \vspace{35pt}
        \centering
        \small
        \setlength{\tabcolsep}{4pt}
        \begin{tabular}{ccccc}
            \toprule
            $n$ & $N'$ & $N$ & $N_A$ & $N_h$ \\
            \midrule
            $2$ & $2$ & $8$ & $10$ & $3600$ \\
            \bottomrule
        \end{tabular}
        \captionof{table}{Reduction numbers}
        \label{tab: reduction numbers}
    \end{minipage}
\end{figure}
Figure \ref{fig: knw decay} shows a somewhat motivating
development of the Kolmogorov $N$-width decay. Namely, the global decay is significantly slower than its $(\mu, t)$-fixed
counterpart. However, the actual Assumption \ref{assumption KnW decay} is not satisfied, since the decay appears to
be still slightly exponential for higher reduction number. This might be due to the strongly confined parameter numbers.
In an industry relevant setting this confinement is justified, as most of the parametric influence can 
be bounded well.

For the sake of visual confirmation, Figure \ref{fig: visual} shows the side by side plots of three trained models.
We will generally use the dimensional numbers for the reduction as shown in Table \ref{tab: reduction numbers}.
We stress, that for higher precision needs the numbers can be scaled while retaining the general ordering, but the
singular value decay suggests, that this is already sufficient to show promising results.
\begin{figure}[t]
    \centering
    \begin{subfigure}{0.99\linewidth}
        \centering
        \includegraphics[width=0.99\linewidth]{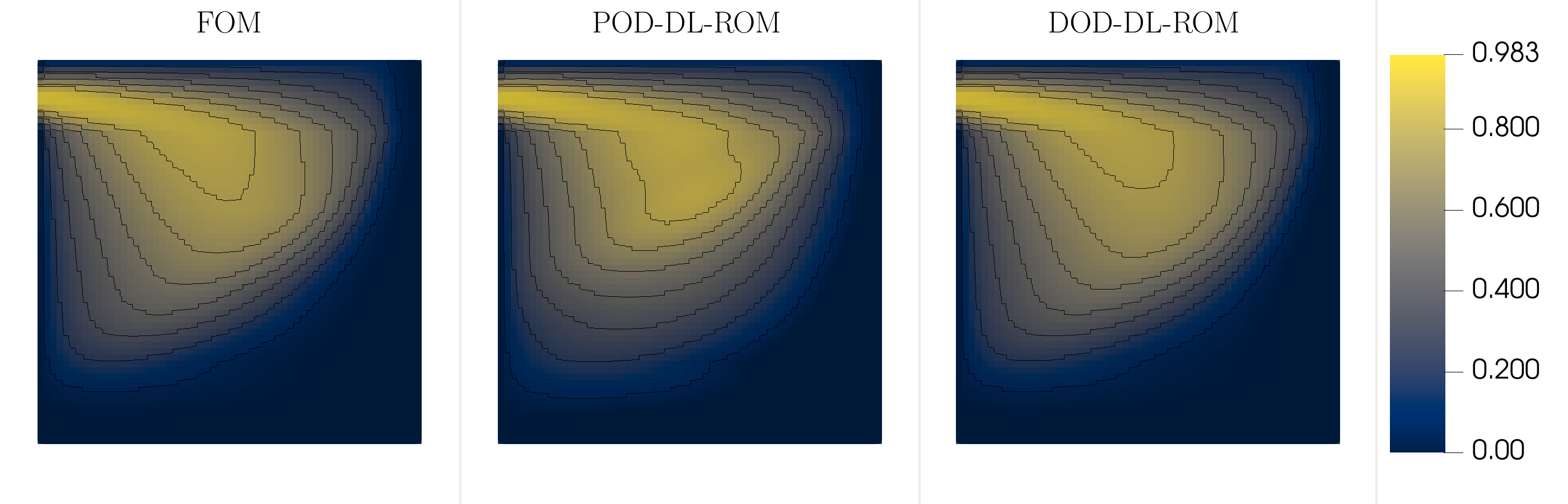}
        \caption{$\mu = (1/30, -\pi/12), \nu = (1/30, 8/15)$ and $t = 0.6$}    
    \end{subfigure}
    \begin{subfigure}{0.99\linewidth}
        \centering
        \includegraphics[width=0.99\linewidth]{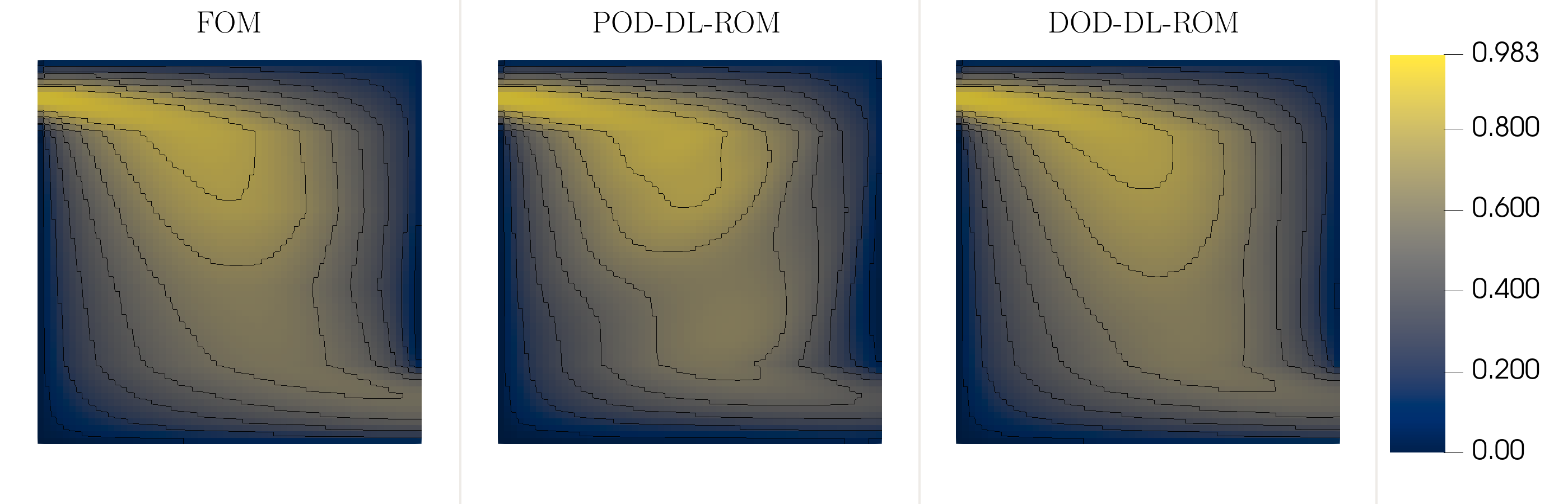}
        \caption{$\mu = (1/30, -\pi/12), \nu = (1/30, 8/15)$ and $t = 1.2$}   
    \end{subfigure}
    \caption{Time-dependent solutions}
    \label{fig: visual}
\end{figure}

Here, we have chosen baseline models with around $~10^5$ number of active weights, to show high precision for
aggressive reduction numbers. Indeed, while this is only one example, the
difference of adaptive versus non-adaptive reduced basis is apparent in the moving front. Analysing 
other snapshots shows, that the POD does have high precision for the most prominent long term features,
but shows caveats for quick moving fronts. While the motivation for
the architechtural design is given in \textcite{fresca2022}, the precise choices can be found in the code repository.
\begin{table}[htbp]
\centering
\begin{tabular}{llrrrrr}
\toprule
Level & Model & $\omega$ & $\omega/\omega_{\text{POD}}$ 
& $t_{\mathrm{fwd}}$ [ms] & $t_{\mathrm{fwd}}/t_{\mathrm{fwd}}^{\mathrm{POD}}$ 
& Speedup \\
\midrule
\multirow{3}{*}{Low}
  & POD-DL-ROM &   308 & 1.00 & 32.17 & 1.00 & $684\times$ \\
  & DOD-DL-ROM &   573 & 1.86 & 33.33 & 1.04 & $660\times$ \\
  & DOD+DFNN   &   444 & 1.44 & 30.86 & 0.96 & $713\times$ \\
\addlinespace[0.3em]

\multirow{3}{*}{Medium}
  & POD-DL-ROM &  5429 & 1.00 & 32.45 & 1.00 & $678\times$ \\
  & DOD-DL-ROM & 10530 & 1.94 & 36.29 & 1.12 & $606\times$ \\
  & DOD+DFNN   &  5753 & 1.06 & 30.87 & 0.95 & $713\times$ \\
\addlinespace[0.3em]

\multirow{3}{*}{High}
  & POD-DL-ROM &  34137 & 1.00 & 34.52 & 1.00 & $637\times$ \\
  & DOD-DL-ROM & 106723 & 3.13 & 36.29 & 1.05 & $606\times$ \\
  & DOD+DFNN   &  67313 & 1.97 & 33.28 & 0.96 & $661\times$ \\
\bottomrule
\end{tabular}
\caption{Active Weight compared against absolute forward timings $t_\mathrm{fwd}$
and speedup against FOM baseline}
\label{tab: weights timing scaling}
\end{table}

Furthermore, we present Table \ref{tab: weights timing scaling} to show a clear relationship between the
number of active weights and scaling of speedup. The models are separated by complexity and an additional
comparison of relative speedup and number of active weights with regard to the POD-DL-ROM is given. Note,
that indeed a more drastic difference of complexity translates with less magnitude to the timings.

\begin{figure}[htbp]
    \centering
    \includegraphics[height=0.45\linewidth, keepaspectratio]{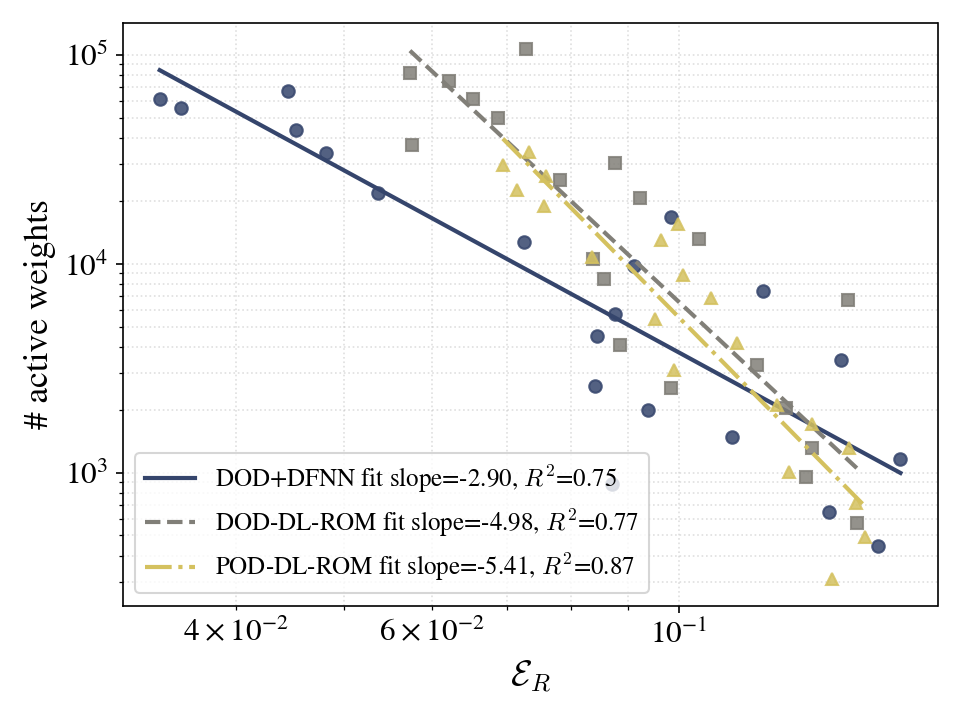}
    \caption{Number of active weights for ROMs in relation to $\Ecal_\text{R}$}
    \label{fig: weight comparison}
\end{figure}
The main result of this paper, however, is the comparative analysis of both DOD- and POD-based
architectures with regard to the
number of weights for a fixed relative error $\Ecal_\text{R}$. For this, one can simply plot the 
growing complexity (in exponential terms) with regard to the development of the relative error
for each model. As can be seen in Figure 
\ref{fig: weight comparison}, the DOD-based methods tend to develop less cost in terms of complexity,
and thus slower speedup, than the POD-based alternative, for growing demand in precision. Notice, that
we can even make out a crossing point, where the initially better performance of the POD-DL-ROM gets 
outcompeted by the improved adaptivity of the DOD. 
The existence of such a point makes sense, as the main bottleneck for the feasibility of
the DOD-based methods is given by 
\begin{equation*}
    \frac{n}{s-1} > \frac{\log(N_h)}{\log(\varepsilon^{-1})}.
\end{equation*}
See Proposition \ref{prop: regularity discussion for POD vs DOD}
for any fixed relative error tolerance $\varepsilon > 0$. Obviously, this is satisfied most likely for very
small errors and containable $N_h$. However, since $s$ is not known,
one can not be sure how small the right-hand side needs to be. 
Note, that the proposition uses $N_h = N_A$ and thus the
actual bottleneck might be a bit less tight, because the complexity bound is only imposed upon the
neural network parts of the model. If the earlier results of Theorem \ref{theo: PAC bound for DOD-DL-ROM}
would be expanded to situations of $N_A < N_h$, one might potentially achieve a 
theoretic frame for that situation aswell.

\section{Conclusion}
In this article, we have introduced a novel adaptive deep learning-based reduced order model, 
denoted DOD-DL-ROM, for hybrid-type parabolic PDEs with distinct decay behaviors of the 
Kolmogorov N-width. By combining the strengths of deep learning-based surrogates with 
projection-based model order reduction, our approach provides a robust, real-time, and 
physically consistent solution for complex industrial flow processes. 
We have rigorously analyzed the error and complexity of the proposed method, 
providing a quantitative dependence of the online performance on the regularity of an 
associated optimal map. Our numerical experiments have demonstrated the superiority of 
DOD-DL-ROM over traditional POD-based ROMs for a catalyst filter benchmark problem. 
The results of this work have significant implications for the development of efficient 
and accurate surrogate models for complex systems, enabling accelerated product development 
and digital transformation in industries reliant on physical simulations. 
Future research directions include the extension of the DOD-DL-ROM approach to 
other types of PDEs and the exploration of its potential applications in 
fields such as fluid dynamics or heat transfer. Another point of interest might be 
the analysis of the regularity of the optimal DOD map as given in 
Assumption \ref{assumption additional lipschitz}.

\section*{Declarations}

\subsection*{Availability of data and materials}
The source code used to generate the datasets supporting the conclusions of this article is available in the doddlrom repositories, \url{https://github.com/dawid-kotowski/doddlrom.git} and \url{https://github.com/dawid-kotowski/darcyflow.git}.

\subsection*{Competing interest}
The authors declare that they have no competing interests.

\subsection*{Funding}
This study was funded by the Deutsche Forschungsgemeinschaft (DFG, German Research Foundation) under Germany's Excellence Strategy EXC 2044/2 –390685587, Mathematics Münster: Dynamics–Geometry–Structure.

\subsection*{Author's contributions}
DK has introduced the definition of the newly discussed methods, developed the main theorems of this article
and provided their proofs.
MO has revised and reworked the Introduction, the Conclusion and any placement in current literature.
Further, MO has reviewed all other contributions, mathematical rigor and advised in crucial points
of the numerical experiments.

\subsection*{Acknowledgements}
Not applicable.


\bibliography{bibliography}

@misc{franco2024,
  author = {Franco, Nicola Rares and Manzoni, Andrea and Zunino, Paolo and Hesthaven, Jan S.},
  title = {Deep orthogonal decomposition: a continuously adaptive data-driven approach to model order reduction},
  year = {2024},
  howpublished = {Preprint, arXiv:2404.18841 [math.NA] (2024)},
  url = {https://arxiv.org/abs/2404.18841},
  note = {arXiv:2404.18841}
}

@misc{brivio2023,
      title={Error estimates for POD-DL-ROMs: a deep learning framework for reduced order modeling of nonlinear parametrized PDEs enhanced by proper orthogonal decomposition}, 
      author={Simone Brivio and Stefania Fresca and Nicola Rares Franco and Andrea Manzoni},
      year={2023},
      eprint={2305.04680},
      archivePrefix={arXiv},
      primaryClass={math.NA},
      url={https://arxiv.org/abs/2305.04680}, 
}

@article{fresca2022,
   title={POD-DL-ROM: Enhancing deep learning-based reduced order models for nonlinear parametrized PDEs by proper orthogonal decomposition},
   volume={388},
   ISSN={0045-7825},
   url={http://dx.doi.org/10.1016/j.cma.2021.114181},
   DOI={10.1016/j.cma.2021.114181},
   journal={Computer Methods in Applied Mechanics and Engineering},
   publisher={Elsevier BV},
   author={Fresca, Stefania and Manzoni, Andrea},
   year={2022},
   month=jan, pages={114181} }

@misc{fresca2020,
      title={A comprehensive deep learning-based approach to reduced order modeling of nonlinear time-dependent parametrized PDEs}, 
      author={Stefania Fresca and Luca Dede and Andrea Manzoni},
      year={2020},
      eprint={2001.04001},
      archivePrefix={arXiv},
      primaryClass={math.NA},
      url={https://arxiv.org/abs/2001.04001}, 
}

@book{quarteroni2015rb,
  title     = {Reduced Basis Methods for Partial Differential Equations},
  author    = {Quarteroni, Alfio and Manzoni, Andrea and Negri, Federico},
  year      = {2015},
  publisher = {Springer},
  address   = {Cham},
  doi       = {10.1007/978-3-319-15431-2}
}

@misc{gühring2019nn,
      title={Error bounds for approximations with deep ReLU neural networks in $W^{s,p}$ norms}, 
      author={Ingo Gühring and Gitta Kutyniok and Philipp Petersen},
      year={2019},
      eprint={1902.07896},
      archivePrefix={arXiv},
      primaryClass={math.FA},
      url={https://arxiv.org/abs/1902.07896}, 
}

@article{yarotsky2017nn,
title = {Error bounds for approximations with deep ReLU networks},
journal = {Neural Networks},
volume = {94},
pages = {103-114},
year = {2017},
issn = {0893-6080},
doi = {https://doi.org/10.1016/j.neunet.2017.07.002},
url = {https://www.sciencedirect.com/science/article/pii/S0893608017301545},
author = {Dmitry Yarotsky}
}

@article{hornik1991nn,
  title = {Approximation capabilities of multilayer feedforward networks},
  journal = {Neural Networks},
  volume = {4},
  number = {2},
  pages = {251-257},
  year = {1991},
  issn = {0893-6080},
  doi = {https://doi.org/10.1016/0893-6080(91)90009-T},
  url = {https://www.sciencedirect.com/science/article/pii/089360809190009T},
  author = {Kurt Hornik}
}

@inproceedings{lumley1967,
  author = {Lumley, J. L.},
  title = {The structure of inhomogeneous turbulent flows},
  booktitle = {Atmospheric Turbulence and Radio Wave Propagation},
  editor = {Yaglom, A. M. and Tatarsky, V. I.},
  year = {1967},
  publisher = {Nauka},
  address ={Moscow},
  pages = {166--178}
}

@article{sirovich1987,
  author = {Sirovich, L.},
  title = {Turbulence and the dynamics of coherent structures. Part I: Coherent structures},
  journal = {Journal of Fluid Mechanics},
  volume = {175},
  pages = {419--427},
  year = {1987},
  doi = {10.1017/S0022112087000199}
}

@article{rowley2009,
  author = {Rowley, C. W. and Mezi{\'c}, I. and Bagheri, S. and Schlatter, P. and Henningson, D. S.},
  title = {Spectral analysis of nonlinear flows},
  journal = {Journal of Fluid Mechanics},
  volume = {641},
  pages = {115--127},
  year = {2009},
  doi = {10.1017/S0022112009992059}
}

@inproceedings{schmid2010,
  author = {Schmid, P. J.},
  title = {Dynamic mode decomposition of numerical and experimental data},
  booktitle = {Proceedings of the 61st Annual Meeting of the APS Division of Fluid Dynamics},
  year = {2010}
}

@article{barrault2004,
  author = {Barrault, M. and Maday, Y. and Nguyen, N. C. and Patera, A. T.},
  title = {An 'empirical interpolation' method: application to efficient reduced-basis discretization of partial differential equations},
  journal = {Comptes Rendus Mathematique},
  volume = {339},
  number = {9},
  pages = {667--672},
  year = {2004},
  doi = {10.1016/j.crma.2004.08.006}
}

@article{peherstorfer2018,
author = {Peherstorfer, Benjamin and Willcox, Karen and Gunzburger, Max},
year = {2018},
month = {01},
pages = {550-591},
title = {Survey of Multifidelity Methods in Uncertainty Propagation, Inference, and Optimization},
volume = {60},
journal = {SIAM Review},
doi = {10.1137/16M1082469}
}

@book{holmes1996,
  author = {Holmes, P. and Lumley, J. L. and Berkooz, G.},
  title = {Turbulence, Coherent Structures, Dynamical Systems and Symmetry},
  publisher = {Cambridge University Press},
  address= {Cambridge},
  year = {1996}
}

@article{kunischvolkwein2001,
  author = {Kunisch, K. and Volkwein, S.},
  title = {Galerkin proper orthogonal decomposition methods for a general equation in fluid dynamics},
  journal = {SIAM Journal on Numerical Analysis},
  volume = {40},
  number = {2},
  pages = {492--515},
  year = {2001},
  doi = {10.1137/S0036142900382612}
}

@article{chaturantabut2010,
  author = {Chaturantabut, S. and Sorensen, D. C.},
  title = {Nonlinear model reduction via discrete empirical interpolation},
  journal = {SIAM Journal on Scientific Computing},
  volume = {32},
  number = {5},
  pages = {2737--2764},
  year = {2010},
  doi = {10.1137/090766498}
}

@book{golubvanloan2013,
  author    = {Golub, Gene H. and Van Loan, Charles F.},
  title     = {Matrix Computations},
  publisher = {Johns Hopkins University Press},
  year      = {2013},
  doi       = {10.56021/9781421407944}
}

@article{carlberg2013,
   title={The GNAT method for nonlinear model reduction: Effective implementation and application to computational fluid dynamics and turbulent flows},
   volume={242},
   ISSN={0021-9991},
   url={http://dx.doi.org/10.1016/j.jcp.2013.02.028},
   DOI={10.1016/j.jcp.2013.02.028},
   journal={Journal of Computational Physics},
   publisher={Elsevier BV},
   author={Carlberg, Kevin and Farhat, Charbel and Cortial, Julien and Amsallem, David},
   year={2013},
   month=jun, pages={623–647} }

@article{carlberg2011,
author = {Carlberg, Kevin and Bou-Mosleh, Charbel and Farhat, Charbel},
year = {2011},
month = {04},
pages = {155 - 181},
title = {Efficient non-linear model reduction via a least-squares Petrov-Galerkin projection and compressive tensor approximations},
volume = {86},
journal = {International Journal for Numerical Methods in Engineering},
doi = {10.1002/nme.3050}
}

@inproceedings{loshchilov2019,
  title={Decoupled Weight Decay Regularization},
  author={Loshchilov, Ilya and Hutter, Frank},
  booktitle={International Conference on Learning Representations (ICLR)},
  year={2019},
}

@article{pymor2016,
  author  = {René Milk and Stephan Rave and Jens Saak},
  title   = {pyMOR -- Generic Algorithms and Interfaces for Model Order Reduction},
  journal = {SIAM Journal on Scientific Computing},
  volume  = {38},
  number  = {5},
  pages   = {S194--S216},
  year    = {2016},
  doi     = {10.1137/15M1026614}
}

@misc{pytorch2019,
      title={PyTorch: An Imperative Style, High-Performance Deep Learning Library}, 
      author={Adam Paszke and Sam Gross and Francisco Massa and Adam Lerer and James Bradbury and Gregory Chanan and Trevor Killeen and Zeming Lin and Natalia Gimelshein and Luca Antiga and Alban Desmaison and Andreas Köpf and Edward Yang and Zach DeVito and Martin Raison and Alykhan Tejani and Sasank Chilamkurthy and Benoit Steiner and Lu Fang and Junjie Bai and Soumith Chintala},
      year={2019},
      eprint={1912.01703},
      archivePrefix={arXiv},
      primaryClass={cs.LG},
      url={https://arxiv.org/abs/1912.01703}, 
}

@misc{geist2020,
      title={Numerical Solution of the Parametric Diffusion Equation by Deep Neural Networks}, 
      author={Moritz Geist and Philipp Petersen and Mones Raslan and Reinhold Schneider and Gitta Kutyniok},
      year={2020},
      eprint={2004.12131},
      archivePrefix={arXiv},
      primaryClass={math.NA},
      url={https://arxiv.org/abs/2004.12131}, 
}

@article{kuratowski1965,
  author    = {K. Kuratowski and C. Ryll-Nardzewski},
  title     = {A general theorem on selectors},
  journal   = {Bull. Acad. Polon. Sci. Sér. Sci. Math. Astronom. Phys.},
  volume    = {13},
  year      = {1965},
  pages     = {397--403}
}

@incollection{newey1994,
title = {Chapter 36 Large sample estimation and hypothesis testing},
series = {Handbook of Econometrics},
publisher = {Elsevier},
volume = {4},
pages = {2111-2245},
year = {1994},
issn = {1573-4412},
doi = {https://doi.org/10.1016/S1573-4412(05)80005-4},
author = {Whitney K. Newey and Daniel McFadden}
}

@article{franco2023,
  author       = {Nicola R. Franco and Andrea Manzoni and Paolo Zunino},
  title        = {A Deep Learning Approach to Reduced Order Modelling of Parameter Dependent Partial Differential Equations},
  journaltitle = {Mathematics of Computation},
  year         = {2023},
  volume       = {92},
  pages        = {483--524},
  doi          = {10.1090/mcom/3781},
  url          = {https://doi.org/10.1090/mcom/3781}
}

@article{franco2023approx,
title = {Approximation bounds for convolutional neural networks in operator learning},
journal = {Neural Networks},
volume = {161},
pages = {129-141},
year = {2023},
issn = {0893-6080},
doi = {https://doi.org/10.1016/j.neunet.2023.01.029},
url = {https://www.sciencedirect.com/science/article/pii/S0893608023000412},
author = {Nicola Rares Franco and Stefania Fresca and Andrea Manzoni and Paolo Zunino}
}

@article{bhattacharya2021,
     author = {Kaushik Bhattacharya and Bamdad Hosseini and Nikola B. Kovachki and Andrew M. Stuart},
     title = {Model {Reduction} {And} {Neural} {Networks} {For} {Parametric} {PDEs}},
     journal = {The SMAI Journal of computational mathematics},
     pages = {121--157},
     year = {2021},
     publisher = {Soci\'et\'e de Math\'ematiques Appliqu\'ees et Industrielles},
     volume = {7},
     doi = {10.5802/smai-jcm.74},
     language = {en},
     url = {https://smai-jcm.centre-mersenne.org/articles/10.5802/smai-jcm.74/}
}

@misc{marwah2021,
      title={Parametric Complexity Bounds for Approximating PDEs with Neural Networks}, 
      author={Tanya Marwah and Zachary C. Lipton and Andrej Risteski},
      year={2021},
      eprint={2103.02138},
      archivePrefix={arXiv},
      primaryClass={cs.LG},
      url={https://arxiv.org/abs/2103.02138}, 
}

@misc{dune,
      title={The DUNE Framework: Basic Concepts and Recent Developments}, 
      author={Peter Bastian and Markus Blatt and Andreas Dedner and Nils-Arne Dreier and Christian Engwer and René Fritze and Carsten Gräser and Christoph Grüninger and Dominic Kempf and Robert Klöfkorn and Mario Ohlberger and Oliver Sander},
      year={2020},
      eprint={1909.13672},
      archivePrefix={arXiv},
      primaryClass={cs.MS},
      url={https://arxiv.org/abs/1909.13672}, 
}

@article{dune-pdelab,
  author = {Bastian, Peter and Heimann, Felix and Marnach, Sven},
  year = {2010},
  month = {01},
  pages = {},
  title = {Generic implementation of finite element methods in the Distributed and Unified Numerics Environment (DUNE)},
  volume = {2},
  journal = {Kybernetika}
}

@Article{numpy,
 title         = {Array programming with {NumPy}},
 author        = {Charles R. Harris and K. Jarrod Millman and St{\'{e}}fan J.
                 van der Walt and Ralf Gommers and Pauli Virtanen and David
                 Cournapeau and Eric Wieser and Julian Taylor and Sebastian
                 Berg and Nathaniel J. Smith and Robert Kern and Matti Picus
                 and Stephan Hoyer and Marten H. van Kerkwijk and Matthew
                 Brett and Allan Haldane and Jaime Fern{\'{a}}ndez del
                 R{\'{i}}o and Mark Wiebe and Pearu Peterson and Pierre
                 G{\'{e}}rard-Marchant and Kevin Sheppard and Tyler Reddy and
                 Warren Weckesser and Hameer Abbasi and Christoph Gohlke and
                 Travis E. Oliphant},
 year          = {2020},
 month         = sep,
 journal       = {Nature},
 volume        = {585},
 number        = {7825},
 pages         = {357--362},
 doi           = {10.1038/s41586-020-2649-2},
 publisher     = {Springer Science and Business Media {LLC}},
 url           = {https://doi.org/10.1038/s41586-020-2649-2}
}

@misc{alkababji2022,
      title={Scheduling Techniques for Liver Segmentation: ReduceLRonPlateau Vs OneCycleLR}, 
      author={Ayman Al-Kababji and Faycal Bensaali and Sarada Prasad Dakua},
      year={2022},
      eprint={2202.06373},
      archivePrefix={arXiv},
      primaryClass={cs.CV},
      url={https://arxiv.org/abs/2202.06373}, 
}

@Article{HO2008,
  author        = {Haasdonk, B. and Ohlberger, M.},
  title         = {Reduced basis method for finite volume approximations of parametrized linear evolution equations},
  doi           = {10.1051/m2an:2008001},
  number        = {2},
  pages         = {277-302},
  volume        = {42},
  document_type = {Article},
  journal       = {Mathematical Modelling and Numerical Analysis},
  source        = {Scopus},
  year          = {2008},
}

@article{DHO12,
 author               = {Drohmann, Martin and Haasdonk, Bernard and Ohlberger, Mario},
 doi                  = {10.1137/10081157X},
 fjournal             = {SIAM Journal on Scientific Computing},
 issn                 = {1064-8275},
 journal              = {SIAM J. Sci. Comput.},
 mrclass              = {65M08 (35K90)},
 mrnumber             = {2914310},
 mrreviewer           = {Anthony John Roberts},
 number               = {2},
 pages                = {A937--A969},
 title                = {Reduced basis approximation for nonlinear parametrized evolution equations based on empirical operator interpolation},
 url                  = {https://doi.org/10.1137/10081157X},
 volume               = {34},
 year                 = {2012},
 }

@article {haasdonk2022new,
    AUTHOR = {Haasdonk, Bernard and Kleikamp, Hendrik and Ohlberger, Mario
              and Schindler, Felix and Wenzel, Tizian},
     TITLE = "{A new certified hierarchical and adaptive {RB}-{ML}-{ROM}
              surrogate model for parametrized {PDE}s}",
   JOURNAL = {SIAM J. Sci. Comput.},
  FJOURNAL = {SIAM Journal on Scientific Computing},
    VOLUME = {45},
      YEAR = {2023},
    NUMBER = {3},
     PAGES = {A1039--A1065},
      ISSN = {1064-8275},
   MRCLASS = {65N30 (65M60 68T07)},
  MRNUMBER = {4586596},
       DOI = {10.1137/22M1493318},
}

@article{klein2025multifidelity,
 author               = {Benedikt Klein and Mario Ohlberger},
 doi                  = {10.1007/s10444-026-10296-6},
 fjournal             = {Advances in Computational Mathematics},
 issn                 = {1019-7168},
 journal              = {Adv. Comput. Math.},
 number               = {19},
 pages                = {1--36},
 title                = {Multi-fidelity Learning of Reduced Order Models for Parabolic PDE Constrained Optimization},
 url                  = {https://doi.org/110.1007/s10444-026-10296-6},
 volume               = {52},
 year                 = {2026},
 }

@Book{MR3672144,
  title     = {Model reduction and approximation},
  doi       = {10.1137/1.9781611974829},
  editor    = {Benner, Peter and Cohen, Albert and Ohlberger, Mario and Willcox, Karen},
  isbn      = {978-1-611974-81-2},
  note      = {Theory and algorithms},
  pages     = {xx+412},
  publisher = {SIAM},
  address   = {Philadelphia, PA, USA},
  series    = {Computational Science \& Engineering},
  volume    = {15},
  mrclass   = {65-06},
  mrnumber  = {3672144},
  year      = {2017},
}

@book {MR3701994,
     TITLE = {Model reduction of parametrized systems},
    SERIES = {MS\&A. Modeling, Simulation and Applications},
    VOLUME = {17},
    EDITOR = {Benner, Peter and Ohlberger, Mario and Patera, Anthony and
              Rozza, Gianluigi and Urban, Karsten},
      NOTE = {Selected papers from the 3rd MoRePaS Conference held at the
              International School for Advanced Studies (SISSA), Trieste,
              October 13--16, 2015},
 PUBLISHER = {Springer},
      year = {2017},
   ADDRESS = {Cham, Germany},
     PAGES = {xii+504},
      ISBN = {978-3-319-58785-1; 978-3-319-58786-8},
   MRCLASS = {65-06 (93-06)},
  MRNUMBER = {3701994},
       DOI = {10.1007/978-3-319-58786-8},
}

@article{OhlRav16,
 author               = {M. Ohlberger and S. Rave},
 journal              = {Proceedings of the Conference Algoritmy},
 pages                = {1--12},
 title                = {Reduced Basis Methods: Success, Limitations and Future Challenges},
 url                  = {http://www.iam.fmph.uniba.sk/amuc/ojs/index.php/algoritmy/article/view/389},
 year                 = {2016},
 }

@ARTICLE{Fokina202567,
	author = {Fokina, Daria and Toktaliev, Pavel and Herkert, Robin and Wenzel, Tizian and Haasdonk, Bernard and Iliev, Oleg},
	title = {Machine Learning Methods Based Prediction of Breakthrough Curves in Reactive Porous Media from Peclet and Damköhler Numbers},
	year = {2025},
	journal = {Studies in Computational Intelligence},
	volume = {1219 SCI},
	pages = {67 – 84},
	doi = {10.1007/978-3-031-96311-7_7},
}

@ARTICLE{Fokina202491,
	author = {Fokina, Daria and Grigoriev, Vasiliy V. and Iliev, Oleg and Oseledets, Ivan},
	title = {Machine Learning Algorithms for Parameter Identification for Reactive Flow in Porous Media},
	year = {2024},
	journal = {Lecture Notes in Computer Science},
	volume = {13952 LNCS},
	pages = {91 – 98},
	doi = {10.1007/978-3-031-56208-2_8},
}

@ARTICLE{Gavrilenko2022378,
	author = {Gavrilenko, Pavel and Haasdonk, Bernard and Iliev, Oleg and Ohlberger, Mario and Schindler, Felix and Toktaliev, Pavel and Wenzel, Tizian and Youssef, Maha},
	title = {A Full Order, Reduced Order and Machine Learning Model Pipeline for Efficient Prediction of Reactive Flows},
	year = {2022},
	journal = {Lecture Notes in Computer Science},
	volume = {13127 LNCS},
	pages = {378 – 386},
	doi = {10.1007/978-3-030-97549-4_43},
}

@article{Biermann2025,
title = {Enabling micro-kinetics based simulation of industrial packed-bed reactors by physics-enhanced neural networks},
journal = {Chemical Engineering Journal},
volume = {519},
pages = {163598},
year = {2025},
issn = {1385-8947},
doi = {https://doi.org/10.1016/j.cej.2025.163598},
url = {https://www.sciencedirect.com/science/article/pii/S1385894725044328},
author = {Felix Biermann and Riccardo Uglietti and Felix A. Döppel and Tim Kircher and Mauro Bracconi and Matteo Maestri and Martin Votsmeier},
}


\begin{appendices}

    \section{Reduced Basis Framework}

    For this subsection we define $u(\mu_j, t_k) \in \R^{N_h}$ for $(\mu_j)_{j = 1}^{N_s} =: \Pcal \subset \Theta$ and 
    $(t_k)_{k = 1}^{N_t} =: \Tcal \subset [0, T]$ to be some vectors. They will later represent the solutions to a 
    parametric PDE for a 
    determined parameter set and time point in a given discrete subspace of the coefficient solution manifold. 
    Here $\Theta \subset \R^p$ represents and arbitrary, compact parameter space for $p \in \N$, 
    $T > 0$ some arbitrary time point, and $N_s, N_t \in \N$ are the parameter and time sample size, respectively. 
    In context of the aforementioned coefficient solution manifold,
    the inner product on $\R^{N_h}$ is defined by 
    using a symmetric positive definite matrix $\G \in \R^{N_h \times N_h}$ via 
    \begin{equation*}
        \langle u, v \rangle_\G := u^T \G v, \qquad \text{for all } u, v \in \R^{N_h},
    \end{equation*}
    and the corresponding norm $\norm{\cdot} := \langle \cdot, \cdot \rangle_\G$.

    It is assumed,
    that we sample the parameters independently and identically distributed, and the time points equidistantly determined
    from their respective spaces, and that 
    \begin{equation*}
        M := \sup\limits_{(\mu, t) \in \Theta \times [0, T]} \norm{u(\mu, t)} < \infty, \quad  m := \inf\limits_{(\mu, t) \in \Theta \times [0, T]}\norm{u(\mu, t)} > 0,
    \end{equation*}
    for all $\mu \in \Theta$ and $t \in [0, T]$ in an essential
    sense. Let also 
    \begin{equation*}
        N_\text{data} := N_s \cdot N_t,
    \end{equation*}
    be the total sample size.

    Notice, that it is possible to switch to a lower dimensional euclidean norm, if the projector in consideration
    is an orthogonal projector on the span of a $\G$-orthonormal matrix with specific identity. Namely, this is true,
    if the matrix $\V \in \R^{N_h \times N}$ 
    for some $N \leq N_h$ is
    $\G$-orthonormal and can be expressed as $\V = \A \tilde{\V}$, where $\A \in \R^{N_A \times N}$
    for $N \leq N_A \leq N_h$ is also $\G$-orthonormal. This is quantified by the following result.

    \begin{proposition} \label{prop: switch to euclidean norm}
        Let $\V \in \R^{N_h \times N}$ 
        for some $N \leq N_h$ be
        $\G$-orthonormal such that
        \begin{equation*}
            \V = \A \tilde{\V},
        \end{equation*}
        where $\A \in \R^{N_h \times N_A}$
        is also $\G$-orthonormal and $\tilde{\V} \in \R^{N_A \times N}$ for $N \leq N_A \leq N_h$. 
        Then we have
        \begin{equation*}
            \norm{u - \V \V^T \G u}^2 = c_\A + \euknorm{u_\text{pre-red} - \tilde{\V} \tilde{\V}^T u_\text{pre-red}}^2,
        \end{equation*}
        for any 
        $u \in \R^{N_h}$ with $c_\A > 0$ being a constant only depending on $\A$ and defining $u_\text{pre-red} = \A^T \G u$.
    \end{proposition}
    \begin{proof} 
        $\A \A^T \G$ is a $\G$-orthonormal 
        projector into the span of $\A$. Thus, its residual $u - \A \A^T \G u$ must be 
        $\G$-orthogonal to all columns of $\A$, since we can decompose 
        $u = u_\perp + u_\parallel$, where $u_\parallel \in \text{span}(\A)$ and 
        $u_\perp \in \R^{N_h} - \text{span}(\A)$. Then the Pythagoras decomposition gives
        \begin{align*}
            \norm{u - \V \V^T \G u}^2 &= \norm{u - \A \tilde{\V} \tilde{\V}^T \A^T \G u}^2 \\
            &= \norm{u - \A \A^T \G u}^2 + \norm{\A \A^T \G u - \A \tilde{\V} \tilde{\V}^T \A^T \G u}^2 \\
            &= \norm{u - \A \A^T \G u}^2 + \euknorm{\A^T \G u - \tilde{\V} \tilde{\V}^T \A^T \G u}^2.
        \end{align*}
        The last step is justified by showing that for any $u \in \R^{N_A}$
        \begin{equation*}
            \norm{\A u} = \langle \A u, \A u \rangle_\G = (\A u)^T \G \A u = u^T \A^T \G \A u = u^T u = \euknorm{u}.
        \end{equation*}
        Now setting $u_\text{pre-red} := \A^T \G u$ gives us the wanted statement.
    \end{proof}

    Next, we will discuss this Proper Orthogonal Decomposition and related concepts. For instance the 
    Kolmogorov $N$-width, which lets us derive many key concepts and is an important quantity in the field of ROM methods. 
    We base the following on \cite{quarteroni2015rb}, but import the generalizations to arbitrary inner products from
    \cite{franco2024}. 
    \begin{definition}[Generalized Proper Orthogonal Decomposition] \label{def: POD}
        Assume $N_\text{data} < \infty$ and let $\G \in \R^{N_h \times N_h}$ be a mass matrix.
        We define the time trajectory $U_j = [u(\mu_j, t_k)]_{k = 1}^{N_t} \in \R^{N_h \times N_t}$ 
        for a given $\mu_j \in \Pcal$ and consequently $U = [U_j]_{j=1}^{N_s}
        \in \R^{N_h \times N_\text{data}}$ 
        the collective snapshot matrix. Then the \emph{(discrete) correlation matrix} is given by
        \begin{equation} \label{eq: discrete correlation matrix}
            K = \frac{\vert \Theta \times [0, T] \vert}{N_{data}} U \G U^T \in \R^{N_h \times N_h}.
        \end{equation}
        $K$ is symmetric and semi-positive definite and thus has real and 
        positive eigenvalues $\sigma^2 \geq \dots \geq \sigma_r^2$. 
        Let $\psi^\mu_k$ be the corresponding eigenvectors of the correlation matrix $K$. 
        Then the POD basis of dimension $N < N_h$ is given via
        \begin{equation*}
            \xi_k = \frac{1}{\sigma_k} U \psi_k, \quad 1 \leq k \leq N,
        \end{equation*}
        for the $N$ largest eigenvalues and their corresponding eigenvectors. Finally, stacking these eigenvectors gives us
        the \emph{POD matrix}
        \begin{equation*}
            \A := \left[\xi_k\right]_{k=1}^N \in \R^{N_h \times N}.
        \end{equation*}
        This matrix is $\G$-orthonormal by construction.
        Further, if $N_s, N_t = \infty$, we define the \emph{continuous correlation matrix} via
        \begin{equation} \label{eq: continuous correlation matrix}
            K_\infty = \int_{\Theta \times [0, T]} u(\mu, t) \G u^T(\mu, t) d(\mu, t) \in \R^{N_h \times N_h},
        \end{equation}
        and its real eigenvalues by $\sigma_{k, \infty}^2$.
    \end{definition}

    Following this definition, is a known legacy result.
    It shows optimality of the POD 
    and relates the projection error of the Proper Orthogonal Decomposition with its eigenvalues.

    \begin{proposition} \label{prop: eigenvalues and optimal projection under sampling}
        Let $\mathcal{V}_N = \{ W \in \R^{N_h \times N} \, \vert \, W^T \G W = I_N \}$ be the set of all 
        $N$-dimensional $\G$-orthonormal matrices, 
        $N_\text{data} := N_t N_s$ and $\A = [\xi]_{i=1}^N$ be the POD matrix according to Definition
        \ref{def: POD}, then
        \begin{multline*}
            \frac{\euknorm{\Theta \times [0, T]}}{N_\text{data}}
            \sum\limits_{i \in I, k = 1, \dots, N_t} \norm{u(\mu_j, t_k) - \A \A^T \G u(\mu_j, t_k)}^2 \\
            = \min\limits_{W \in \mathcal{V}_N} \frac{\euknorm{\Theta \times [0, T]}}{N_\text{data}}
            \sum\limits_{i \in I, k = 1, \dots, N_t} 
            \norm{u(\mu_j, t_k) - W W^T \G u(\mu_j, t_k)}^2 
            = \sum_{j=N+1}^r \sigma_k^2.
        \end{multline*}
    \end{proposition}
    \begin{proof}
        See \cite[p.125--126]{quarteroni2015rb}.
    \end{proof}

    \begin{remark} \label{remark: continuous correlation existence}
        The equivalent statement can be made for a POD matrix $\A_\infty \in \R^{N_h \times N}$ coming from 
        the continuous correlation matrix $K_\infty \in \R^{N_h \times N_h}$. See \cite{kunischvolkwein2001}.
    \end{remark}

    \begin{proposition} \label{prop: existence and convergence to sigma infty}
        Let $K$ and $K_\infty$ and their respective eigenvalues be defined as in Definition \ref{def: POD}, 
        then letting $N_s, N_t \to \infty$, we get 
        \begin{equation*}
            \sum_{k > N} \sigma_k^2 \longrightarrow \sum_{k > N} \sigma_{k, \infty}^2, \quad \mbox{ a.s.}
        \end{equation*}
    \end{proposition}
    \begin{proof}
        See \cite[Subsection 2.2]{brivio2023}.
    \end{proof}

    \section{Approximation Theory}

    Define the set of functions bounded in Sobolev norm by $B>0$ as
    \begin{equation*}
        \Fcal_{n,d,p,B} := \big\{ f \in W^{n,p}((0,1)^d) \, \big\vert \, \snorm{f}[n][p][d] \leq B \big\}.
    \end{equation*}
    Then, the following approximation results for neural networks hold.

    \begin{theorem}[Yarotsky Theorem, cf. \cite{yarotsky2017nn}] \label{theo: yarotsky}
        Let $d \in \N, n \in \N, B = 1$ and $f \in \Fcal_{n ,d , \infty, B}$ some arbitrary function.
        Then there is a constant $c = c(d, n) > 0$ with the following properties:
        
        For any $\varepsilon \in (0, 1)$, there is a neural network architecture 
        $\Acal_\varepsilon = \Acal_\varepsilon(d, n, \varepsilon)$ with $d$-dimensional input and one-dimensional output such
        that, there is a neural network $\Phi_\varepsilon^f$ that has architecture $\Acal_\varepsilon$ and satisfies
        \begin{equation*}
            \snorm{f - R_\rho(\Phi_\varepsilon^f)}[0][\infty][d] \leq \varepsilon,
        \end{equation*}
        with
            $L_{\Acal_\varepsilon} \leq c \cdot \ln(1 / \varepsilon)$,
            $\omega_{\Acal_\varepsilon} \leq c \cdot \varepsilon^{-d/n} \ln(1/\varepsilon)$,
            $N_{\Acal_\varepsilon} \leq c \cdot \varepsilon^{-d/n} \ln(1/\varepsilon)$.
    \end{theorem}

    \begin{theorem}[Gühring Theorem, cf. \cite{gühring2019nn}] \label{theo: gühring}
        Let $d \in \N, n \in \N_{\geq 2}, B > 0, 1 \leq p \leq \infty$ 
        and $0 \leq s \leq 1$ and $f \in \Fcal_{n ,d , p, B}$ some arbitrary function.
        Then there is a constant $c = c(d, n, p, B, s) > 0$ with the following properties:
        
        For any $\varepsilon \in (0, 1)$, there is a neural network architecture 
        $\Acal_\varepsilon = \Acal_\varepsilon(d, n, p, B, s, \varepsilon)$ with $d$-dimensional input and one-dimensional output such
        that, there is a neural network $\Phi_\varepsilon^f$ that has architecture $\Acal_\varepsilon$ and satisfies
        \begin{equation*}
            \snorm{f - R_\rho(\Phi_\varepsilon^f)}[s][p][d] \leq \varepsilon,
        \end{equation*}
        with
            $L_{\Acal_\varepsilon} \leq c \cdot \ln(1 / \varepsilon)$,
            $\omega_{\Acal_\varepsilon} \leq c \cdot \varepsilon^{-d/(n-s)} \ln(1/\varepsilon)$,
            $N_{\Acal_\varepsilon} \leq c \cdot \varepsilon^{-d/(n-s)} \ln(1/\varepsilon)$.
    \end{theorem}

    \section{Proofs for Main Matter} \label{appendix: proofs for main matter}

    \begin{proof}[Proof of Lemma \ref{lemma: DOD functional is lipschitz}.]
        Let $\mu, \mu' \in \Theta$, $t, t' \in [0, T]$ 
        and $\V, \V' \in [-1, 1]^{N_h \times N'}$ arbitrary.
        For the following calculation, we compose the preliminary identity for 
        $u := \uht$ and $u' := u_h^{\mu', \nu, t'}$,
        and the projectors $P := \V \V^T \G$, $P' := \V' (\V')^T \G$,
        \begin{equation} \label{eq: optimal dod help equation}
            \begin{aligned}
                (\I - P)u - (\I - P')u' &= u - Pu - u' + P'u' \\
                &= u - Pu - u' + P'u' + Pu' - Pu' \\
                &= u - u' - Pu + Pu' + P'u' - Pu' \\
                &= (\I - P)(u - u') + (P' - P)u'.
            \end{aligned}
        \end{equation}
        Then we can calculate with our auxiliary 
        equation (\ref{eq: optimal dod help equation}),
        \begin{align*}
            \euknorm{J(\mu, t, \V) - J(\mu', t', \V')}
            &= \int_{\Xcal} 
            \Big| \, \norm{u - P u}^2 - \norm{u' - P' u'}^2 \, \Big| \, d(\nu)
            \\
            &= \int_{\Xcal}
            \Big| \, \big\langle (u - Pu) + (u' - P'u'),\; (u - Pu) - (u' - P'u') \big\rangle \, \Big| \, d(\nu)
            \\
            &\leq \int_{\Xcal}
            \Big( \norm{u - Pu} + \norm{u' - P'u'} \Big)\;
            \norm{ (u - Pu) - (u' - P'u') } \, d(\nu)
            \\
            &= \int_{\Xcal}
            \Big( \norm{u - Pu} + \norm{u' - P'u'} \Big)\;
            \norm{ (\I - P)(u - u') + (P' - P)u' } \, d(\nu)
            \\
            &= \int_{\Xcal}
            \Big( \norm{u - Pu} + \norm{u' - P'u'} \Big)\;
            \Big( \norm{(\I - P)(u - u')} + \norm{(P' - P)u'} \Big) \, d(\nu)
            \\
            &\leq \int_{\Xcal}
            \Big( \norm{u} + \norm{u'} \Big)\;
            \norm{u - u'} \, d(\nu)
            + \int_{\Xcal}
            \Big( \norm{u} + \norm{u'} \Big)\;
            \opnorm{P' - P} \; \norm{u'} \, d(\nu)
            \\
            &\leq |\Xcal| \, (2M)\, L \big(\euknorm{\mu-\mu'} + \euknorm{t-t'}\big)
            \;+\; |\Xcal| \, (2M^2)\, \opnorm{P' - P}
            \\
            &\leq |\Xcal| \Big( 2ML \big(\euknorm{\mu-\mu'} + \euknorm{t-t'}\big) + 2 M^2 \, \opnorm{\V - \V'} \Big) \\
            &\leq |\Xcal| \, (L + 2M)\, \big(\euknorm{\mu-\mu'} + \euknorm{t-t'} + \opnorm{\V - \V'}\big).
        \end{align*}
        The key justification here, is the bound on $\opnorm{P - P'}$, that can be achieved by considering 
        \begin{align*}
            \opnorm{P - P'} &= \opnorm{(\V' - \V) (\V')^T \G - \V ((\V')^T - \V^T) \G} \\
            &\leq \opnorm{\V' - \V} \cdot \euknorm{(\V')^T \G}_\text{op} + \opnorm{\V} \cdot \euknorm{((\V')^T - \V^T) \G}_\text{op} \\
            &\leq 2 \opnorm{\V - \V'}.
        \end{align*}
        Here, we denote with $\euknorm{\cdot}_\text{op}$ the standard (euclidean) spectral norm. Hence,we have $\opnorm{\V} = \euknorm{\V^T \G}_\text{op} = 1$ for
        any $\G$-orthonormal matrix $\V \in \R^{N_h \times N'}$.
        
        To prove existence of the optimal DOD $s: \Theta \times [0, T]
        \to [-1, 1]^{N_A \times N'}$, we remark, that the space $[-1, 1]^{N_A \times N'}$ is a compact one,
        and the objective DOD functional $J$ is continuous, which allows us to use 
        the Kuratowski-Ryll-Nardzewski Theorem, found in \cite{kuratowski1965}. The specifications of
        how this theorem can be applied to this specific situation can be found in \cite[Appendix B.]{franco2024},
        whereas we will refrain from an extensive discussion, since we assume a stronger quality later on anyway.
    \end{proof}

    \begin{proof}[Proof of Proposition \ref{prop: DOD Existence}.]
        Let $N', N_A, N_h \in \N$ be fixed by Assumption \ref{assumption: dimension hierarchy}. 
        We define $J$, the objective DOD functional for $\Theta'$ and $s$, the
        optimal DOD for $\Theta'$, according to Definition \ref{def: DOD functional}. Then the optimal DOD
        \begin{equation*}
            s: (\mu, t) \mapsto \underset{{\tilde{\V} \in [-1, 1]^{N_A \times N'}}}{\text{argmin}}  J(\mu, t, \A \tilde{\V}),
        \end{equation*}
        is bounded, by continuity of J and compactness of its image. We know that for $\V = \A \tilde{\V}$ being $\G$-orthonormal, its operator
        norm is $1$, and hence, after rescaling, each entry can be inferred to be w.l.o.g. in $[-1, 1]$. This is
        due to the norm equivalence on this finite dimensional matrix space. 
        
        Continuing the reasoning from before, we have 
        $s \in L^2(\Theta \times [0, T]; \R^{N_A \times N'})$, i.e.
        \begin{equation*}
            \norm{s}^2_{L^2(\Theta \times [0, T]; \R^{N_A \times N'})} := 
            \int_{\Theta \times [0, T]} \euknorm{s(\mu, t)}^2 \, d(\mu, t)  < \infty.
        \end{equation*}
        Define $[s]_{ij}$ as the canonically lifted
        entry in position $(i, j) \in \{1, \dots, N_A\} \times \{1, \dots, N'\}$.

        This means that by Hornik Theorem (see \cite{hornik1991nn}), there is a neural network architecture 
        $[\tilde{\V}_0]_{ij}$ with $(p+1)$-dimensional input and one-dimensional output, 
        and a choice of neural network weights $\theta^*_{ij} \in \Theta([\tilde{\V}_0]_{ij})$ 
        for that architecture, satisfying
        \begin{equation*}
            \norm{[s]_{ij} - R_\rho([\tilde{\V}_0]_{ij}(\theta^*_{ij}))}^2_{L^2(\Theta \times [0, T]; \R)} \leq \epsilon,
        \end{equation*}
        for each $(i, j) \in \{1, \dots, N_A\} \times \{1, \dots, N'\}$.
        We then write $[\tilde{\V}_0]_{ij} := [\tilde{\V}_0]_{ij}(\theta^*_{ij})$ for each of
        these networks. Thus, by using network parallelization, we can set the non-restrained inner DOD module as
        \begin{equation*}
            \tilde{\V}_0 := P\left(P([\tilde{\V}_0]_{1 1}, \dots, [\tilde{\V}_0]_{1 N'}), \dots, 
            P([\tilde{\V}_0]_{N_A  1}, \dots,
            [\tilde{\V}_0]_{N_A N'})\right),
        \end{equation*}
        with a joint set of weights $\theta^*_{\Sigma ij} \in \Theta(\tilde{\V}_0)$. These can technically be 
        computed using the canonical lift and the definition of parallelized networks. 
        The realization of this network then attains
        \begin{equation*}
            \norm{s - R_\rho(\tilde{\V}_0)}^2_{L^2(\Theta \times [0, T]; \R^{N_A \times N'})} \leq
            \epsilon.
        \end{equation*}
        We need to ensure, that the provided realization of the non-restrained inner DOD module is indeed in the pre-image of $J$.
        For this, we show that there is a network, which indeed attains the accuracy, but is also entrywise $[-1,1]$.
        Let $\zeta = ((A_1, b_1), (A_2, b_2), (A_3, b_3))$ be given via
        \begin{align*}
            A_1 := \begin{bmatrix}
                \I_{\R^{N_A \times N'}} 
            \end{bmatrix},
            \quad
            b_1 := \begin{pmatrix}
                \I_{\R^{N_A \times N'}}
            \end{pmatrix},&
            \quad
            A_2 := \begin{bmatrix}
                -\I_{\R^{N_A \times N'}}
            \end{bmatrix},
            \quad
            b_2 := \begin{pmatrix}
                \I_{\R^{N_A \times N'}}
            \end{pmatrix}, \\
            A_3 := \begin{bmatrix}
                -\I_{\R^{N_A \times N'}}
            \end{bmatrix},&
            \quad
            b_3 := \begin{pmatrix}
                \I_{\R^{N_A \times N'}}
            \end{pmatrix}
        \end{align*}
        The realization of this network is equivalent to the entry-wise operator of $x \mapsto 1 - \rho(2 - \rho(\rho(x + 1))) 
        \in [-1, 1]$
        for $x \in \R$ on all entries of 
        a $(N_A \times N')$-dimensional matrix. Note, that to make this neural network truly adhere to the definition given in
        \textcite{gühring2019nn}, which needs the
        dimensions to be natural numbers, we would have to reinterpret any matrix as a canonically stacked vector. 
        We omit this for the sake of readability. This does not affect the number of layers or weights.

        Notice, that the Lipschitz constant of $\zeta$ as a realized map is $1$.
        Further, we define
        \begin{equation*}
            \tilde{\V}_{\mu, t} := \left(R_\rho(\zeta) \circ R_\rho(\tilde{\V}_0(\theta^*_{\Sigma ij})) \right)
            (\mu, t), \quad \text{for } (\mu, t) \in \Theta \times [0, T],
        \end{equation*}
        a new network, called the inner DOD module, with a final
        architecture $\Phi_{\tilde{\V}}$ and weights $\theta^*_\text{DOD} 
        \in \Theta(\Phi_{\tilde{\V}})$. Note, that these weights can be computed using the canonical lift, and 
        the definition of network concatenations (see \cite{gühring2019nn}). 
        This gives us the complete definition of the DOD matrix via
        \begin{equation*}
            \V_{\mu, t} := \A \cdot \tilde{\V}_{\mu, t} \qquad \text{for } (\mu, t) \in \Theta \times [0, T].
        \end{equation*}
        
        The complexity of the inner DOD module only differs from $\tilde{\V}_0$ by constants, but it also realizes
        \begin{equation*}
            \int_{\Theta \times [0, T]} \norm{\V_{\mu, t} - \A s(\mu, t)}^2 d(\mu, t)
            \leq \int_{\Theta \times [0, T]} \euknorm{R_\rho(\tilde{\V}_0)(\mu, t) - s(\mu, t)}^2
            d(\mu, t)
            \leq \epsilon,
        \end{equation*}
        by the $1$-Lipschitz property of $R_\rho(\zeta)$ and the argument of $\norm{\A x} = \euknorm{x}$
        for any $x \in \R^{N_h}$.
        Additionally, as a matrix, the DOD is indeed in the pre-image of $J$. We further define, with
        the Proposition \ref{prop: switch to euclidean norm} for the switch to the euclidean norm,
        the constant
        \begin{equation*}
            c_\A' := \int_{\Theta \times [0, T]} \bigg( \sup\limits_{\nu \in \Theta'} 
            \underbrace{\norm{\uht - \A \A^T \G \uht}^2}_{= c_\A} \bigg) d(\mu,t).
        \end{equation*} 
        Then we can calculate with $\uhtprered := \A^T \G \uht$
        \begin{multline} \label{eq: last step for optimal DOD}
            \int_{\Theta \times \Theta' \times [0, T]} 
            \norm{\uht - \V_{\mu, t} \V_{\mu, t}^T \G \uht}^2 d(\mu, \nu, t) \\ 
            \leq \int_{\Theta \times [0, T]}
            \left\vert J(\mu, t, \V_{\mu, t}) - J(\mu, t, \A s(\mu, t)) \right\vert d(\mu, t) +
            \int_{\Theta \times [0, T]} J(\mu, t, \A s(\mu, t)) d(\mu, t) \\
            \leq \epsilon + \int_{\Theta \times [0, T]} \left( \min\limits_{\V \in \R^{N_A \times N'}} J(\mu, t, \A \V)\right) 
            d(\mu, t) \\
            \leq \epsilon + c_\A' + \int_{\Theta \times [0, T]} \left( \min\limits_{\tilde{\V} \in \R^{N_A \times N'}} 
            \sup\limits_{\nu \in \Theta'} \euknorm{\uhtprered - \tilde{\V} \tilde{\V}^T \uhtprered}^2 \right) d(\mu, t).
        \end{multline}
    \end{proof}

    \begin{proof}[Proof of Theorem \ref{theo: PAC bound for DOD+DFNN}.]
        This proof is in its essence highly related to the one of Theorem \ref{theo: PAC bound for DOD-DL-ROM}.
        As such, we will only briefly explain the necessary steps.
        We leverage Theorem \ref{theo: relative error DOD+DNN} and the preliminaries to attain the bounds
        \begin{equation*}
            \Ecal_\text{DOD} < \frac{\varepsilon}{2} \quad \text{and} \quad \Ecal_\text{S} < \frac{\varepsilon}{4}
        \end{equation*}
        in the exact same way as has been done in the proof of Theorem \ref{theo: PAC bound for DOD-DL-ROM}. 
        For the bound on the neural network error, we use the Lipschitz property of $\Gcal$ under 
        Assumption \ref{assumption parameter-to-solution map}. 
        Indeed, let 
        \begin{equation*}
            \Scal_{N'} :=  \{\Phi_*(\mu, \nu, t) := \V_{\mu, t}^T \G \uht \in \R^N \, \vert \, (\mu, \nu, t) 
            \in \Theta \times \Theta' \times [0, T] \},
        \end{equation*}
        for the architecture and choice of weights of the DOD 
        from the proof of Theorem \ref{theo: PAC bound for DOD-DL-ROM}. 
        We again assume w.l.o.g. a normalization
        \begin{equation*}
            \Theta \times \Theta' \times [0, T] \subset (0, 1)^{p+q+1}, \quad \Scal_{N'} \subset (0, 1)^{N'}.
        \end{equation*}
        The optimal coefficient map
        $\Phi_*: (0, 1)^{p+q+1} \to (0, 1)^{N'}$ 
        then inherits the Lipschitz continuity, i.e. for any pair $(\mu, \nu, t), (\mu', \nu', t') 
        \in \Theta \times \Theta' \times [0, T]$, we have
        \begin{align*}
            \euknorm{\Phi_*(\mu, \nu, t) - \Phi_*(\mu', \nu', t')} &= \euknorm{\V_{\mu, t}^T \G (\uht - u_h^{\mu', \nu', t'})}
            \\
            &\leq \euknorm{\uht - u_h^{\mu', \nu', t'}} \\
            &\leq L \cdot (\euknorm{\mu - \mu'} + \euknorm{\nu - \nu'} + 
            \euknorm{t - t'}).
        \end{align*}
        Here, we have exploited
	that $\V_{\mu, t}$ is $\G$-orthonormal
        for any parameter choice.
        From the Lipschitz continuity we infer
        $\Phi_* \in \Fcal_{1, p+q+1, \infty, 1 + L}$. With that argument, we can 
        use Yarotsky Theorem \ref{theo: yarotsky} to get the complexities entry-wise, as seen before,
        with error margin
        \begin{equation*}
            \Ecal_\text{NN} < \frac{\varepsilon}{4}.
        \end{equation*}
        Combining all these, eventually gets us the wanted bounds.
    \end{proof}

\end{appendices}

\end{document}